\documentclass[11pt,a4paper,reqno]{amsart}
\usepackage[T1]{fontenc}
\usepackage{amsmath}
\usepackage{phaistos}
\usepackage{protosem}
\usepackage{amsthm}
\usepackage{amsfonts}
\usepackage{amssymb}
\usepackage{graphicx}
\usepackage{color}
\usepackage[colorlinks=true,linkcolor=black,citecolor=blue]{hyperref}
\usepackage{amsbsy}
\usepackage{mathrsfs}
\usepackage{bbm}
\usepackage{todonotes}

\newtheorem{theorem}{Theorem}[section]
\newtheorem{lemma}[theorem]{Lemma}

\newtheorem{proposition}[theorem]{Proposition}
\newtheorem{corollary}[theorem]{Corollary}

\theoremstyle{remark}
\newtheorem{remark}[theorem]{\it \bf{Remark}\/}

\numberwithin{equation}{section}
\catcode`@=11
\def\section{\@startsection{section}{1}%
  \z@{1.5\linespacing\@plus\linespacing}{.5\linespacing}%
  {\normalfont\bfseries\large\centering}}
\catcode`@=12
\newcommand{\be}{\begin{equation}}
\newcommand{\ee}{\end{equation}}
\newcommand{\bea}{\begin{eqnarray}}
\newcommand{\eea}{\end{eqnarray}}
\newcommand{\bee}{\begin{eqnarray*}}
\newcommand{\eee}{\end{eqnarray*}}
\newcommand{\norm}[1]{{\left\vert\kern-0.25ex\left\vert\kern-0.25ex\left\vert #1 
    \right\vert\kern-0.25ex\right\vert\kern-0.25ex\right\vert}}

\def\pa{\partial}

\def\RR{\mathbb{R}}

\def\TT{\mathcal{T}}

\def\wt{\tilde{w}}
\def\sigmat{\tilde{\sigma}}

\def\de{\delta}

\def\uh{\hat{u}}

\catcode`@=11
\def\supess{\mathop{\operator@font Sup\,ess}}
\catcode`@=12

\def\RR{\mathbb{R}}

\def\e{\varepsilon}

\def\bar#1{{\overline #1}}

\def\Wt{\tilde{W}}
\def\Sigmat{\tilde{\Sigma}}

\def\R2+{\RR ^2_+}

\def\dt{\tilde{d}}
\def\et{\tilde{e}}

\def\pa{\partial}

\def\lim{\mathop{\rm lim}}
\def\T{\mathcal T}

\def\sup{\mathop{\rm sup}}

\def\l{\lambda}

\def\th{{\rm th}}
\def\log{{\rm log}}

\def\rhoh{\hat{\rho}}

\def\T{\Theta}

\def\Phit{\widetilde{\Phi}}

\def\th{\tilde{H}}
\def\pa{\partial}

\def\psit{\tilde{\psi}}

\def\Psit{\tilde{\Psi}}

\def\pa{\partial}

\def\NL{\textrm{NL}}

\def\rhoh{\hat{\rho}}

\def\th{\tilde{h}}

\def\Et{\tilde{E}}

\def\mt{\tilde{m}}

\def\Dt{\tilde{D}}
\def\T{\mathcal T}

\def\NLt{\widetilde{\NL}}

\def\xit{\tilde{\xi}}
\def\th{\tilde{h}}

\def\Thetam{\Theta_{\rm main}}
\def\uh{\hat{u}}

\DeclareMathOperator{\peye}{\textproto{o}}

\def\Wte{\Wt_{\hskip -.1pc\peye}}
\def\Sigmate{\Sigmat_{\hskip -.1pc\peye}}
\def\wte{\wt_{\hskip -.1pc\peye}}
\def\sigmate{\sigmat_{\hskip -.1pc\peye}}
\def\psite{\psit_{\hskip -.1pc\peye}}

\title{On smooth holosphere solutions to compressible Euler equation}

\author{Jihoi Kim}
\address[Jihoi Kim]{University of Cambridge, United Kingdom.}
\email{rk614@cam.ac.uk}

\begin{document}
\maketitle

\begin{abstract} 
We consider the barotropic Euler equations in dimension $d\ge 2$ with decaying density at spatial infinity. The phase portrait of the nonlinear ODE governing the equation for spherically symmetric self-similar solutions has been introduced in  the pioneering work of Guderley \cite{guderley}. The existence of smooth decaying solutions through the sonic line has been obtained in \cite{MRRSprofile} in a suitable range of parameters. We work in this paper in a {\em different range} of parameters and show the existence of holosphere like solutions which are smooth through the sonic line, but display vanishing density at a sphere.
\end{abstract}

\section{Introduction}
\subsection{Setting of the problem}
In this paper, we consider the isentropic compressible Euler equations in dimension $d\ge 2$, $y\in \RR^d$,
\be
\label{eq: Compressible Euler}
\left|\begin{array}{l}\pa_t\rho+\nabla\cdot(\rho u)=0\\
\rho\pa_tu+\rho u\cdot\nabla u+\nabla p=0\\
p=\frac{\gamma-1}{\gamma}\rho^\gamma\\
\rho(t,y)>0.
\end{array}\right.
\ee 
We aim to construct a family of radially symmetric self-similar profiles and are specifically interested in solutions which decay at infinity
\be
\label{vneivenoenenevnove}
 \lim_{|y|\to +\infty}(\rho(t,y),u(t,y))=0.
\ee 
The existence of such solutions is a classical problem and the first step towards the description of singularity formation.

\subsection{Self-similar equation}
We introduce the self-similar renormalization 
\be
\label{eq: Self-similar Renormalization}
\left|\begin{array}{ll}
\rho(t,y)=\left(\frac{\l}{\nu}\right)^{\ell}\rhoh(\tau,Z)\\
u(t,y)=\ \frac{\lambda}{\nu}\uh(\tau,Z)\\
Z=\frac{y}{\l}, \ \ \frac{d\tau}{dt}=\frac{1}{\nu}\\
-\frac{\l_\tau}{\l}=1, \ \ -\frac{\nu_\tau}{\nu}=r
\end{array}\right.
\quad \text{for}\quad \ell=\frac{2}{\gamma-1},\quad r>1
\ee
maps \eqref{eq: Compressible Euler} on $[0,T)$ onto the global in time $\tau$ renormalized flow
\be
\label{eq: Renormalized Flow}
\left|\begin{array}{l}
\pa_\tau \rhoh+\ell(r-1)\rhoh+\Lambda\rhoh+\nabla \cdot(\rhoh\uh)=0\\
\pa_\tau \uh+(r-1)\uh+\Lambda \uh+\uh\cdot\nabla \uh+\nabla (\rhoh^{\gamma-1})=0\\
\Lambda =Z\cdot \nabla
\end{array}\right.       
\ee
A self-similar profile is a stationary solution to \eqref{eq: Renormalized Flow}:
\be
\label{eq: Self-similar Equation}
\left|\begin{array}{l}
\ell(r-1)\rhoh+\Lambda\rhoh+\nabla \cdot(\rhoh \uh)=0\\
(r-1)u+\Lambda \uh+\uh \cdot\nabla \uh+\nabla (\rhoh^{\gamma-1})=0\\
\end{array}\right.
\ee
which produces a blow up solution for \eqref{eq: Compressible Euler} with the rate of concentration 
$$
\l(t)=\l_0(T-t)^{\frac1r}, \ \ \nu(t)=r(T-t).
$$

\subsection{The Guderley phase portrait} 

In the pioneering work \cite{guderley,sedov}, all solutions to \eqref{eq: Self-similar Equation} with spherical symmetry are mapped through the Emden transform
\be
\label{eq: Emden Transform}
\left|\begin{array}{l}(\rhoh(Z))^{\frac{\gamma-1}{2}}=\sqrt{\frac \ell 2}Z\sigma(x)\\
\uh(Z)=- Zw(x)\\
Z=e^x
\end{array}\right.
\ee
onto the autonomous system of nonlinear ODE's:
\be
\label{eq: Autonomous System}
 \left|\begin{array}{l}
\Delta w'=-\Delta_1\\
\Delta \sigma'=-\Delta_2
\end{array}\right.
\ee
with the explicit non linearities
\be
\label{eq: Delta}
\left|\begin{array}{l}
\Delta=(w-1)^2-\sigma^2\\
\Delta_1=w(w-1)(w-r)-d\left(w-\frac{\ell(r-1)}{d}\right)\sigma^2\\
\Delta_2=\frac{\sigma}{\ell}\left[(\ell+d-1)w^2-w(\ell+d+\ell r-r)+\ell r-\ell \sigma^2\right].
\end{array}\right.
\ee
Here, 
$$\ell=\frac{2}{\gamma-1},\quad r>0$$
are {\em free parameter} and the blow up speed respectively. Essential canonical features of the phase portrait are the following:\\

\noindent\underline{Sonic lines} $w-1=\pm \sigma$. This is exactly the set $\Delta=0$ where the ODE degenerates.\\

\noindent\underline{$P_5$ point}. The point 
$$
P_5=\left(\sigma(P_5)=\frac{r\sqrt{d}}{d+\ell},\ w(P_5)=\frac{\ell r}{d+\ell}\right)
$$
is an endpoint of the dynamical system \eqref{eq: Autonomous System}, i.e.
$$
\left|\begin{array}{l}
\Delta_1(P_5)=\Delta_2(P_5)=0\\
\Delta(P_5)\neq 0
\end{array}\right.
$$
Integral curves end there.\\

\noindent\underline{Points $P_1, P_2, P_3$}. Trajectories can only cross the sonic line at the triple points, where $\Delta=\Delta_1=\Delta_2=0,$ which  are $(0,0), (r,0), P_1=(1,0)$ and two other points on the sonic line $w+\sigma=1$ which we refer to as $P_2,P_3$ and which exist thanks to the constraint \eqref{eq: Parameters}.\\

\noindent\underline{Point $P_6$}. It is the point at infinity $(+\infty,\frac{\ell(r-1)}{d})$. The unique separatrix curve coming out of $P_6$ is the unique (up to scaling) smooth radially symmetric solution at the origin.\\

\noindent\underline{$P_4$ point}. The point $P_4=(0,0)$ attracts solutions which vanish near $x\to \infty$.\\

\noindent\underline{Critical values}. Let \be
\label{eq: Critical r}
\left|\begin{array}{l}
r^*(d,\ell)=\frac{d+\ell}{\ell+\sqrt{d}},\\
r_+(d,\ell)=1+\frac{d-1}{(1+\sqrt{\ell})^2},
\end{array}\right.
\ee
then the relative position of $r$ with respect to these two critical values governs the relative position of the points $P_2,P_3,P_5$.

\begin{figure}
\centering
\includegraphics[width=13cm]{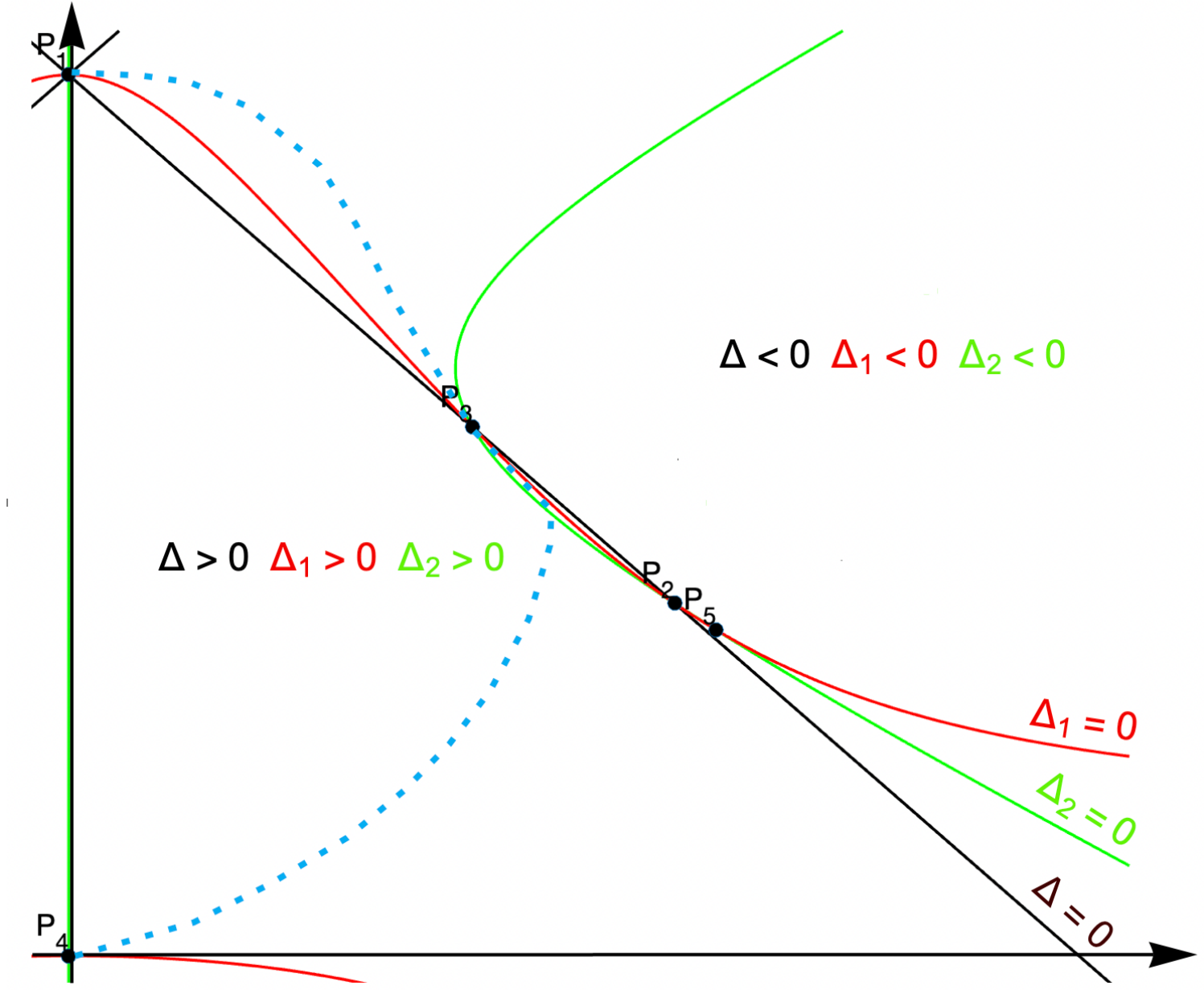}
\caption{Phase portrait in the range $r^*(d,\ell)<r<r_+(d,\ell)$. Dashed curve is the trajectory of the solution constructed in Theorem \ref{thmmain}.}
\label{fig: Phase Portrait}
\end{figure}

\subsection{Previous results} The study of globally smooth and decaying solutions at $+\infty$ becomes a connection problem bewteen the critical points $P_6,P_4$. In the range of parameters studied in \cite{MRRSprofile},
$$
d\ge 2,\quad \left|\begin{array}{ll}1<r<r^*(d,\ell) &\text{for } 0<\ell<d,\\
r^*(d,\ell)<r<r_+(d,\ell) &\text{for }  d<\ell.
\end{array}\right.
$$
 $P_5$ is at the left of $P_2$ which ensures that the unique separatrix coming out of $P_4$ will meet the sonic line at $P_2$, and may then continue to $P_4$. However the solution will canonically lose regularity when passing through the sonic line at $P_2$. The heart of the analysis in \cite{MRRSprofile} is to show that for a {\em discrete} choice of blow up speed $r$ near the critical value $r_*$, a connection can be constructed which coincides with the unique solution which passes through the sonic line at $P_2$ in a $\mathcal C^\infty$ way. This improved regularity is the heart of the proof of the linear, and then non linear dynamical stability of these profiles leading to the construction of blow up solutions for both the Navier-Stokes and Euler compressible problems, \cite{MRRSfluid}.\\
 
\subsection{New range of parameters and holosphere trajectory}  We are interested in this paper on the new range of parameters 
\be
\label{eq: Parameters}
d\ge 2, \ \ d-1<\ell<d, \ \ r^*(d,\ell)<r<r_+(d,\ell),
\ee
which forces $P_5$ at the right of $P_2$. Hence the smooth solution at the origin emerging from $P_4$ terminates in $P_5$ which means growth in space, and this dooms the plan for the construction of blow up solutions proposed in \cite{MRRSfluid}.\\ 
We thus focus our attention onto {\em another} trajectory. A classical analysis of the phase portrait ensures the uniqueness of a trajectory coming out of $P_1$. Under the additional restriction 
\be
\label{vneoneonoevnoeiv}
d-1<\ell<d,\quad 0<r_+(d,\ell)-r\ll1,
\ee
we can show that this trajectory will cross the sonic line at $P_3$ and may then reach $P_4$. The solution decays self similarly at infinity as it approaches $P_4$, but at $P_1$ the density vanishes at some finite radius: this is vaccuum in the data. Such kind of $P_1,P_3,P_4$ connection is called a {\em holosphere} trajectory. 

\subsection{Statement of the result} Our main claim in this paper is that we can adapt the regularity analysis of \cite{MRRSprofile} to prove for a large {\em discrete} set of parameters the existence of a holosphere trajectory which is $\mathcal C^\infty$ through the sonic line. The precise statement is the following.

\begin{theorem}[Existence of $C^\infty$ asymptotically vanishing self-similar profiles] 
\label{thmmain}
Let $d\ge 2$. Let the critical speed $r_+(d,\ell)$ be given by \eqref{eq: Critical r}. Then for all $d\ge2$, there exist a function 
$$
S_\infty(d,\ell):\Bbb N^*\backslash\{1\}\times \mathcal O_d\to \Bbb R
$$ such that for any $\ell\in \mathcal O_d$ obeying the condition
\be
\label{conditionell}
S_\infty(d,\ell)\neq 0,
\ee
there exists a discrete sequence $$
\left|\begin{array}{l}
r^*(d,\ell)<r_n<r_+(d,\ell), \ \ |r_n-r_+(d,\ell)|\ll1\\
\lim_{n\to +\infty} r_n=r_+(d,\ell)
\end{array}\right.
$$ such that the unique holosphere trajectory coming out of $P_1$ terminates in $P_4$ and passes through $P_3$ in a $\mathcal C^\infty$ way.
\end{theorem}

The function $S_\infty(d,\ell)$ appears in the asymptotic analysis of the flow near $P_3$. It can be explicitly expressed as a normally convergent series 
$$S_\infty(d, \ell)=\sum_{n=0}^{+\infty} u_n(d,\ell), \ \ |u_n(d,\ell)|\leq\frac{c_{d,\ell}}{1+n^2}$$ 
where the series $u_n(d,\ell)$ satisfies an explicit though complicated non linear induction relation. The proof of convergence of the series yields the analyticity of the mapping $\ell\mapsto S_\infty(d,\ell)$ in a suitable open set of the complex plane.

\begin{lemma}[Holomorphic extension, \cite{MRRSprofile}]
\label{lemmaisolated} 
The function $\ell\mapsto S_\infty(d,\ell)$ extends holomorphically to an open neighbourhood $\Omega_d$ of $(0,d)$ in the complex plane.
\end{lemma}

We do not know how to compute analytically the zeroes of $S_\infty(d,\ell)$. However, for $d=2,3$, Lemma \ref{lemmaisolated} ensures that, unless the function vanishes identically, the possible zeroes are isolated with possible accumulation points $(0,d,+\infty)$ only\footnote{Note that all the conclusions of Lemma \ref{lemmaisolated} can be extended to higher dimension $d\ge 4$.}.\\

Moreover, since $u_n(d,\ell)$ is given by an explicit induction relation on the coefficients, we can perform an elementary numerical computation of the series. We give the results in dimension $d=2$ and $d=3$ which will be used for the study of the compressible Euler and Navier-Stokes equations. The assertion in Lemma \ref{lemmaisolated} allows us to check the non-vanishing condition for small values 
of $\ell$ only.\\

\noindent{\bf Numerical claim} [Numerical study of the zeroes of $S_\infty(d,\ell)$, case $d=2,3$]
{\em In the case $d=2$ and $d=3$, we have  
\be
\label{estiamtesfisonbis}
\left|\begin{array}{l}
S_\infty(2,\ell)>0\  \ \mbox{for}\ \ \ell=0.1,\\
S_\infty(3,\ell)>0\  \ \mbox{for}\ \ \ell=0.1.
\end{array}\right.
\ee}

A major open problem after this work is to understand whether the holosphere can be achieved as the leading order blow up profile of a solution emerging from a smooth well localized data, either for the Euler or Navier-Stokes problem. This problem is deeply linked to vaccuum formation in compressible fluid dynamics and is known to be delicate. The existence of $\mathcal C^\infty$ at $P_3$ connections is a first {\em essential} step for the dynamical understanding of these solutions.\\

\subsection*{Aknowledgments} The author would like to thank his advisor P. Rapha\"el for his guidance and advice during the preparation of this work. It is supported by the ERC/UKRI advanced grant SWAT and Cambridge Commonwealth Trust.

\section{Geometry of the phase portrait}

We study in this section the geometry of the phase portrait in the range of parameters \eqref{eq: Parameters}.\\

\subsection{Roots of $\Delta,\Delta_1,\Delta_2$} $\Delta$ has been normalized to vanish on the sonic lines
$$
\{\Delta=0\}=\{w=1+\sigma\}\cup\{w=1-\sigma\}
$$
which are independent of the parameters. We study the roots of $\Delta_2$ and $\Delta_1$. The proof of the results in this section is purely algebraic and we will often recall results from \cite{MRRSprofile}. We introduce a constant
$$
w_e=\frac{\ell(r-1)}{d}
$$
which will appear frequently throughout this section.

\begin{lemma}[Roots of $\Delta_2$] 
\label{lem: Delta_2 Roots}
Assume \eqref{eq: Parameters}. There exists $\sigma_2^{(0)}(d,\ell)\in[0,\infty)$ such that the roots of $\Delta_2$ in the range $\sigma \ge 0$ are given by
\be
\label{eq: Delta_2 Roots}
\left|\begin{array}{l}
w^{\pm}_2(\sigma)=\frac{1}{2(\ell+d-1)}\left[2\ell+d-1-\frac{dw_e(1-\ell)}{\ell}\pm\sqrt{I(\sigma)}\right]\\
\sigma\ge \sigma_2^{(0)}
\end{array}\right.
,\quad w_e:=\frac{\ell(r-1)}{d}
\ee
where
\be
\label{eq: I and J}
\left|\begin{array}{lll}
J(w_e) &=& \displaystyle d^2\left(\frac{1-\ell}{\ell}\right)^2w_e^2-\frac{2d(d-1)(\ell+1)}{\ell}w_e+(d-1)^2,\\[3mm]
 I(\sigma) &=& J(w_e)-4dw_e+4\ell(\ell+d-1)\sigma^2.
 \end{array}\right.
 \ee 
 Moreover, 
\be
\label{eq: Monotonicity of w_2^pm}
\forall \sigma> \sigma_2^{(0)}, \ \  (w_2^-)'(\sigma)<0, \ \ (w_2^+)'(\sigma)>0.
\ee
\end{lemma}
\begin{proof}
The proof is verbatim the same as that of Lemma 2.1 in \cite{MRRSprofile}.
\end{proof}

\begin{lemma}[Roots of $\Delta _1$] 
\label{lem: Delta_1 Roots} 
Assume \eqref{eq: Parameters}. For all $\sigma\ge 0$, the equation $\Delta_1(w,\sigma)=0$ has exactly three distinct root branches $w_1(\sigma)<w_2(\sigma)<w_3(\sigma)$ which satisfy the following:\\\\
\underline{relative positions}: $\forall \sigma \ge 0$, 
\be
\label{eq: Relative Positions of w_i}
 -w_1(\sigma)\leq 0<w_e<w_2(\sigma)\leq 1<r\leq w_3(\sigma).
\ee
\underline{monotonicity}: $\forall \sigma>0$, \ \ 
\be
\label{eq: Monotonicity of w_i}
w_1'(\sigma)<0, \ \ w'_2(\sigma)<0, \ \ w_3'(\sigma)>0\ \ \mbox{for}\ \ \sigma>0.
\ee
\underline{asymptotics}:
\be
\label{eq: Asymptotes of w_i at 0}
\left|\begin{array}{lll}
\displaystyle w_1(\sigma)=-\frac{dw_e}{r}\sigma^2+\mathcal O(\sigma^3),\\
\displaystyle w_2(\sigma)=1-\frac{d(1-w_e)\sigma^2}{r-1}+\mathcal O(\sigma^3),\\
\displaystyle w_3(\sigma)=r+\frac{d(r-w_e)}{r(r-1)}\sigma^2+\mathcal O(\sigma^3),
\end{array}\right. \ \ \mbox{as}\ \ \sigma \to 0
\ee
and 
\be
\label{eq: Asymptotes of w_i at infty}
\left|\begin{array}{lll}
\displaystyle w_1(\sigma)=-\sqrt{d}\sigma +\mathcal O(1),\\[1mm]
\displaystyle w_2(\sigma)=w_e+\mathcal O(\sigma^{-2}),\\[1mm]
\displaystyle w_3(\sigma)=\sqrt{d}\sigma +\mathcal O(1)
\end{array}\right.\ \ \mbox{as}\ \ \sigma\to +\infty.
\ee
\end{lemma}
\begin{proof}
The proof is verbatim the same as that of Lemma 2.3 in \cite{MRRSprofile}.
\end{proof}

\begin{lemma}[Double roots]
\label{lem: Double Roots}
Assume \eqref{eq: Parameters}. The solutions to $\Delta_1=\Delta_2=0$ are:
\be
\label{eq: Double Roots}
\left|\begin{array}{lll} 
P_1=(0,1), \ \  P_2=(1-w_-,w_-)\ \ P_3=(1-w_+,w_+),\\
P_4=(0,0), \ \ P_5=(\sigma_5,w_5), \ \ P'_5=(0,r),
\end{array}\right.
\ee
where the points are defined as follows:\\\\
\underline{$P_5$ point}. 
\be
\label{eq: P_5}
P_5=\left(\sigma(P_5)=\frac{r\sqrt{d}}{d+\ell},\ w(P_5)=\frac{\ell r}{d+\ell}\right)
 \ee
\underline{$P_2$, $P_3$ points}. Let $J(w_e)$ as in \eqref{eq: I and J}. Then
\be
\label{eq: w_pm}
w_{\pm}=\frac{1}{2(d-1)}\left(dw_e+d-1-\frac{dw_e}{\ell}\pm\sqrt{J(w_e)}\right)
\ee 
and $P_2$, $P_3$ are on the phase portrait iff  
\be
\label{eq: Existence of P_2, P_3}
w_\pm  \ \ \mbox{real} \Leftrightarrow (w_e<w_{\ell}^- \ \ \mbox{or}\ \ w_e>w_\ell^+)
\ee
with 
\be
\label{eq: w_l^pm}
w_{\ell}^{\pm}=\frac{\ell(d-1)}{d(1-\ell)^2}\left[\ell+1\pm2\sqrt{\ell}\right].
\ee
If \eqref{eq: Parameters} then, $w_e<w_\ell^-$ holds so $P_2$, $P_3$ are indeed on the phase portrait.\\\\
\underline{Location}.  When defined, $P_2,P_3,P_5$ are located on the curve of the middle root $(\sigma, w_2(\sigma))$ of $\Delta_1$. Moreover, $P_2,P_5$ are on the curve of the lower root $w_2^-$ of $\Delta_2$.\\\\
\underline{Position of the middle root}. Let $w_e<w_\ell^-$ and $w_2(\sigma)$ be the middle root of $\Delta _1$, then the relative position of the middle root with respect to the sonic line is:
\be
\label{eq: Position of Middle Root}
\sigma+w_2(\sigma)\left|\begin{array}{l}
>1\ \ \mbox{for}\ \ 0<\sigma<\sigma(P_3)\\
<1\ \ \mbox{for}\ \ \sigma(P_3)<\sigma<\sigma(P_2)\\
>1\ \ \mbox{for}\ \ \sigma>\sigma(P_2).
\end{array}\right.
\ee
\end{lemma}
\begin{proof}
The proof is verbatim the same as that of Lemma 2.4 in \cite{MRRSprofile}.
\end{proof}

\begin{lemma}[Geometry of the roots as $r\uparrow r_+$]
\label{phasperotrai}
Suppose \eqref{eq: Parameters} holds. Then\\\\
\underline{Relative positions}.
$$
\sigma(P_3)<\sigma(P_2)<\sigma(P_5),
$$
\underline{Convergence}.  For $r=r_+(d,\ell)$, $P_2=P_3$ and
\be
\label{eq: Eye Property}
\lim_{r\uparrow r_+} |P_2-P_3|=0
\ee
\underline{Positions of $P_2,P_3$ on $\Delta_2=0$ curve}.  For $0<r_+(d,\ell)-r\ll1$ sufficiently small, $P_2$ and $P_3$ lie on the lower root curve $(\sigma,w_2^-(\sigma))$ of $\Delta_2$. Also, we have
\be
\label{eq: P_3^hat}
w_2^+(\sigma_3)>1.
\ee
\end{lemma}
\begin{proof}
\underline{Relative positions}.  In view of \eqref{eq: Double Roots}, we infer from \eqref{eq: Parameters} that
$$
\sigma(P_2)-\sigma(P_3)=\frac{\sqrt{J(w_e)}}{d-1}>0.
$$
Note also that
$$
\sigma(P_5)-\sigma(P_3)=\frac{\sqrt{J(w_e)}+A}{2(d-1)},\quad \left|
\begin{array}{l}
A=\left(\frac{2(d-1)\sqrt d}{d+\ell}-1+\ell\right)(r-r_0(\ell))\\
r_0(\ell)=\frac{d+\ell-2}{\frac{2(d-1)\sqrt d}{d+\ell}-1+\ell}>0.
\end{array}\right.
$$
Note that
$$
\frac{2(d-1)\sqrt d}{d+\ell}-1+\ell>0
$$
and
$$
r^*(d,\ell)-r_0=\frac{(\sqrt d-1)(d+\ell)(d-\ell)}{(\sqrt d+\ell)(2(d-1)\sqrt d+(\ell-1)(d+\ell))},
$$
so for $\ell<d$, we have $r^*(d,\ell)>r_0$. Thus, by \eqref{eq: Parameters}, we infer $A>0$ i.e. $\sigma(P_3)<\sigma(P_5)$. Moreover,
$$
w(P_5)+\sigma(P_5)=\frac{\ell r}{d+\ell}+\frac{r\sqrt d}{d+\ell}=\frac{r}{r^*(d,\ell)}>1.
$$
Thus, in view of \eqref{eq: Position of Middle Root}, since $\sigma(P_3)<\sigma(P_5)$ it follows that $\sigma(P_2)<\sigma(P_5)$.\\\\
\underline{Convergence}.  This is an immediate consequence of $J(w_e(r=r_+))=0$.\\\\
\underline{Positions of $P_2,P_3$ on $\Delta_2=0$ curve}.  Recall from Lemma \ref{lem: Double Roots} that $P_2$ lies on the lower root curve $(\sigma,w_2^-(\sigma))$ of $\Delta_2$. Also, for $r=r_+(d,\ell)$, we know that $P_2=P_3$. By continuous dependence of the point $P_3$ and the curves $(\sigma,w_2^-(\sigma))$ in $r$, we infer that $P_3$ lies on the curve $(\sigma,w_2^-(\sigma))$ for $0<r_+(d,\ell)-r\ll1$. We now prove \eqref{eq: P_3^hat}. In view of \eqref{eq: Slopes at P_3} and \eqref{eq: Parameters}, we have in the limit $r\uparrow r_+$ that
$$
\ell(\sigma_3^\infty)^2=\frac{\ell}{(1+\sqrt \ell)^2}> \frac{d-1}{(1+\sqrt \ell)^2}=\frac{d}{\ell} w_e(r_+)
$$
which implies
\bee
&&I(\sigma_3^\infty)-\left(-\frac{d(\ell-1)}{\ell} w_e(r_+)+d-1\right)^2\\
&=&-4\frac{d(d-1)}{\ell} w_e(r_+)-4dw_e(r_+)+4\ell(\ell+d-1)(\sigma_3^\infty)^2>0\\
&\Rightarrow&\sqrt{I(\sigma_3^\infty)}>-\frac{d(\ell-1)}{\ell} w_e(r_+)+d-1\\
&\Rightarrow& w_2^+(\sigma_3^\infty)>1.
\eee
This, together with the continuous dependence of $\sigma_3$ and the curve $w_2^+$ in $r$, \eqref{eq: P_3^hat} follows.
\end{proof}

\subsection{Slopes at $P_3$}
For $(\sigma_3,w_3)=(\sigma(P_3),w(P_3))$, we define the first order derivatives of $\Delta_1,\Delta_2$ at $P_3$:
\be
\label{eq: Slopes}
\left|\begin{array}{llll}
c_1=\pa_w\Delta_1(P_3)=3w_3^2-2(r+1)w_3+r-d\sigma_3^2\\
c_2=\pa_w\Delta_2(P_3)=\frac{\sigma_3}{\ell}[2w_3(\ell+d-1)-(\ell+d+\ell r-r)]\\
c_3=\pa_\sigma \Delta_1(P_3)=-2d\sigma_3w_3+2\ell(r-1)\sigma_3\\
c_4=\pa_\sigma\Delta_2(P_3)=-2\sigma_3^2,
\end{array}\right.
\ee

\begin{lemma}[Sign of the slopes] 
\label{lem: Sign of Slopes}
Assume \eqref{eq: Parameters}. Then
\be
\left|\begin{array}{l}
c_i<0, \ \ 1\le i\le 4\\
c_2c_3-c_1c_4<0.
\end{array}\right.
\ee
\end{lemma}
\begin{proof}
The proof is verbatim the same as that of Lemma 2.8 in \cite{MRRSprofile}.
\end{proof}

\begin{remark}[Slopes and eigenvalues]
We will also compute the slopes of any integral curve passing through $P_3$. it turns out that there are only two possible values:  
\be
\label{eq: c_pm}
c_\pm=\frac{c_4-c_1\pm\sqrt{(c_1-c_4)^2+4c_2c_3}}{2|c_2|}.
\ee 
This follows since $c_\pm$ are the solutions of the equation
\be
\label{eq: c_pm 2}
c_{\pm}=\frac{c_1c_\pm+c_3}{c_2c_\pm+c_4}.
\ee
The characteristic matrix
$$
\mathcal A(P_3)=\begin{pmatrix} c_1 &c_3\\ c_2&c_4\end{pmatrix}
$$ 
possesses the following eigenvalues:  
\be
\label{eq: lambda_pm}
\l_{\pm}=\frac{c_1+c_4\pm \sqrt{(c_1-c_4)^2+4c_2c_3}}{2}
\ee 
It may be diagonalized as follows:
$$
P^{-1}\left[\mathcal A(P_3)\right]P=\left(\begin{array}{ll} \l_+&0\\ 0&\l_-\end{array}\right)
$$
with
\be
\label{eq: P}
P=\left(\begin{array}{ll} c_-&c_+\\ 1&1\end{array}\right), \ \ P^{-1}=\frac{1}{c_+-c_-}\left(\begin{array}{ll} -1&c_+\\1&-c_-\end{array}\right).
\ee
\end{remark}
\begin{lemma}[Estimates on the slopes]
\label{lem: Estimate Slopes}
Assume \eqref{eq: Parameters} and let 
\be
\label{eq: A}
 A =\frac{\l_-}{\l_+}= \frac{c_1c_4-c_2c_3}{(c_4+c_2c_-)^2},
\ee
Then,
\be
\label{eq: Sign of Eigenvalues}
\left|\begin{array}{l}
c_-<0<c_+\\
c_4+c_2c_- <0\\
A>1\\
-\frac{c_4}{c_2}< c_- < -\frac{c_3}{c_1}<-1\\
\l_-<\l_+<0
\end{array}\right.
\ee
\end{lemma}
\begin{proof}
We prove $-\frac{c_4}{c_2}< c_- < -\frac{c_3}{c_1}<-1$. The rest of the proof is verbatim the same as that of Lemma 2.9 in \cite{MRRSprofile}. Since $\{\Delta_1=0\}$ intersects $w=1-\sigma$ at $P_2$ and $P_3$, $(\sigma, w_2(\sigma))$ is below $w=1-\sigma$ for $\sigma>1-w_+$ and above for $\sigma<1-w_+$. Since $-c_3/c_1$ is the slope of $w=w_2(\sigma)$ at $P_3$, we infer 
$$
-\frac{c_3}{c_1}<-1.
$$
It remains to compare $c_-$ to $-c_3/c_1$. We compute
$$
\begin{aligned}
c_-+\frac{c_3}{c_1} =&\ c_-+\frac{|c_3|}{|c_1|}=\ \frac{|c_1|(c_4-c_1 -\sqrt{(c_1-c_4)^2+4c_2c_3})+2c_2c_3}{2|c_1||c_2|}\\ 
=&\ \frac{c_1(c_1-c_4)+2c_2c_3 -|c_1|\sqrt{(c_1-c_4)^2+4c_2c_3}}{2|c_1||c_2|}. 
\end{aligned}
$$
Now, we have
$$
\begin{aligned}
&\ \Big(c_1(c_1-c_4)+2c_2c_3\Big)^2 - \Big(|c_1|\sqrt{(c_1-c_4)^2+4c_2c_3}\Big)^2\\
=&\  c_1^2(c_1-c_4)^2+4c_2c_3c_1(c_1-c_4)+4c_2^2c_3^2 -c_1^2(c_1-c_4)^2-4c_1^2c_2c_3\\
=&\ 4c_2c_3(c_2c_3-c_1c_4) < 0
\end{aligned}
$$
and hence $c_-+\frac{c_3}{c_1}< 0.$ Thus, we infer $c_-<-\frac{c_3}{c_1}<-1$ so done by $c_4+c_2c_-<0$.
\end{proof}

\section{General properties of the dynamical system \eqref{eq: Autonomous System}}

In this section we establish the general properties of Lemma \ref{vnioneneno} for the dynamical system \eqref{eq: Autonomous System}. Assume that \eqref{eq: Parameters} hold so that the shape of the phase portrait is given by Figure \ref{fig: Phase Portrait}. We recall that  $w_2(\sigma)$ is the middle root of $\Delta_1$ and $w_2^{-}(\sigma)$ is the smallest root of $\Delta_2$ given by \eqref{eq: Delta_2 Roots}. The arguments in this section are classical and are given for the reader's convenience.

\subsection{The solution emerging from $P_1$}

We first claim the existence and uniqueness (up to the scaling symmetry) of a spherically symmetric solution to \eqref{eq: Autonomous System} which exists on the interval $[Z_1,Z_3]$ and, in the variables of Emden 
transform, corresponds to the integral curve $P_1-P_3$.

\begin{lemma}[The solution emerging from $P_1$]
\label{lem: Solution from P_1}
Assume \eqref{eq: Parameters}, $r_+(d,\ell)-r\ll 1$ and recall from \eqref{eq: Delta_2 Roots} the definition of $w_e$.
\begin{enumerate}
\item Existence: There is $\sigma_0>0$ small enough and a unique curve solution $w(\sigma)$ to \eqref{eq: Autonomous System} on $[0,\sigma_0]$ with $w(0)=1$. It admits the asymptotic expansion:
\be
\label{eq: Asymptote at P_1}
w(\sigma)=1+\frac{\ell^2(w_e-1)}{(\ell+2)w_e}\sigma^2+\mathcal O_{\sigma\to 0}(\sigma^4).
\ee
\item Original variables: The curve corresponds to a spherically symmetric solution of \eqref{eq: Autonomous System} defined on the interval $|Z|\in [Z_1,Z_3]$. This solution belongs to 
$C^\infty(Z_1<|Z|< Z_3)$.\\
\vskip .3pc
\item Reaching $P_3$: The solution curve is invertible with inverse $\sigma(w)$ near $\sigma=0$. Then, we have $\sigma(w)\in C^\infty((w(P_3),1))$ with $\sigma(w(P_3))=\sigma(P_3)$ (see Figure \ref{fig: Phase Portrait}). Furthermore,
\be
\label{eq: Bound on Sigma}
\forall w\in (w(P_3),1),\quad \sigma(w)<\sigma(P_3).
\ee
\end{enumerate}
\end{lemma}

\begin{proof}[Proof of Lemma \ref{lem: Solution from P_1}] This follows from the asymptotic behavior of the polynomials $\Delta_1,\Delta_2$.\\

\noindent{\bf step 1} Flow near $P_1$. Note that $\widehat{w}=w-1$ satisfies $\widehat{w}(0)=0$ and
\bee
\widehat{w}' +\frac{\ell}{\sigma}\widehat{w} &=& \frac{\widehat{w}w(w-r)-d(w-w_e)\sigma^2}{\frac{\sigma}{\ell}\Big[(\ell+d-1)w^2-w(\ell+d+\ell r-r)+\ell r-\ell \sigma^2\Big]} +\frac{\ell}{\sigma}\widehat{w}\\
&=& \frac{1}{\sigma}\cdot\frac{(\widehat{w}^2-\sigma^2)((\ell+d)w-\ell r)}{\frac{1}{\ell}((\ell+d-1)w^2-(\ell+d+r(\ell-1))w+\ell r)-\sigma^2}=: \frac{1}{\sigma}\cdot\frac{(\widehat{w}^2-\sigma^2)P(w)}{Q(w)-\sigma^2}
\eee
for some polynomials $P$, $Q$ non-vanishing at $1$, so
\be
\label{eq: Self-similar Equation Near P_1}
\left(\sigma^\ell\widehat{w}\right)' = \sigma^\ell\left(\widehat{w}' +\frac{\ell}{\sigma}\widehat{w}\right)= \sigma^{\ell-1}\frac{(\widehat{w}^2-\sigma^2)P(w)}{Q(w)-\sigma^2}.
\ee
\underline{Existence}. We solve
\be
\label{eq: Integral Equation Near P_1}
\widehat{w} =\sigma^{-\ell}\int_0^\sigma\frac{(\widehat{w}^2-\bar\sigma^2)P(w)}{Q(w)-\bar\sigma^2}\bar\sigma^{\ell-1}d\bar\sigma=: \mathcal F[\widehat{w}]
\ee
using an elementary fixed point argument, which yields the existence and uniqueness on $\sigma \in[0,\sigma_0]$ of $\widehat{w}$ of solution with
\bee
\label{eq: Priori Bound Near P_1}
|\widehat{w}| \lesssim \sigma^2\textrm{ for }\sigma\le\sigma_0
\eee
for some $\sigma_0>0$. \\\\
\noindent\underline{Asymptotics as $\sigma\to 0$}. The integral equation \eqref{eq: Integral Equation Near P_1} for $\widehat{w}$ and the fact that
\bee
\frac{(\widehat{w}^2-\sigma^2)((\ell+d)w-\ell r)}{\frac{1}{\ell}((\ell+d-1)w^2-(\ell+d+r(\ell-1))w+\ell r)-\sigma^2}=\ell^2\left(1-\frac{1}{w_e}\right)\sigma^2+\mathcal O_{\sigma\to0}(\sigma^4)
\eee
and 
\bee
\widehat{w} =\sigma^{-\ell}\int_0^\sigma\frac{(\widehat{w}^2-\bar\sigma^2)P(w)}{Q(w)-\bar\sigma^2}\bar\sigma^{\ell-1}d\bar\sigma=\ell^2\left(1-\frac{1}{w_e}\right)\sigma^{-\ell}\int_0^\sigma\bar\sigma^{\ell+1}d\bar\sigma+\mathcal O_{\sigma\to0}(\sigma^4)
\eee
yield \eqref{eq: Asymptote at P_1}.\\
\vskip .3pc
\noindent{\bf step 2} Reaching $P_3$. Let $w(\sigma)$ be the unique curve entering $P_3$ constructed in step 1. In view of the asymptotic behaviour \eqref{eq: Asymptote at P_1} we have that $w(\sigma)\in (w_2(\sigma),1)$ which implies that $\Delta_2> 0$, $\Delta_1< 0$ near $\sigma=0$ along our solution curve. Then, $\sigma(w)$ is monotone decreasing hence, is invertible near $\sigma=0$ i.e. there exists $0<1-w_0\ll1$ and smooth inverse $\sigma(w)$ on $(w_0,1)$.\\\\
Let $w_0\in [w(P_3),1]$ be the minimal value such that $\sigma(w)\in [w_2^{-1}(w),\sigma(P_3)]$ for all $w\in (w_0,1)$. We claim that for $d-1<\ell$ and $r_+-r\ll1$ that $w_0=w(P_3)$ so that $\sigma(w(P_3))=\sigma(P_3)$. Since $\Delta<0$, $\Delta_1<0$ on 
$$
\Big\{(\sigma,w)\,\Big|\,w\in (w(P_3),1),\,\sigma\in(w_2^{-1}(w),\sigma(P_3)]\Big\}
$$
which is a set of ordinary points for the system \eqref{eq: Autonomous System} i.e. $\Delta$ is non-vanishing, if $\sigma(w_0)\in (w_2^{-1}(w_0),\sigma(P_3))$, then we can extend the curve $\sigma(w)$ to $(w_0-\varepsilon, 1)$ with $\sigma(w)\in [w_2^{-1}(w),\sigma(P_3)]$ for some $\varepsilon>0$, contradicting minimality of $w_0$. Thus, either $\sigma(w_0)=w_2^{-1}(w_0)$ or $\sigma(w_0)=\sigma(P_3)$.\\\\
If $\sigma(w_0)=w_2^{-1}(w_0)$ at $x=x_0$, then 
$$
\frac{dw}{d\sigma}\bigg|_{x=x_0}=0,\quad \frac{d\sigma}{dx}\bigg|_{x=x_0}>0.
$$
Then, $\sigma(w(x))<w_2^{-1}(w(x))$ for $0<x_0-x\ll1$, a contradiction since $\sigma(w(x))\in [w_2^{-1}(w(x)),\sigma(P_3)]$ for $x<x_0$.\\\\
Then, $\sigma(w_0)=\sigma(P_3)=1-w_+$. Let 
$$
\widehat{P}_3=\left(\sigma_3, w_2^+(\sigma_3)\right)
$$
so that $\Delta_2(\widehat P_3)=0$. Recall from \eqref{eq: P_3^hat} that for $d-1< \ell$ and $r_+-r\ll1$, it holds that $w(\widehat P_3)>1$ so we have for $\sigma=\sigma(P_3)$, $w\in(w(P_3),1]$ that $\Delta_1<0$, $\Delta_2<0$. Then, we can extend the curve $\sigma(w)$ to $(w_0-\varepsilon,1)$ with $\sigma(w)\in [w_2^{-1}(w),\sigma(P_3)]$ for some $\varepsilon>0$, contradicting the minimality of $w_0$. Thus, $w_0=w(P_3)$ and the claim follows.
\end{proof}

\subsection{Solutions crossing red between $P_2$ and $P_3$}

We now analyze trajectories that cross the middle root $w_2(\sigma)$ of $\Delta_1$ between $P_2$ and $P_3$.

\begin{lemma}[Solutions crossing green between $P_2$ and $P_3$]
\label{lem: Solutions Near P_4}
 Assume \eqref{eq: Parameters}. Let $\sigma_3<\sigma^*<\sigma_2$ and $w_u(\sigma)$ and $w_d(\sigma)$ be the upward and downward flow solutions respectively to \eqref{eq: Autonomous System} with the data $w(\sigma^*)=w_2^-(\sigma^*)$. Then:\\
 \vskip .3pc
\noindent{\em 1. Downward flow}: $w_d\in C^\infty((0,\sigma^*])$ and
 $$
 \lim_{\sigma\downarrow 0}w_d(\sigma)=0.
 $$
 Moreover, $\sigma\to 0$ corresponds to $x\to +\infty$ and there exist $(w_\infty,\sigma_\infty)\in \Bbb R\times \Bbb R_+^*$ such that
\be
\label{eq: Asymptote at P_4}
\left|\begin{array}{l} 
\sigma_d(x) = \sigma_\infty e^{-rx}\left(1+\mathcal O_{x\to\infty}(e^{-rx})\right),\\
w_d(x) = w_\infty e^{-rx}\left(1+\mathcal O_{x\to\infty}(e^{-rx})\right).
\end{array}\right.
\ee
{\em 2. Upward flow}: $w_u\in C^\infty((\sigma_3,\sigma^*])$ and 
\be
\label{eq: Behaviour Near P_3}
\left|\begin{array}{l}
\forall \sigma_3<\sigma<\sigma^*, \ \ w^-_2(\sigma)<w_u(\sigma)<w_2(\sigma)\\
\lim_{\sigma\downarrow  \sigma_3}w_u(\sigma)=w_3,
\end{array}\right.
\ee
see Figure \ref{fig: Phase Portrait}.
\end{lemma}
\begin{remark}
The above Lemma shows that all such solutions provide admissible $P_3-P_4$ connections.
\end{remark}
\begin{proof}[Proof of Lemma \ref{lem: Solutions Near P_4}] The fact that the solution generates a $P_3-P_4$ connection with forward flow trapped in the region \eqref{eq: Behaviour Near P_3} follows again directly from the phase portrait of Figure \ref{fig: Phase Portrait} and the monotonicity \eqref{eq: Monotonicity of w_2^pm} and \eqref{eq: Monotonicity of w_i}. We leave that to the reader, while we focus on the proof of the asymptotic expansion \eqref{eq: Asymptote at P_4} near $\sigma=0$.\\

\noindent{\bf step 1} Behavior of $w(\sigma)$. Let $\varphi=\frac{w}{\sigma}$
so that from \eqref{eq: Autonomous System},
\bee
\frac{dw}{d\sigma}=\frac{\varphi(\sigma\varphi-1)(\sigma\varphi-r)-d(\sigma\varphi-w_e)\sigma}{r -\frac{\ell+d+\ell r-r}{\ell}\sigma\varphi+\frac{\ell+d-1}{\ell}\sigma^2\varphi^2- \sigma^2}.
\eee
Since
\bee
\frac{dw}{d\sigma}=\sigma\frac{d\varphi}{d\sigma}+\varphi,
\eee
$\varphi$ solves
\bee
\sigma\frac{d\varphi}{d\sigma} &=& \frac{\varphi(\sigma\varphi-1)(\sigma\varphi-r)-d(\sigma\varphi-w_e)\sigma}{r -\frac{\ell+d+\ell r-r}{\ell}\sigma\varphi+\frac{\ell+d-1}{\ell}\sigma^2\varphi^2- \sigma^2}-\varphi\\
&=& \frac{ -(1+r) \sigma\varphi^2 +dw_e\sigma +\frac{\ell+d+\ell r-r}{\ell}\sigma\varphi^2 +\sigma^2\varphi^3 -d\sigma^2\varphi  -\frac{\ell+d-1}{\ell}\sigma^2\varphi^3+ \sigma^2\varphi}{r -\frac{\ell+d+\ell r-r}{\ell}\sigma\varphi+\frac{\ell+d-1}{\ell}\sigma^2\varphi^2- \sigma^2}
\eee
i.e.
\bea
\frac{d\varphi}{d\sigma} = \frac{ -(1+r)\varphi^2 +dw_e +\frac{\ell+d+\ell r-r}{\ell}\varphi^2 +\sigma\varphi^3 -d\sigma\varphi  -\frac{\ell+d-1}{\ell}\sigma\varphi^3+ \sigma\varphi}{r -\frac{\ell+d+\ell r-r}{\ell}\sigma\varphi+\frac{\ell+d-1}{\ell}\sigma^2\varphi^2- \sigma^2}.
\eea
This is a regular ODE at $\sigma=0$ and an elementary fixed point argument ensures the behaviour 
\bee
\varphi(\sigma)=\varphi(0)+\mathcal O_{\sigma\to0}(\sigma)
\eee
where $\varphi(0)$ is a real constant. Hence
\be
\label{eq: Asymptote of phi at P_4}
w(\sigma)=\sigma\varphi=\varphi(0)\sigma+\mathcal O_{\sigma\to 0}(\sigma^2).
\ee
The fact that the solution $w_d(\sigma)$ belongs to $ C^\infty((0,\sigma^*))$ follows from above since 
$P_4$ is the only singular point on this interval.\\

\noindent{\bf step 2} Behavior in $x$ and $Z$.  From \eqref{eq: Asymptote of phi at P_4}:
\bee
\frac{dx}{d\sigma}&=& -\frac{\Delta(\sigma, w_d(\sigma))}{\Delta_2(\sigma, w_d(\sigma))}= -\frac{(w_d(\sigma)-1)^2-\sigma^2}{\frac{\sigma}{\ell}\Big[(\ell+d-1)w_d(\sigma)^2-(\ell+d+\ell r-r)w_d(\sigma)+\ell r-\ell \sigma^2\Big]}\\
&=& -\frac{1}{r\sigma} + \mathcal O_{\sigma\to0}(1)
\eee
and hence
\bee
x &=& -\frac{1}{r}\log \sigma + x_2+ \mathcal O_{\sigma\to0}(\sigma)\quad \Rightarrow \quad e^x=\frac{e^{x_2}}{\sigma^{\frac{1}{r}}}\left(1+\mathcal O_{\sigma\to0}(\sigma)\right).
\eee
for some real constant $x_2$. This yields the following expansion for $\sigma$ as $x\to \infty$
\bee
\sigma(x) &=& \sigma_\infty e^{-rx}\left(1+\mathcal O_{x\to\infty}(e^{-rx})\right)
\eee
for some real constant $\sigma_\infty>0$. Plugging in the expansion for $w(\sigma)$, we infer
\bee
w_b(x) &=& w_\infty e^{-rx}\left(1+\mathcal O_{x\to\infty}(e^{-rx})\right)
\eee
for some real constant $w_\infty$, and \eqref{eq: Asymptote at P_4} is proved.
\end{proof}
The results of the previous two sections provide the proof of all the statements of the following lemma:

\begin{lemma}[General structure of spherically symmetric self-similar solutions]
\label{vnioneneno}
Assume \eqref{eq: Parameters}. Then for $0<r_+(d,\ell)-r\ll1$,
\begin{enumerate}
\item Solutions near the vacuum sphere: there exists a unique trajectory of \eqref{eq: Autonomous System} which connects $P_1$ to $P_3$. This trajectory corresponds to the unique (up to scaling)  (local) spherically symmetric solution to \eqref{eq: Autonomous System} which is $C^\infty$ on the interval $|Z|\in (Z_1,Z_3)$ with $Z_1$ corresponding 
to $P_1$ and $Z_3$ to $P_3$.
\vskip .4pc
\item Solutions near infinity: there exists a one parameter family of trajectories $P_3-P_4$. These curves correspond to the spherically symmetric solution to \eqref{eq: Autonomous System} and are in $C^\infty$
of the interval $|Z|\in (Z_3,\infty)$ with $\infty$ corresponding to $P_4$.
\vskip .4pc
\item Connection at $P_3$: in both cases, $P_3$ is reached in finite time, i.e., $0<Z_3<\infty$, 
and the solutions constructed in (1) and (2) can be glued continuously to each other.
\end{enumerate}
\end{lemma}
Our goal is to construct $C^\infty$ solution on $\{|Z|> Z_1\}$. At this point it is already clear that the crux of the matter is the point $P_3$.

\subsection{Diagonalized system at $P_3$}

The point $P_3$ will play an essential role in the proof of Theorem \ref{thmmain}. The dynamical properties of this point in the regime \eqref{eq: Parameters} can only be seen after passing to the diagonalized variables \eqref{eq: P}. We recall the values of the slopes \eqref{eq: Slopes}, \eqref{eq: c_pm}, \eqref{eq: lambda_pm}, the diagonalization matrices \eqref{eq: P} and the non degeneracy properties of Lemma \ref{lem: Sign of Slopes} and Lemma \ref{lem: Estimate Slopes} in the range \eqref{eq: Parameters}. We rewrite the system in coordinates which diagonalize its linear part. We will also introduce a time variable and recast the system as a dynamical flow approaching the point $P_3$ (from either side) as $t\to\infty$.

\begin{lemma}[Equations in the diagonal form]
\label{lem: Diagonalization}
Assume \eqref{eq: Parameters}. Let 
\be
\label{eq: Diagonalization}
\left|\begin{array}{l}
w=w_3+W\\
\sigma=\sigma_3+\Sigma\\
\frac{dt}{dx}=-\frac{1}{\Delta}
\end{array}\right., \ \ X=\left|\begin{array}{l} W\\\Sigma\end{array}\right., \ \ Y=P^{-1}X=\left|\begin{array}{ll} \Wt\\\Sigmat \end{array}\right.
\ee
then \eqref{eq: Autonomous System} becomes:
\be
\label{eq: Diagonalized Autonomous System}
\frac{dY}{dt}=\frac{1}{c_+-c_-}\left|\begin{array}{l}
\mathcal G_1\\
\mathcal G_2
\end{array}\right.
\ee
with
\bea
\label{eq: G_1}
\nonumber \mathcal G_1&=&(c_+-c_-)\l_+\Wt+\dt_{20}\Wt^2+\dt_{11}\Wt\Sigmat+\dt_{02}\Sigmat^2+\dt_{30}\Wt^3+\dt_{21}\Wt^2\Sigmat+\dt_{12}\Wt\Sigmat^2+\dt_{03}\Sigmat^3\\
& = & -\Delta_1+c_+\Delta_2,
\eea
\bea
\label{eq: G_2}
\nonumber \mathcal G_2&=&(c_+-c_-)\l_-\Sigmat+\et_{20}\Wt^2+\et_{11}\Wt\Sigmat+\et_{02}\Sigmat^2+\et_{30}\Wt^3+\et_{21}\Wt^2\Sigmat+\et_{12}\Wt\Sigmat^2+\et_{03}\Sigmat^3\\
& = & \Delta_1-c_-\Delta_2
\eea
and where the values of the coefficients are collected in \eqref{eq: Coefficients of Autonomous System}, \eqref{eq: Coefficients of Autonomous System 2}.
\end{lemma}
\begin{remark}
Now that we have the same set up as in \cite{MRRSprofile}, the analysis of the ODE near $P_3$ will be identical to that of \cite{MRRSprofile} at least the algebraic level (although the values of the coefficients are different).
\end{remark}
\begin{proof}[Proof of Lemma \ref{lem: Diagonalization}] This is a direct computation.\\\\
\noindent{\bf step 1} Reexpressing $\Delta,\Delta_1,\Delta_2$. Let $w=w_3+W$, $\sigma=\sigma_3+\Sigma,$ we compute the nonlinear terms
\bea
\label{eq: Delta_1 in W, Sigma}
\nonumber &&\Delta_1=w^3-(r+1)w^2+rw-dw\sigma^2+\ell(r-1)\sigma^2\\
\nonumber& = & w_3^3+3w_3^2W+3w_3W^2+W^3-(r+1)(w_3^2+2w_3W+W^2)+r(w_3+W)\\
\nonumber&-&d(w_3+W)(\sigma^2_3+2\sigma_3\Sigma+\Sigma^2)+\ell(r-1)(\sigma_3^2+2\sigma_3\Sigma+\Sigma^2)\\
& = & c_1W+c_3\Sigma+d_{20}W^2+d_{11}W\Sigma+d_{02}\Sigma^2+W^3-dW\Sigma^2
\eea
and 
\bea
\nonumber &&\Delta_2=\frac{\sigma}{\ell}\Big[(\ell+d-1)w^2-w(\ell+d+\ell r-r)+\ell r-\ell \sigma^2\Big]\\
\nonumber & = & \frac{\sigma_3+\Sigma}{\ell}\left[(\ell+d-1)(w_3^2+2w_3W+W^2)-(\ell+d+\ell r-r)(w_3+W)+\ell r-\ell(\sigma_3^2+2\sigma_3\Sigma+\Sigma^2)\right]\\
& = & c_2W+c_4\Sigma+e_{20}W^2+e_{11}W\Sigma+e_{02}\Sigma^2+e_{21}W^2\Sigma-\Sigma^3
\eea
where the parameters are given by \eqref{eq: Coefficients of Autonomous System}. Here, we have used that $\Delta_1=0$, $\Delta_2=0$ at $P_3$.\\

\noindent{\bf step 2} $Y$ variable. We now pass to the $Y$ variable: from {\bf step 1} and \eqref{eq: P}, it follows that 
\bee
&&\Delta_1(X)= c_1(c_-\Wt+c_+\Sigmat )+c_3(\Wt+\Sigmat)+d_{20}(c_-\Wt+c_+\Sigmat )^2+d_{11}(c_-\Wt+c_+\Sigmat )(\Wt+\Sigmat)\\
&+& d_{02}(\Wt+\Sigmat)^2+(c_-\Wt+c_+\Sigmat )^3-d(c_-\Wt+c_+\Sigmat )(\Wt+\Sigmat)^2\\
& = & (c_1c_-+c_3)\Wt+(c_1c_++c_3)\Sigmat\\
& + & (d_{20}c_-^2+d_{11}c_-+d_{02})\Wt^2+(2c_-c_+d_{20}+(c_-+c_+)d_{11}+2d_{02})\Wt\Sigmat+(d_{20}c_+^2+d_{11}c_++d_{02})\Sigmat^2\\
& + & (c_-^3-dc_-)\Wt^3+(3c_-^2c_+-2dc_--dc_+)\Wt^2\Sigmat+(3c_-c_+^2-dc_--2dc_+)\Wt\Sigmat^2+(c_+^3-dc_+)\Sigmat^3
\eee
and 
\bee
&& \Delta_2(X)= c_2(c_-\Wt+c_+\Sigmat )+c_4(\Wt+\Sigmat)\\
&+& e_{20}(c_-\Wt+c_+\Sigmat )^2+e_{11}(c_-\Wt+c_+\Sigmat )(\Wt+\Sigmat)+e_{02}(\Wt+\Sigmat)^2\\
& + & e_{21}(c_-\Wt+c_+\Sigmat )^2(\Wt+\Sigmat)-(\Wt+\Sigmat)^3\\
& = & (c_2c_-+c_4)\Wt+(c_2c_++c_4)\Sigmat\\
& + & (e_{20}c_-^2+e_{11}c_-+e_{02})\Wt^2+(2e_{20}c_-c_++e_{11}(c_-+c_+)+2e_{02})\Wt\Sigmat+  (e_{20}c_+^2+e_{11}c_++e_{02})\Sigmat^2\\
& + & (e_{21}c_-^2-1)\Wt^3+(e_{21}(c_-^2+2c_-c_+)-3)\Wt^2\Sigmat+(e_{21}(2c_-c_++c_+^2)-3)\Wt\Sigmat^2+(e_{21}c_+^2-1)\Sigmat^3.
\eee
Linear terms on the right hand side of \eqref{eq: Coefficients of Autonomous System 2} yield
\bee
&&\frac{1}{c_+-c_-}\left|\begin{array}{l}-((c_1c_-+c_3)\Wt+(c_1c_++c_3)\Sigmat)+c_+[(c_2c_-+c_4)\Wt+(c_2c_++c_4)\Sigmat]\\ 
(c_1c_-+c_3)\Wt+(c_1c_++c_3)\Sigmat-c_-[(c_2c_-+c_4)\Wt+(c_2c_++c_4)\Sigmat]
\end{array}\right.\\
& = &\frac{1}{c_+-c_-}\left|\begin{array}{l} 
[-(c_1c_-+c_3)+c_+(c_2c_-+c_4)]\Wt+[-(c_1c_++c_3)+c_+(c_2c_++c_4)]\Sigmat\\
((c_1c_-+c_3)-c_-(c_2c_-+c_4))\tilde{W}+[(c_1c_++c_3)-c_-(c_2c_++c_4)]\Sigmat
\end{array}\right.
\eee
We then recall the $c_\pm$ equation \eqref{eq: c_pm 2} which kills the off diagonal term
$$\left|\begin{array}{l}
-(c_1c_++c_3)+c_+(c_2c_++c_4)=0\\
(c_1c_-+c_3)-c_-(c_2c_-+c_4)=0
\end{array}\right.
$$ 
and compute:
\be
\label{eq: lambda_pm in c}
\left|\begin{array}{ll}
\frac{-(c_1c_-+c_3)+c_+(c_2c_-+c_4)}{c_+-c_-}=\frac{-c_-(c_2c_-+c_4)+c_+(c_2c_-+c_4)}{c_+-c_-}=c_2c_-+c_4=\l_+\\
\frac{(c_1c_++c_3)-c_-(c_2c_++c_4)}{c_+-c_-}=\frac{c_+(c_2c_++c_4)-c_-(c_2c_++c_4)}{c_+-c_-}=c_2c_++c_4=\l_-\\
\end{array}\right.
\ee
Introducing the time variable $t$ (we note that $t\to \infty$ corresponds to both $Z\uparrow\downarrow Z_3$): 
$$\frac{dt}{dx}={-}\frac{1}{\Delta}$$
yields from \eqref{eq: P},
\bea
\label{eq: Coefficients of Autonomous System 2}
&&(c_+-c_-) \frac{dY}{dt}=\begin{pmatrix}-1 & c_+\\1 &-c_-\end{pmatrix} \left|\begin{array}{l}\Delta_1\\\Delta_2\end{array}\right.\\
\nonumber&=&\left|\begin{array}{l}
(c_+-c_-)\l_+\Wt+\dt_{20}\Wt^2+\dt_{11}\Wt\Sigmat+\dt_{02}\Sigmat^2+\dt_{30}\Wt^3+\dt_{21}\Wt^2\Sigmat+\dt_{12}\Wt\Sigmat^2+\dt_{03}\Sigmat^3\\
(c_+-c_-)\l_-\Sigmat+\et_{20}\Wt^2+\et_{11}\Wt\Sigmat+\et_{02}\Sigmat^2+\et_{30}\Wt^3+\et_{21}\Wt^2\Sigmat+\et_{12}\Wt\Sigmat^2+\et_{03}\Sigmat^3
\end{array}\right.
\eea
where the coefficients $\dt_{ij}$ and $\et_{ij}$ can be computed explicitly using the above expression for $\Delta_1$ and $\Delta_2$ in terms of $\Wt$ and $\Sigmat$.
\end{proof}

\subsection{Integral curves passing through $P_3$}

\begin{lemma}[Slope of the curves converging to $P_3$]
\label{lem: Solutions Near P_3}
Assume \eqref{eq: Parameters}. Let $c_+$, $c_-$, $A$ be given by \eqref{eq: c_pm}, \eqref{eq: A}.  Then, 
\begin{enumerate}
\item all  integral curves of \eqref{eq: Autonomous System} converging to $P_3$ have slope given either by $c_+$ or $c_-$, 
 
\item there is only one curve with slope $c_+$, while all the others curves have slope $c_-$,

\item the unique curve converging to $P_3$ with the slope $c_+$ exists on $0<\sigma<\sigma_3$ and converges to $P_4$ as $\sigma \to 0$.
\end{enumerate}
\end{lemma}

\begin{proof}
The Jacobian matrix of the autonomous system \eqref{eq: Diagonalized Autonomous System} at the equilibrium $(\Wt, \Sigmat)=(0,0)$ is diagonal with two negative eigenvalues $\l_-<\l_+<0$. Thus, standard results imply that $(0,0)$ is an asymptotically stable node, and the following holds for the  trajectories converging to $(0,0)$ as $t\to \infty$
\begin{enumerate}
\item there exists a unique  trajectory tangent to the eigenvector of the Jacobian matrix corresponding to the smallest eigenvalue $\l_-$, i.e. these trajectories satisfy
\bee
\lim_{t\to \infty}\frac{\Wt}{\Sigmat}=0,
\eee
\item all the other trajectories are tangent to the eigenvector of the Jacobian matrix corresponding to the largest eigenvalue $\l_+$, i.e. this trajectory satisfies 
\bee
\lim_{t\to \infty}\frac{\Sigmat}{\Wt}=0.
\eee
\end{enumerate}
Coming back to $(W, \Sigma)$, we infer that the slope of any curve converging to $P_3$ is either $c_+$ or $c_-$, and there is a unique one with slope $c_+$, while all the others have slope $c_-$.\\\\
The fact that the unique curve converging to $P_3$ with slope $c_+$ is defined all the way to $\sigma=0$ and attracted to $P_4$ is a straightforward consequence of the phase portrait in Figure \ref{fig: Phase Portrait}.
\end{proof}

\begin{lemma}[Regularity at $P_3$, \cite{MRRSprofile}]
\label{lem: Regularity at P_3}
Assume \eqref{eq: Parameters}. Let $c_-$, $A$ be given by \eqref{eq: c_pm}, \eqref{eq: A} and assume that 
$$
A=K+\alpha, \ \ K\in \Bbb N\backslash\{0,1\}, \ \ 0<\alpha<1.
$$ 
Then, there exists a unique solution curve which is $C^\infty$ at $P_3$ with slope $c_-$. Furthermore, all the other curves converging to $P_3$ with slope $c_-$ (see Lemma \ref{lem: Solutions Near P_3}), are  $C^{K+\alpha}$ both on the right and on the left of $P_3$.
\end{lemma}

\section{Semi classical renormalization of the flow near $P_3$}

Our aim in this section is to start the proof of Theorem \ref{thmmain} with the renormalization of the flow \eqref{eq: Autonomous System} for $0<r_+-r\ll 1$ near $P_3$. For the rest of this paper, we assume 
$$
d-1<\ell<d,\quad 0<r_+-r\ll1
$$
so that the conclusion of Lemma \ref{vnioneneno} holds.

\begin{remark}[Notation for the parameters]
\label{rem: Parameters}
 From now on and for the rest of this paper we adopt the following notation: all slopes, characteristic eigenvalues and geometrical parameters involved in the renormalization of the flow, Lemma \ref{lem: Diagonalization}, depend on $r$, and will be noted with an $\infty$ superscript when evaluated at $r=r_+$. All variables will be noted with an $\peye$ subscript when evaluated at $P_2$. The non degeneracy and signs of some of these limiting values will be crucial in the forthcoming analysis, and all relevant values are collected in Appendix \ref{app: Slopes and Eigenvalues}.
\end{remark}

\subsection{Strategy of the proof of Theorem \ref{thmmain}}
\label{stegerguerproof}
%%%%%%%%%%%%%%%%%%%%%%%%%%%%%%%%%%%%%%%%%%%

We describe the main steps of the proof of Theorem \ref{thmmain}.\\

\noindent{\bf step 1} Renormalization. We introduce a suitable semi classical parameter $b>0$, see \eqref{eq: b} and use the geometry of the ``eye'' $P_3,P_2$ to produce a suitable renormalization. Here we use a fundamental degeneracy of the phase portrait, see Lemma \ref{phasperotrai}, which gives the eye property \be
\label{neioenenoevno}
\lim_{r\uparrow r_+}|P_3(r)-P_2(r)|=0,
\ee
{\em and} the slopes of $w_2(\sigma)$ and $w_2^-(\sigma)$ converge to the same value, see Figure \ref{fig: Phase Portrait}. Passing to the diagonalized variables \eqref{eq: Diagonalization} and after explicit suitable reductions, this degeneracy leads to a quadratic cancellation, \eqref{eq: P_3 in tilde Variables}. Solving for $w(\sigma)$, we are left with the study of a problem of the form:
 \be
 \label{nneionienoene}
 u(1-u)\Theta' -\left[\gamma-2+(\nu_b+3)u\right]\Theta=-\frac{\mathcal G}{b}.
 \ee
 Here we introduced the parameters which appear in the renormalization process
 \be\label{gara}
 \left|\begin{array}{l}
 \gamma=\frac{c_\infty(d,\ell)(1+o_{b\to 0}(1))}{b}, \ \ c_\infty>0\\
 \nu_b=\nu_\infty(d,\ell)+o_{b\to 0}(1), \ \ \nu_\infty\neq 0
\end{array}\right.
\ee
The trajectory is $\Theta(u)$ with $u=0$ at $\sigma_3$, $u=1$ at $\sigma_{\hskip -.1pc\peye}$.  $\mathcal G$ is an explicit nonlinear term. An analogous reduction can be performed to the left of $P_3$.\\
  
\noindent{\bf step 2} Main part of the solution. The nonlinear ODE \eqref{nneionienoene} has a regular singular point at the origin. Therefore, it admits a $\mathcal C^\infty$ solution with an holomorphic expansion at the origin $$\Theta(u)=\sum_{k=0}^{+\infty}\theta_k(b,d,\ell) u^k$$ where $\theta_k(b,d,\ell)$ is given by an explicit $b$ dependent induction relation. We let 
\be
\label{veiovneoneenoe}
\gamma-1=K+\alpha_\gamma, \ \ 0<\alpha_\gamma<1
\ee
and truncate the holomorphic expansion at the critical frequency\footnote{which corresponds to the limit of regularity of a generic solutions, see Lemma \ref{lem: Regularity at P_3}.}:
\be
\label{fmeneneneoe}
\Theta(u) =\sum_{k=0}^{K-2}\theta_k u^k+(-1)^{K-1}S_{K-1}\Theta_{\rm main}(u)-\T(r_\mathcal G)
\ee 
where $\Theta_{\rm main}=O(u^K)$ is an explicit integral, and $\T(r_\mathcal G)$ is a remainder which is of higher order. Our first fundamental observation is that there exists a strong limit 
\be
\label{limitexists}
\lim_{b\to 0}S_{K-1}= S_\infty(d,\ell).
\ee
The proof relies on bounding the formal series solution to a limiting problem with $b=0$ first. This is done in
 Proposition \ref{prop: Bound for Formal Limit Problem}, which belongs to the realm of nonlinear Maillet theorems, \cite{malgrange, sibuya}. 
 The original problem \eqref{nneionienoene} can be thought of as a $b$-deformation of the limiting problem. 
 The challenge however is that we need uniform estimates for all frequencies up to the critical value $K$ {\em which itself is of size $\sim \frac 1b$}\\

\noindent{\bf step 3} Non vanishing of $S_\infty(d,\ell)$. The proof of finiteness of $S_\infty(d,\ell)$ implies the analycity of the mapping $\ell\mapsto S_\infty(d,\ell)$ away from the critical point i.e., $\ell<d$. This number can be reexpressed as an explicit normally convergent series, but we do not know how to prove analytically that it is non zero. We therefore perform a numerical study of this convergent series which allows us to provide windows of parameters $(d,\ell)$ for which 
\be
\label{veioeonveonoe}
S_\infty(d,\ell)\ne 0.
\ee

\noindent{\bf step 4} No oscillation at the right of $P_3$. In the variable $\Theta$, it is easily seen that the $P_3-P_{\hskip -.1pc\peye}$ separatrix satisfies $|\Theta|\lesssim_{\ell,d} 1$. Hence we pick a large enough (in absolute value) constant $\Theta^*\gg1 $ and aim at reaching the value 
\be
\label{ceioneoneon}
\Theta(u^*)=\Theta^*, \ \ 0<u^*<1.
\ee 
The second fundamental observation is that the function $\Theta_{\rm main}$ can be analyzed explicitly near $u=0$ 
\bea
\label{formalurlajrioje}
\Theta_{\rm main}&=& \Gamma(\alpha_\gamma)\Gamma(1-\alpha_\gamma)K^{\nu_b+3-\alpha_\gamma}u^{K-1}\\
\nonumber &\times & \left\{\left[1+o_{b\to 0}(1)\right]\left[\frac{1}{\Gamma(1-\alpha_\gamma)}+\frac{K+\nu_b+2}{\Gamma(2-\alpha_\gamma)}u\right]+{\rm lot}\right\}.
\eea
In a boundary layer close to the integer values
\be
\label{neoineneoen}
\alpha_\gamma\in (\e_1,2\e_1)\cup (1-2\e_2,1-\e_2), \ \ \e_i=o_{b\to 0}(1),
\ee 
we can ensure that \eqref{ceioneoneon} happens for a small $0<u^*(\alpha_\gamma)\ll 1$.  For the  interval $2\e_1<\alpha_\gamma<1-2\e_2$, we need to understand $\Theta_{\rm main}$ away from $u=0$ where the truncated Taylor expansion no longer dominates, and here we use the {\em explicit integral representation of $\Theta_{\rm main}$} to show that \eqref{ceioneoneon} happens for $u^*(\alpha_\gamma)<\frac 12$.
The conclusion is that for every $\alpha_\gamma\in(0,1)$ except maybe a very small $b$ dependent boundary layer around the integer values, the solution to \eqref{fmeneneneoe} reaches \eqref{ceioneoneon} in time $0<u^*<\frac 12$. This means that we are leaving a large neighborhood around the $P_1-P_3$ separatrix with a prescribed sign $\Theta^*\gg 1$. A further use of monotonicity properties of the flow \eqref{eq: Autonomous System} allows us to conclude that the integral curve will intersect 
either the root branch $w_2^-(\sigma)$ or $w_2(\sigma)$ for some $\sigma_3<\sigma^*<\sigma_2$. 
The $(-1)^K$ prefactor in \eqref{fmeneneneoe} dictates that the former happens when $K$ is even while the latter holds 
when $K$ is odd (if $S_\infty>0$ and the other way around if $S_\infty<0$.) Once the trajectory reaches $w_2^-(\sigma)$, by Lemma \ref{lem: Solutions Near P_4}, it then continues on to $P_4$, as desired.\\

\noindent{\bf step 5} Oscillations at the left of $P_3$. The analysis of the flow to the right of $P_3$  produces the same decomposition \eqref{formalurlajrioje} but with $u<0$. We then observe since $\lim_{\alpha_\gamma\uparrow 1} \Gamma(1-\alpha_\gamma)=+\infty $ that by choosing $\alpha_\gamma$ in a boundary layer close to respectively $0$ or $1$, the sign $u<0$ allows us to reach 
$$\Theta(u^*)= \left|\begin{array}{l}
\Theta^*\ \ \mbox{for}\ \ \e_1<\alpha_\gamma<2\e_1\\
-\Theta^*\ \ \mbox{for}\ \ \e_2<1-\alpha_\gamma<2\e_2
\end{array}\right., \ \ u^*<0, \ \ |u^*|\ll1.$$ 
Then the curve will exit through $w_2^-(\sigma)$ in the first case, and through the branch 
$w_2(\sigma)$ in the second case. This holds for $S_\infty>0$.
For $S_\infty<0$ the picture is reversed. \\

\noindent{\bf step 6} Conclusion by continuity. Given $S_\infty(d,\ell)\ne 0$ with the suitable sign, we vary the parameter $\alpha_\gamma\in(\e_1,1-\e_2)$ continuously and conclude that the $\mathcal C^\infty$ curve through $P_3$ crosses the green at the right of $P_3$ for all $\alpha_\gamma\in(\e_1,1-\e_2)$, and is located on the right of the $P_3-P_4$ trajectory given by Lemma \ref{lem: Solution from P_1} at the beginning of the $\alpha_\gamma$ interval and is located on the left of the $P_3-P_4$ trajectory at the end of this interval. Hence an elementary continuity argument implies the existence of at least one value $\alpha_\gamma^*\in(\e_1,1-\e_2)$ such that the solution curve intersects the $P_3-P_4$ trajectory. The two curves intersect each other away 
from singular points and thus, by uniqueness, must coincide.  It follows easily that this constructed solution satisfies the conclusions of Theorem \ref{thmmain}. In other words, as long as the non degeneracy condition \eqref{veioeonveonoe} is satisfied, the integer interval $\gamma\in [K+1,K+2]$, with $K$ given by \eqref{veiovneoneenoe} large enough and of suitable parity, contains at least one $\mathcal C^\infty$ $P_1-P_4$ solution. Since $\gamma$ is related to $b$ through 
\eqref{gara} and $b=o_{r\uparrow r_+}(1)$, the above construction produces an infinite family of global $\mathcal C^\infty$ solutions 
parametrized by the speeds $r_n\uparrow r_+(d,\ell)$ for each $(d,\ell)$ such that $S_\infty(d,\ell)\ne 0$.

%%%%%%%%%%%%%%%%%%%%%%%%%%%%%%%%%%%%%%%%%%%%%%%%%%%%%%%

\subsection{Degeneracy of the geometry at $r_+(d,\ell)$}

We use Lemma \ref{lem: Diagonalization} and the $Y$ variable \eqref{eq: Diagonalization} to map \eqref{eq: Autonomous System} onto \eqref{eq: Diagonalized Autonomous System}. The starting point of the renormalization procedure is the following fundamental degeneracy property as $r\uparrow r_+(d,\ell)$.

\begin{lemma}[Degeneracy in the diagonalized system]
Let 
\be
\label{eq: b}
b=\sqrt{r_+-r},\quad \mu_+=\frac{\l_+}{b}
\ee
then
\be
\label{eq: P_3 in tilde Variables}
\left|\begin{array}{l}
\Wte=-b\frac{(c_+-c_-)\mu_+}{\dt_{20}}+\mathcal O(b^2)\\
\Sigmate=-\frac{\et_{20}\Wte^2}{(c_+-c_-)\l_-}+\mathcal O(b^3)\\
\mu_+=\mu_+^\infty+\mathcal O(b)<0
\end{array}\right.
\ee
where the non degenerate limiting values are computed in Appendix \ref{app: Slopes and Eigenvalues}.
\end{lemma}

\begin{proof} The value of $\sigma_3(r)$ is computed from \eqref{eq: w_pm} and hence $\sigma_3\in C^\infty(1,r_+)$. Moreover, $J(r)$ is from \eqref{eq: I and J} a second order polynomial with roots $r_+=1+\frac{d-1}{(1+\sqrt{\ell})^2}<r_-=1+\frac{d-1}{(1-\sqrt{\ell})^2}$ and hence the root $r_+$ is simple. Since $r^*<r_+$, we conclude that with the definition \eqref{eq: b}, $\sigma_3(r)$ and the slopes coefficients $c_i(r)$ given by \eqref{eq: Slopes} are smooth functions of $b$ on $[0,b^*]$, $0<b^*(d,\ell)\ll1$ universal small enough. We now explicitly check that the determinant $(c_1-c_4)^2+4c_2c_3$ which appears in the definition of the slopes and eigenfunctions \eqref{eq: c_pm}, \eqref{eq: lambda_pm} is non degenerate at $r_+$, see limiting values in Appendix \ref{app: Slopes and Eigenvalues} and the non degeneracy of $\l_-$, which ensures that $c_\pm,\l_\pm$ are smooth functions of $b$ all the way to $b=0$. The eye property \eqref{eq: Eye Property} thus implies $$|W_{\hskip -.1pc\peye}|+|\Sigma_{{\hskip -.1pc\peye}}|=|w_3-w_2|+|\sigma_3-\sigma_2|\lesssim Cb.$$
Since the coefficients of the matrix $P^{-1}$ given by \eqref{eq: P} do not degenerate at $r_+$ from direct check, we conclude
$$|\Wte|+|\Sigmate|\le Cb.$$
The coefficients $\dt_{ij},\et_{ij}$ of the polynomials of the RHS of \eqref{eq: Diagonalized Autonomous System} are computed from \eqref{eq: Coefficients of Autonomous System 2} and are $\mathcal O(1)$ at $r_+$. Moreover $\mathcal G_1, \mathcal G_2$  vanish at $P_3$ and this forces \eqref{eq: P_3 in tilde Variables}. Note also, that the sign of $\mu_+$ follows from \eqref{eq: Sign of Eigenvalues}.
\end{proof}

\subsection{Renormalization}

We now proceed to the renormalization of \eqref{eq: Diagonalized Autonomous System} for $0<b\ll1$.

\begin{lemma}[Renormalization and quasilinear formulation]
\label{lem: Quasilinear Formulation}
Let 
\be
\label{eq: Quasilinear Variables}
\left|\begin{array}{l}
\Wt=-b\wt\\
\Sigmat=b^2\sigmat
\end{array}\right., \ \ 
\left|\begin{array}{l}
\psit=\frac{\sigmat}{\wt}=\psit_{\hskip -.1pc\peye} \phi\\
\wt=\wt_{\hskip -.1pc\peye} u
\end{array}\right., \ \ \phi(u)=u+(1-u)\Psi(u)
\ee
then \eqref{eq: Diagonalized Autonomous System} is mapped to the quasilinear problem
\bea
\label{eq: Quasilinear Equation}
 &&\left[1+H_2+G_2\Psi+\NLt_2\right]u(1-u)\frac{d\Psi}{du}\\
\nonumber&+& \left[(1-2u)(1+H_2+G_2\Psi+\NLt_2)-\gamma(1+G_1)+2uG_2\right]\Psi\\
\nonumber& = &u\left[\gamma bH_1-2(1+H_2)+\frac{\gamma b\NLt_1}{x}-2\NLt_2\right]
\eea
where $H_1,H_2,G_1,G_2$ are explicit polynomials in $(b,u)$ given by \eqref{eq: H_i}, \eqref{eq: G_i}, and the nonlinear terms $(\NLt_{i})_{i=1,2}$ are given by \eqref{eq: NL_i}.
Moreover, 
\be
\label{eq: w, psi at P_3}
\left|\begin{array}{l}
\wt_{\hskip -.1pc\peye}=\frac{(c^\infty_+-c^\infty_-)\mu_+^\infty}{\dt_{20}^\infty}+\mathcal O(b)\\
\psit_{\hskip -.1pc\peye}=-\frac{\et^\infty_{20}\mu^\infty_+}{\dt^\infty_{20}\l^\infty_-}+\mathcal O(b).
\end{array}\right.
\ee
\end{lemma}

\begin{proof} 
The proof is verbatim the same as that of Lemma 4.3 in \cite{MRRSprofile} where it is also proved that $\mathcal G_1$, $\mathcal G_2$ appearing in Lemma \ref{lem: Diagonalization} can be expressed in terms of
\be
\label{eq: Psi}
\Phit=(1-u)\Psi,\quad u=\frac{x}{b}
\ee
as
\bea
\label{eq: G_1 reexpressed}&& \Delta_1-c_-\Delta_2=\mathcal G_2\\
\nonumber & = & -b^2|\l_-|\psite\wte(c_+-c_-)u\left[u(1-u)bH_1(b,u)+(1+G_1(bu))\Phit+\NL_1(u,\Phit)\right]
 \eea
and
\bea
\label{eq: G_2 reexpressed} &&-\Delta_1+c_+\Delta_2= \mathcal G_1\\
\nonumber& = & b^2\wte(c_+-c_-)|\mu_+|u\left[(1-u)\left[1+H_2(b,u)\right]+G_2(bu)\Phit+\NL_2(u,\Phit)\right]
\eea
where 
\be
\label{eq: Nonlinear Terms in F_1}
\left|\begin{array}{l}
H_1(b,u)=-(\Et_{11}+\Et_{30})+b(\Et_{02}+\Et_{21})(1+u)-b^2\Et_{12}(1+u+u^2)+b^3\Et_{03}(1+u+u^2+u^3)\\
G_1(x)=\Et_{11}x-(2\Et_{02}+\Et_{21})x^2+2\Et_{12}x^3-3\Et_{03}x^4\\
\NL_1(u,\Phit)= b\Phit^2\left[-\Et_{02}bu+\Et_{12}b^2u^2-3\Et_{03}b^3u^3\right]- b^2\Phit^3\left[\Et_{03}b^2u^2\right]
\end{array}\right.
\ee
and
\be
\label{eq: Nonlinear Terms in F_2}
\left|\begin{array}{l} \begin{aligned}H_2(b,u)=&b(\Dt_{11}+\Dt_{30})u-b^2(\Dt_{02}+\Dt_{21})u(1+u)+b^3\Dt_{12}u(1+u+u^2)\\
&-b^4\Dt_{03}u(1+u+u^2+u^3)\end{aligned}\\
G_2(x)=-\Dt_{11}x+(2\Dt_{02}+\Dt_{21})x^2-2\Dt_{12}x^3+3\Dt_{03}x^4\\
\NL_2(u,\Phit)=b\Phit^2\left[\Dt_{02}bu-\Dt_{12}b^2u^2+3\Dt_{03}b^3u^3\right]+b^2\Phit^3\left[\Dt_{03}b^2u^2\right]
\end{array}\right.
\ee
for $\Dt_{ij}$ and $\Et_{ij}$ defined in \eqref{eq: D and E}.

\end{proof}

\section{Bounding the Taylor series of the formal limit problem}
\label{sec: Formal Limit Problem}
We now start the analysis of the non linear ODE \eqref{eq: Quasilinear Equation} for $0<u<1$. It has a regular singular point at the origin and our first task is to estimate the growth of the Taylor coefficients of solutions' expansions at $u=0$. This will be done in two steps. First, in this section we estimate the growth of the coefficients for a formal $b=0$ limiting system, Proposition \ref{prop: Bound for Formal Limit Problem}. This will make appear the function $S_\infty(d,\ell)$. Then, in section \ref{sec: General Problem} we will obtain uniform bounds in $b$ for the Taylor coefficients associated to \eqref{eq: Quasilinear Equation} {\em for frequencies $k\lesssim \frac{1}{b}$}.

\subsection{Formal limit $b=0$} Recall \eqref{eq: Quasilinear Equation} and let $$\Psi(u)=\Psit(x), \quad x=bu$$ then 
\bea
\label{eq: Quasilinear Equation in x}
&& \left[1+H_2+G_2\Psit+\NLt_2\right]x(b-x)\Psit'\\
\nonumber&+& \left[(b-2x)(1+H_2+G_2\Psit+\NLt_2)-b\gamma(1+G_1)+2xG_2\right]\Psit\\
\nonumber& = &x\left[\gamma bH_1-2(1+H_2)+\frac{\gamma b\NLt_1}{x}-2\NLt_2\right].
\eea
We introduce the parameters
\be
\label{eq: a and nu}
\left|\begin{array}{l}
a=\gamma b=\frac{|\l_-|}{|\mu_+|}>0\\
\nu=-\gamma b(\Dt_{11}+\Dt_{30}-\Et_{11}),
\end{array}\right.
\ee
which have a well defined limit as $b\to 0$ noted $a_\infty,\nu_\infty$,
and assume 
\be
\label{eq: Positivity of nu_infty}
\nu_\infty(d,\ell)>0.
\ee
In view of the explicit formulas of Appendix \ref{app: Slopes and Eigenvalues}, the {\em formal} limit $b\to 0$ is:
\bea
\label{eq: Formal Limit Equation}
&&\left[1+H^\infty_{20}+G^\infty_2\Psit+\NLt^\infty_{20}\right](-x^2)\Psit'\\
\nonumber&+& \left[-2x(1+H^\infty_{20}+G^\infty_2\Psit+\NLt^\infty_{20})-a_\infty(1+G^\infty_1)+2xG^\infty_2\right]\Psit\\
\nonumber& = &x\left[a_\infty H^\infty_{10}-2(1+H^\infty_{20})+\frac{a_\infty\NLt^\infty_{10}}{x}-2\NLt^\infty_{20}\right]
\eea
where we recall that the superscript $^\infty$ means that we compute all parameters $(\Dt_{ij},\Et_{ij})$ given by \eqref{eq: D and E} in their well defined limit $b=0$. 
Let us stress the fact that this is {\em not a dynamical limit}, since the change variables $x=bu$ maps the 
original flow on the set $x=0$. Our claim is that for fixed order $k$, the  Taylor coefficients associated to \eqref{eq: Formal Limit Equation} are the $b=0$ limit of the Taylor coefficients associated to \eqref{eq: Quasilinear Equation} up to a suitable renormalization, see \eqref{eq: psi_k Limit}.

\subsection{Boundedness of the limiting series}
The condition \eqref{eq: Positivity of nu_infty} holds, by a direct examination, at $r^+$ and $\ell<d$ from \eqref{eq: nu_infty at r^+}.\\

Our aim in this section is to prove the following bound.

\begin{proposition}[Boundedness for \eqref{eq: Formal Limit Equation}]
\label{prop: Bound for Formal Limit Problem} 
Assume \eqref{eq: Positivity of nu_infty}. Then there exists $c_{\nu_\infty,a_\infty}>0$ such that the following holds. Let $\Psit$ be the unique $C^\infty$ local solution to \eqref{eq: Formal Limit Equation} on $[0,x_0]$, then the sequence $$\psit_k=\frac{\Psit^{(k)}(0)}{k!}$$ satisfies 
\be
\label{eq: psit Bound}
|\psit_k|\leq c_{\nu,a}\frac{\Gamma(k+\nu_\infty+2)}{a_\infty^k}.
\ee
\end{proposition}

\subsection{Conjugation formula}
We start by conjugating \eqref{eq: Formal Limit Equation} to an explicitly solvable (at the linear level) problem.

\begin{lemma}[Conjugation]
\label{lem: Conjugation}
There exist functions $\xi(x), (\mu_j(x),\nu_j(x))_{1\le j\le 4}$, holomorphic in a neighborhood of the $x=0$,
dependent on $(d,\ell,r)$,  such that the change of variables
\be
\label{eq: Conjugation Variables}
\left|\begin{array}{l}
\Psit(x)=M(x)\Phi(x), \ \ M(x)=e^{-\int_0^x \xi(y)dy}\\
\Theta(x)=\frac{\Phi(x)}{x}
\end{array}\right.
\ee 
maps \eqref{eq: Formal Limit Equation} to 
\be
\label{eq: Conjugated Limit Problem}
x^2\Theta'+[a_\infty+(\nu_\infty+3)x]\Theta=\mu_0+x\left[x\sum_{j=1}^4\mu_j\Theta^j+x^2(x\Theta')\sum_{j=1}^{4}\nu_j\Theta^j\right].
\ee
\end{lemma}

\begin{proof}[Proof of Lemma \ref{lem: Conjugation}] This is an explicit computation.\\

\noindent{\bf step 1} Linear conjugation. We solve the linear problem
\be
\label{eq: Linear Problem}
\left[x+xH^\infty_{20}\right](-x)\Psit'+ \left[-2x(1+H^\infty_{20})-a_\infty(1+G^\infty_1)+2xG^\infty_2\right]\Psit =xF.
\ee
This is equivalent to
$$
\Psit'+\left[\frac{2}{x}+\frac{a_\infty(1+\Et^\infty_{11}x)}{x^2(1+H^\infty_{20})}+\frac{a_\infty(G^\infty_1-\Et^\infty_{11}{x})-2xG^\infty_2}{x^2(1+H^\infty_{20})}\right]\Psit=-\frac{F}{x(1+H^\infty_{20})}.
$$
We have
\bee
&&\frac{a_\infty(1+\Et^\infty_{11}x)}{x^2(1+H^\infty_{20})}=\frac{a_\infty}{x^2}(1+\Et^\infty_{11}x)\left[1-(\Dt^\infty_{11}+\Dt^\infty_{30})x+\frac{1}{1+H^\infty_{20}}-1+(\Dt^\infty_{11}+\Dt^\infty_{30}){x}\right]\\
& = & \frac{a_\infty}{x^2}(1+\Et^\infty_{11}x)\left[\frac{1}{1+H^\infty_{20}}-1+(\Dt^\infty_{11}+\Dt^\infty_{30}){x}\right]-a_\infty\Et^\infty_{11}(\Dt^\infty_{11}+\Dt^\infty_{30})\\
& + & \frac{a_\infty}{x^2}+\frac{a_\infty[\Et^\infty_{11}-(\Dt^\infty_{11}+\Dt^\infty_{30}{)}]}{x}
\eee
Let 
\bee
\xi(x)&=&\frac{a_\infty}{x^2}(1+\Et^\infty_{11}x)\left[\frac{1}{1+H^\infty_{20}}-1+(\Dt^\infty_{11}+\Dt^\infty_{30}){x}\right]-a_\infty\Et^\infty_{11}(\Dt^\infty_{11}+\Dt^\infty_{30})\\
&+ & \frac{a_\infty(G^\infty_1-\Et^\infty_{11}{x})-2xG^\infty_2}{x^2(1+H^\infty_{20})}
\eee
Observe that
$H^\infty_{20}-(\Dt^\infty_{11}+\Dt^\infty_{30}){x}$ and $G^\infty_1-\Et^\infty_{11}{x}$ are polynomials in $x$  starting with $x^2$, while $G^\infty_2$ 
is a polynomial beginning with $x$. The function $\xi(x)$ is then holomorphic in a neighborhood of $x=0$. We have
$$
\eqref{eq: Linear Problem}\Leftrightarrow \Psit'+\left[\frac{a_\infty}{x^2}+\frac{\nu_\infty+2}{x}+\xi(x)\right]\Psit=-\frac{F}{x(1+H^\infty_{20})}.
$$
Define
$$
M(x)=e^{-\int_0^x \xi(y)dy}
$$
a holomorphic function in a neighborhood of $x=0$, and introduce the change of variables $$\Psit(x)=M(x)\Phi(x), \ \ G(x)=-\frac{F(x)}{M(x)(1+H^\infty_{20}(x))}.$$
We have obtained the conjugation formula:
\be
\label{eq: Conjugated Linear Problem}
\eqref{eq: Linear Problem}\Leftrightarrow \Phi'+\left[\frac{a_\infty}{x^2}+\frac{\nu_\infty+2}{x}\right]\Phi=\frac{G}{x}.
\ee

\noindent{\bf step 2} Conjugation for the nonlinear problem. The equation \eqref{eq: Formal Limit Equation} is in the form \eqref{eq: Linear Problem}
$$\left[x+xH^\infty_{20}\right](-x)\Psit'+ \left[-2x(1+H^\infty_{20})-a_\infty(1+G^\infty_1)+2xG^\infty_2\right]\Psit =x(F_0+F)$$ 
for the source term
$$F_0=a_\infty H^\infty_{10}-2(1+H^\infty_{20})$$ and the nonlinear term
\bee
F& = & \frac{a_\infty\NLt^\infty_{10}}{x}-2\NLt^\infty_{20}+\left[G^\infty_2\Psi+\NLt^\infty_{20}\right](2\Psi+x\Psi')
\eee
We conjugate this using \eqref{eq: Conjugated Linear Problem}:
$$
\left|\begin{array}{l}
\Psit(x)=M(x)\Phi(x), \ \ G(x)=-\frac{F(x)+F_0}{M(x)(1+H^\infty_{20}(x))} \\
x^2\Phi'+\left[a_\infty+(\nu_\infty+2)x\right]\Phi= xG
\end{array}\right.
$$
We express the nonlinearity in terms of $\Phi$ and obtain a representation:
$$
G=\xi_0+{\sum_{j=2}^4\xi_j\Phi^j+(x\Phi')\sum_{j=1}^3(\xit_j\Phi^j)}
$$
where $(\xi_j,\xit_j)$ are explicit holomorphic functions in a neighborhood of the origin. We have obtained the equivalent nonlinear problem:
\be
\label{eq: Conjugated Nonlinear Problem}
x^2\Phi'+\left[a_\infty+(\nu_\infty+2)x\right]\Phi=x\left[\xi_0+{\sum_{j=2}^4\xi_j\Phi^j+(x\Phi')\sum_{j=1}^3(\xit_j\Phi^j)}\right].
\ee
A direct computation shows $\Phi(0)=0$ so for $\Theta=\frac{\Phi}{x}$,
\bee
\nonumber \eqref{eq: Conjugated Nonlinear Problem} &\Leftrightarrow &x^2\Theta'+[a_\infty+(\nu_\infty+3)x]\Theta=\xi_0+{\sum_{j=2}^4\xi_jx^j\Theta^j+x(x\Theta')\sum_{j=1}^3(\xit_jx^j\Theta^j)+\sum_{j=1}^3\xit_j(x\Theta)^{j+1}}\\
&\Leftrightarrow &x^2\Theta'+[a_\infty+(\nu_\infty+3)x]\Theta=\mu_0+x\left[x\sum_{j=1}^4\mu_j\Theta^j+x^2(x\Theta')\sum_{j=1}^{4}\nu_j\Theta^j\right]
\eee
where $\mu_{{j}},\nu_j$ are holomorphic functions of $x$ in a neighbourhood of $0$, and \eqref{eq: Conjugated Limit Problem} is proved.
\end{proof}

\subsection{The nonlinear induction relation}

The uniqueness of a local $C^\infty$ solution to \eqref{eq: Conjugated Limit Problem}, and thus \eqref{eq: Formal Limit Equation}, follows from an elementary fixed point argument which is left to the reader. We let $$\mu_{jk}=\frac{\mu_j^{(k)}(0)}{k!}, \ \ \nu_{jk}=\frac{\nu_j^{(k)}(0)}{k!},\ \ \theta_k=\frac{\Theta^{(k)}(0)}{k!}, \ \ \phi_k=\frac{\Phi^{(k)}(0)}{k!}$$ and claim the following fundamental nonlinear bound.

\begin{lemma}[Bound on the limiting sequence]
For some large enough universal constant $c(\nu_\infty,a_\infty)>0$:
\be
\label{eq: phi_k Bound}
|\phi_k|\leq c(\nu_\infty,a_\infty)\frac{\Gamma(k+\nu_\infty+2)}{a_\infty^k}, \ \ \forall k\ge 1.
\ee
Also, we have
\be
\label{eq: Nonlinear Bound on Limiting Sequence}
\sum_{k=0}^{+\infty}\left|\frac{a_\infty^kg_k}{\Gamma(\nu_\infty+k+3)}\right|\leq c(\nu_\infty,a_\infty). 
\ee
where $g_k$ is defined as
\be
\label{eq: Expansion of Theta and G}
\left|\begin{array}{l}
\Theta=\sum_{k=0}^{+\infty}\theta_kx^k\\
 G=\mu_0+x\left[x\sum_{j=1}^4\mu_j\Theta^j+x^2(x\Theta')\sum_{j=1}^4\nu_j\Theta^j\right]=\sum_{k=0}^{+\infty}g_kx^k.
 \end{array}\right.
 \ee
\end{lemma}

\begin{proof}
The proof is verbatim the same as that of Lemma 5.4 in \cite{MRRSprofile}.
\end{proof}

\subsection{Proof of Proposition \ref{prop: Bound for Formal Limit Problem}}

In view of \eqref{eq: phi_k Bound} and the conjugation formula \eqref{eq: Conjugation Variables}, the bound \eqref{eq: psit Bound} directly follows from the following continuity lemma.

\begin{lemma}[Continuity]
\label{cneinveineonevni}
Let $M(x)$ be holomorphic in a neighborhood of $0$, then there exists $c_M>0$ such that for all functions $\Phi$ which are $C^\infty$ at the origin, 
\be
\label{vnoenvonveneovn}
\sup_{k\ge 0}\frac{a_\infty^k|(M\Phi)_k|}{\Gamma(k+\nu_\infty+2)}\leq c_M\sup_{k\ge 0}\frac{a_\infty^k|\phi_k|}{\Gamma(k+\nu_\infty+2)}.
\ee
\end{lemma}
\begin{proof}
The proof is verbatim the same as that of Lemma 5.5 in \cite{MRRSprofile}.
\end{proof}

\section{Bounding the Taylor series for \eqref{eq: Quasilinear Equation}}
\label{sec: General Problem}

We now start the study of the full problem \eqref{eq: Quasilinear Equation}. We first aim at obtaining uniform in $b$ bounds on the coefficients of the Taylor series as well as the convergence to the limiting problem as $b\to 0$. We recall the notation $$\gamma-1=K+\alpha_\gamma, \ \ K\in \Bbb N^*, \ \ 0<\alpha_\gamma<1.$$

\subsection{$b$ dependent conjugation}

We conjugate the $b$ dependent problem \eqref{eq: Quasilinear Equation} to an explicitly solvable (at the linear level) $b$-dependent equation.

\begin{lemma}[Linear conjugation]
\label{vbjbvbvebi} Let $\xi(x)$ be as in \eqref{eq: xi}. There exist a function $\xit_b(x)$ which is holomorphic in a $b$-independent neighborhood of $x=0$ such that the conjugation
\be
\label{defphi}
\Psi(u)=M_b(x)\Phi(u),\ \ G=-\frac{F}{M_b(1+H_{20})}
\ee
with 
\be
\label{deflkenrnal}
M_b(x)=e^{\xit_b(x)}
\ee
and 
\be
\label{defnub}
\left|\begin{array}{l}
\nu_b=\nu+\xi(b)\\
\nu=-\gamma b(\Dt_{11}+\Dt_{30}-\Et_{11})
\end{array}\right.
\ee
maps 
\be
\label{odetoinvertbis}\left[1+H_{20}\right]u(1-u)\Psi'+\left[(1-2u)(1+H_{20})-\gamma(1+G_1)+2uG_2\right]\Psi=uF
\ee
to
\be
\label{conjuguatedflrow}
\Phi'-\left[\frac{\gamma-1}{u}+\frac{\gamma+\nu_b+1}{1-u}\right]\Phi=-\frac{G}{1-u}.
\ee
\end{lemma}

\begin{remark} We see immediately the fundamental difference between \eqref{conjuguatedflrow} and \eqref{eq: Conjugated Linear Problem}: for the $b$-dependent problem, the point $u=0$ is a regular singular point, while it is a singular singular point for $b=0$. 
As a result, we will see a change in the behavior of the Taylor series for large frequencies, which will be reflected in the nature  of the weight $w_{\gamma,\nu}$, see Lemma \ref{lemmanvnowvowv}.
\end{remark}

\begin{proof}[Proof of Lemma \ref{vbjbvbvebi}] This is a direct computation.\\

\noindent{\bf step 1} Linear conjugation. We rewrite \eqref{odetoinvertbis}
\be
\label{vennoenoneo}
\Psi'-\frac{\zeta'_b}{\zeta_b}\Psi=\frac{ F}{(1-u)(1+H_{2,0})}
\ee 
with
\bee
&&\frac{\zeta'_b}{\zeta_b}=\frac{\gamma(1+G_1)}{u(1-u)(1+H_{2,0})}-\frac{2G_2}{(1-u)(1+H_{2,0})}-\frac{1-2u}{u(1-u)}\\
& = & \frac{\gamma(1+G_1)}{u(1-u)(1+H_{2,0})}-\frac{2G_2}{(1-u)(1+H_{2,0})}-\frac{1}{u}+\frac{1}{1-u}
\eee
From \eqref{eq: Nonlinear Terms in F_1}:
$$
\left|\begin{array}{l}
G_1(x)=\Et_{11}x+x\tilde{G}_1\\
 \tilde{G}_1=-(2\Et_{02}+\Et_{21})x+2\Et_{12}x^2-3\Et_{03}x^3
 \end{array}\right.
$$
and 
$$\frac{\gamma(1+G_1)}{u(1-u)}=\frac{\gamma(1+\Et_{11}bu+bu\tilde{G}_1)}{u(1-u)}=\frac{\gamma}{u(1-u)}+\frac{\Et_{11}b\gamma}{1-u}+\frac{\gamma b \tilde{G}_1}{1-u}$$ yields
$$\frac{\zeta'_b}{\zeta_b}=\left[\frac{\gamma}{u(1-u)}+\frac{\Et_{11}b\gamma}{1-u}\right]\frac{1}{1+H_{2,0}}+\frac{\gamma b \tilde{G}_1-2G_2}{(1-u)(1+H_{2,0})}-\frac{1}{u}+\frac{1}{1-u}
.$$
We recall that $x=bu$ and rewrite
\bee
&&\frac{\gamma}{u(1-u)(1+H_{2,0})}=\frac{\gamma}{u(1-u)}\left[1-(\Dt_{11}+\Dt_{30})bu+\left(\frac{1}{1+H_{2,0}}-1+(\Dt_{11}+\Dt_{30})x\right)\right]\\
& =& \gamma\left(\frac{1}{u}+\frac{1}{1-u}\right)-\frac{\gamma b(\Dt_{11}+\Dt_{30})}{1-u}+ \frac{\gamma b}{1-u}\left[\frac{1}{x}\left(\frac{1}{1+H_{2,0}}-1+(\Dt_{11}+\Dt_{30})x\right)\right]
\eee
and 
\bee
\frac{\Et_{11}b\gamma}{(1-u)(1+H_{2,0})}=\frac{\Et_{11}b\gamma}{1-u}+\frac{\Et_{11}b\gamma}{1-u}\left[\frac{1}{1+H_{2,0}}-1\right].
\eee
We have therefore obtained the formula:
$$
\frac{\zeta'_b}{\zeta_b}= \gamma\left(\frac{1}{u}+\frac{1}{1-u}\right)-\frac{1}{u}+\frac{1}{1-u}
-\frac{\gamma b(\Dt_{11}+\Dt_{30})}{1-u}+\frac{\Et_{11}b\gamma}{1-u}+\frac{\xi(x)}{1-u}
$$
with
\be
\label{eq: xi}
\xi(x)=\frac{\gamma b \tilde{G}_1-2G_2}{{1+H_{2,0}}}+   \frac{\gamma b}{x}\left(\frac{1}{1+H_{2,0}}-1+(\Dt_{11}+\Dt_{30})x\right)+\gamma b\Et_{11}\left(\frac{1}{1+H_{2,0}}-1\right).
\ee
Then, from  \eqref{eq: a and nu}, we rewrite:
\be
\label{expressionzetab}
\frac{\zeta'_b}{\zeta_b}= \frac{\gamma-1}{u}+\frac{\gamma+\nu+1}{1-u}+\frac{\xi(x)}{1-u}.
\ee

\noindent{\bf step 2} Computation of the kernel. We now use the analyticity of $\xi$ in $|x|<\frac{1}{C}$ and $\xi(0)=0$ to compute:
\bee
&&\frac{\xi(x)}{1-u}=\frac{1}{1-u}\sum_{k=1}^{+\infty} \xi_kx^k=\frac{1}{1-u}\sum_{k=1}^{+\infty} \xi_kb^ku^k=  \frac{\sum_{k=1}^{+\infty} \xi_kb^k}{1-u}-\sum_{k=1}^{+\infty} \xi_kb^k\frac{1-u^k}{1-u}\\
& = & \frac{\xi(b)}{1-u}-\sum_{k=1}^{+\infty} \xi_kb^k\sum_{j=0}^{k-1}u^j.
\eee
Let $\nu_b$ be given by \eqref{defnub}, we have obtained:
$$\frac{\zeta'_b}{\zeta_b}=\frac{\gamma-1}{u}+\frac{\gamma+\nu_b+1}{1-u}-\sum_{k=1}^{+\infty} \xi_kb^k\sum_{j=0}^{k-1}u^j.$$
We compute a primitive of the remaining term:
\bee
\sum_{k=1}^{+\infty}\xi_kb^k\sum_{j=0}^{k-1}\frac{u^{j+1}}{j+1}=\sum_{j=0}^{+\infty}\frac{u^{j+1}}{j+1}\sum_{k=j+1}^{+\infty}\xi_kb^k=\sum_{j=1}^{+\infty}\frac{u^{j}}{j}\sum_{k=j}^{+\infty}\xi_kb^k=\sum_{j=1}^{+\infty}\tilde{\xi}_{b,j}b^ju^j=\sum_{j=1}^{+\infty}\tilde{\xi}_{b,j}x^j
\eee
with $$\tilde{\xi}_{b,j}=\frac{1}{j}\sum_{k=j}^{+\infty}\xi_kb^{k-j}=\frac{1}{j}\sum_{k=0}^{+\infty}\xi_{k+j}b^k.$$ The holomorphic bound $|\xi_j|\le C^j$ ensures
\be
\label{boundxit}
|\xit_{b,j}|\leq\frac1j\sum_{k=0}^{+\infty}b^kC^{k+j}\leq \frac{C^j}{j}\frac{1}{1-bC}\leq\frac{C^j}{j}
\ee
for $0<b<b^*$ universal small enough. We have therefore obtained the formula
\be
\label{reformulationkernela}
\frac{\zeta'_b}{\zeta_b}=\frac{\gamma-1}{u}+\frac{\gamma+\nu_b+1}{1-u}+\frac{d}{du}\xit_b(x)
\ee
where 
\be
\label{vnovoeonven}
\xit_b(x)={-}\sum_{j=1}^{+\infty}\tilde{\xi}_{b,j}x^j, \ \ |\xit_{b,j}|\le C^j
\ee
 is holomorphic in a neighborhood of $x=0$ independent of $b$.\\

\noindent{\bf step 3} Conclusion. From \eqref{vennoenoneo}, \eqref{reformulationkernela}:
$$\Psi'-\left[\frac{\gamma-1}{u}+\frac{\gamma+\nu_b+1}{1-u}+{\frac{d}{du}}\xit_b(x)\right]\Psi=\frac{ F}{(1-u)(1+H_{20})}$$ 
and \eqref{conjuguatedflrow} follows.
\end{proof}

\subsection{The $b$-dependent discrete weight}

We study the discrete weight associated to \eqref{conjuguatedflrow}.

\begin{lemma}[Properties of the weight]
\label{lemmanvnowvowv}
Let 
\be
\label{defweight}
w_{\gamma,\nu}(k)=\frac{\Gamma(\gamma-1-k)\Gamma(\nu+k+2)}{\Gamma(\gamma-1)},\ \ k\in \Bbb N.
\ee
then:\\ 
\begin{enumerate}
\item  value for $k\ge K$: $\forall j\ge 0$, 
\be
\label{cnineneonevnev:0}
w_{\gamma,\nu_b}(K+j)=(-1)^j\Gamma(\alpha_\gamma)\Gamma(1-\alpha_\gamma)\frac{\Gamma(K+j+\nu_b+2)}{\Gamma(K+\alpha_\gamma)\Gamma(j+1-\alpha_\gamma)}.
\ee
\item $\nu$ dependence:
\be
\label{cneioneineonoen}
\forall  k\ge 0, \ \ \frac{\gamma w_{\gamma,\nu_b}(k+1)}{w_{\gamma-1,\nu_b+1}(k)}=\frac{\gamma}{\gamma-2}
\ee
\item induction property:
\be
\label{inucnoitnwrmwjaggmamk}
\forall k\ge 1, \ \ (\gamma-k-2)w_{\gamma-1,\nu_b+1}(k)-(k+\nu_b+2)w_{\gamma-1,\nu_b+1}(k-1)=0.
\ee
\end{enumerate}
\end{lemma}

\begin{proof} {\eqref{cnineneonevnev:0}} directly follows from \eqref{defweight}, \eqref{formulafdebasenegatif}.
\\
We then compute for {$k\ge 0$}:
\bee
\frac{\gamma w_{\gamma,\nu_b}(k+1)}{w_{\gamma-1,\nu_b+1}(k)}=\gamma\frac{\Gamma(\gamma-1-(k+1))\Gamma(k+1+\nu_b+2)\Gamma(\gamma-2)}{\Gamma(\gamma-1)\Gamma(\gamma-2-k)\Gamma(k+\nu_b+3)}{=\frac{\gamma \Gamma(\gamma-2)}{\Gamma(\gamma-1)}}=\frac{\gamma}{\gamma-2}
\eee
and \eqref{cneioneineonoen} is proved. We now turn to the proof of the induction formula \eqref{inucnoitnwrmwjaggmamk}: for {$k\ge 1$}, 
\bee
&&(\gamma-k-2)w_{\gamma-1,\nu_b+1}(k)-(k+\nu_b+2)w_{\gamma-1,\nu_b+1}(k-1)\\
& = & (\gamma-k-2)\frac{\Gamma(\gamma-2-k)\Gamma(k+\nu_b+1+2)}{\Gamma(\gamma-2)}\\
&-&(k+\nu_b+2)\frac{\Gamma(\gamma-2-(k-1))\Gamma(k-1+\nu_b+1+2)}{\Gamma(\gamma-2)}=0.
\eee
\end{proof}

\subsection{Boundedness of the sequence}

We claim the following $b$-dependent nonlinear bound which, in a certain sense,  is a deformation of
 \eqref{eq: psit Bound}.

\begin{proposition}[$b$-dependent boundedness]
\label{propboundbdependent}
There exist universal constants $c_{\nu,a}>0$ and $0<b^*\ll 1$ such that the following holds for all $0<b<b^*$.
Let $\Psi$ be a solutions of  \eqref{eq: Quasilinear Equation} and define $$\psi_k=\frac{1}{k!}\frac{d^k\Psi}{du^k}(0),$$ then 
\be
\label{bounduniformb}
\forall \,0\le k\le K, \ \ |\psi_k|\leq c_{\nu,a}w_{\gamma,\nu_b}(k).
\ee
Moreover, let $\Psit^\infty$ be the unique local $C^\infty$ solution to the limiting problem \eqref{eq: Formal Limit Equation} and 
$$\psit_k^\infty=\frac{1}{k!}\frac{d^k{\Psit}^\infty}{dx^k}(0),$$ 
then 
\be
\label{eq: psi_k Limit}
\forall k\ge 0, \ \ \lim_{b\to 0 }\frac{\psi_k}{b^k}=\psit_k^\infty,\quad 
\lim_{b\to 0}\frac{g_k}{b^k} = g_k^\infty
\ee
where $g_k$ is the holomorphic expansion \eqref{eq: Holomorphic expansion G} of the nonlinearity $\mathcal G$ as in \eqref{vneivnoneoneonnee} and $g_k^\infty$ from \eqref{eq: Expansion of Theta and G} is the corresponding values from the limiting problem.
\end{proposition}

The rest of this section is devoted to the proof of Proposition \ref{propboundbdependent}.

\subsection{Nonlinear conjugation}

The proof of \eqref{bounduniformb} now requires a careful track of the $b$-dependencies  in the full problem  \eqref{eq: Quasilinear Equation}. The first step is to use Lemma \ref{vbjbvbvebi} and analyze the nonlinear conjugated problem.\\

\noindent{\bf step 1} Nonlinear conjugation. Recall \eqref{eq: Quasilinear Equation}:
$$
\label{nenvneonoenvi}
 \left[1+H_{20}\right]u(1-u)\Psi'+\left[(1-2u)(1+H_{20})-\gamma(1+G_1)+2uG_2\right]\Psi= u\mathcal F
$$
with recalling $\frac{1}{u}=\frac{b}{x}$: 
\bea
\label{defmahtclf}
&& \mathcal F= \gamma bH_1-2(1+H_2)\\
\nonumber &+& \frac{\gamma b\NLt_1}{x}-2\NLt_2-(1-u)\left[\sum_{j=1}^3b^jH_{2j}(x)+G_2\Psi+\NLt_2\right]\Psi'\\
\nonumber& + &\left\{-\frac{b}{x}(G_2\Psi+\NLt_2)+2(G_2\Psi+\NLt_2)-\frac{b(1-2u)}{x}\sum_{j=1}^3b^jH_{2j}(x)\right\}\Psi
\eea
From \eqref{defphi}, \eqref{conjuguatedflrow} we now obtain the nonlinear conjugated problem:
\be
\label{cneineionoeno}
\Phi'-\left[\frac{\gamma-1}{u}+\frac{\gamma+\nu_b+1}{1-u}\right]\Phi=-\frac{\mathcal G}{1-u}
\ee

\noindent{\bf step 2} Computation of $\mathcal G$. We plug \eqref{defphi} into \eqref{defmahtclf} and decompose 
\be
\label{cneineinveoneonnkenpe}
\mathcal G=\mathcal G_0+\mathcal L(\Phi)+\NL(\Phi)
\ee 
as follows.\\

\noindent\underline{Source term}. We have 
$$\mathcal G_0(b,x){=-\frac{\gamma b H_1-2(1+H_2)}{M_b(1+H_{20})}}$$ 
which, from \eqref{vnovoeonven}, admits  a holomorphic expansion in a neighborhood (independent of $b$) of $x=0$  i.e.
\be
\label{neinineonenoenv:avoidmultipledefinedlabel}
\mathcal G_{{0}}(x)=\sum_{k=0}^{+\infty}g_{0k}x^k
\ee 
for some $b$-dependent coefficients $g_{0k}$ satisfying  
\be
\label{neinineonenoenv}
|g_{0k}|\leq C^k
\ee
for some $C>0$ independent of $b$.

\noindent\underline{Small linear term}. It is given explicitly by
\bea
\label{formulalphi}
\nonumber \mathcal L(\Phi)&=&\frac{1}{M_b(1+H_{20})}\left\{(1-u)\left[\sum_{j=1}^3b^jH_{2j}(x)\right]\frac{d}{du}(M_b(x)\Phi(u))+\frac{b(1-2u)}{x}\sum_{j=1}^3b^jH_{2j}(x)M_b\Phi\right\}\\
\nonumber& = & \frac{1}{M_b(1+H_{20})}\left\{\left[M_{{b}}'(x)(b-x)\sum_{j=1}^3b^jH_{2j}(x)+\sum_{j=1}^3\left(\frac{b^{j+1}H_{2j}(x)}{x}-2b^jH_{2j}\right)\right]\Phi\right.\\
&+& \left.M_b(x)\sum_{j=1}^3\left(b^{j+1}\frac{H_{2j}(x)}{x}-b^jH_{2j}(x)\right)u\Phi'\right\}
\eea
Therefore,
$$\mathcal L(\Phi)=b\left[(bh_1(x)+xh_2(x))\Phi+(bh_3{(x)}+xh_4{(x)})u\Phi'\right]$$ where $$h_j(x)=\sum_{k=0}^{+\infty}h_{jk}x^k, \ \ |h_{jk}|\leq C^k$$ for some $C>0$ independent of $b$. 

\noindent\underline{Nonlinear term}. We have 
\bea
\label{formaulnonlienaterm}
\nonumber\NL(\Phi) &=& -\frac{1}{M_b(1+H_{20})}\bigg\{\frac{\gamma b\NLt_1}{x}-2\NLt_2-\left(\frac{b}{x}-1\right)\left[G_2\Psi+\NLt_2\right]u\Psi'\\
&+&(G_2\Psi+\NLt_2)\left(2-\frac bx\right)\Psi\bigg\}
\eea
and thus
 $\NL(\Phi)$ is given, structurally, by $$\NL(\Phi)=x\sum_{j=2}^4m^{(1)}_j\Phi^j+b\sum_{j=2}^4m^{(2)}_j\Phi^j+\left[x\sum_{j=1}^3m^{(3)}_j\Phi^j+b\sum_{j=1}^3m^{(4)}_j\Phi^j\right]u\Phi'$$
with $$m_j^{(\ell)}(x)=\sum_{k=0}^{+\infty}m_{jk}^{(\ell)}x^k, \ \ |m_{jk}^{(\ell)}|\le C^k.$$

\noindent\underline{Conclusion}. We have obtained the conjugated nonlinear problem \eqref{cneineionoeno} with
\bea
\label{vneivnoneoneonnee}
 \mathcal G& = & \mathcal G_0+b\left[(bh_1(x)+xh_2(x))\Phi+(bh_3{(x)}+xh_4{(x)})u\Phi'\right]\\
\nonumber&+& x\sum_{j=2}^4m^{(1)}_j\Phi^j+b\sum_{j=2}^4m^{(2)}_j\Phi^j+ \left[x\sum_{j=1}^3m^{(3)}_j\Phi^j+b\sum_{j=1}^3m^{(4)}_j\Phi^j\right]u\Phi'.
\eea

\noindent{\bf step 3} Final change of variables. We now let 
\be
\label{suihfoenioeneoi}
\Phi=bu\Theta=x\Theta
\ee
 so that \eqref{cneineionoeno} becomes:
 \bea
 \label{thetaeqaiotjoihs}
 \nonumber &&\Phi'-\left[\frac{\gamma-1}{u}+\frac{\gamma+\nu_b+1}{1-u}\right]\Phi=-\frac{\mathcal G}{1-u}\\
  &\Leftrightarrow& u(1-u)\Theta' -\left[\gamma-2+(\nu_b+3)u\right]\Theta=-\frac{\mathcal G}{b}.
 \eea
 We now express $\mathcal G$ in terms of $\Theta$ and track the orders of vanishing in $x$. We compute
 $$u\Phi'=u[bu\Theta'+b\Theta]=x(u\Theta'+\Theta).$$ 
 Then,
 \bee
 && b\left[(bh_1(x)+xh_2(x))\Phi+(bh_3{(x)}+xh_4{(x)})u\Phi'\right]\\
 &=&\left[b^2h_1+bxh_2\right]x\Theta+(b^2h_3+bxh_4)[ux\Theta'+x\Theta]\\
  & = & \left[b^2x\th_1+bx^2\th_2\right]\Theta+\left[b^2x\th_3+bx^2\th_4\right]u\Theta'.
   \eee
 
 For the nonlinear term:
 \bee
 && x\sum_{j=2}^4m^{(1)}_j\Phi^j+b\sum_{j=2}^4m^{(2)}_j\Phi^j+ \left[x\sum_{j=1}^3m^{(3)}_j\Phi^j+b\sum_{j=1}^3m^{(4)}_j\Phi^j\right]u\Phi'\\
 & = & x\sum_{j=2}^4x^jm^{(1)}_j\Theta^j+b\sum_{j=2}^4m^{(2)}_jx^j\Theta^j+\left[x\sum_{j=1}^3m^{(3)}_jx^j\Theta^j+b\sum_{j=1}^3m^{(4)}_jx^j\Theta^j\right](xu\Theta'+x\Theta)\\
 & = & \sum_{j={2}}^{{4}}x^{j+{1}}\mt^{(1)}_j\Theta^{{j}}+b\sum_{j=2}^4m^{(2)}_jx^j\Theta^j+  \left[\sum_{j=1}^3\mt^{(3)}_jx^{j+2}\Theta^j+b\sum_{j={1}}^{{3}}\mt^{(4)}_jx^{j{+1}}\Theta^j\right](u\Theta').
 \eee
 We now rewrite
 \bea
 \label{fialfrmaulg}
 &&\mathcal G= \mathcal G_0+\left[b^2x\th_1+bx^2\th_2\right]\Theta+\left[b^2x\th_3+bx^2\th_4\right]u\Theta'\\
 \nonumber & + &  \sum_{j=2}^4x^{j+1}\mt^{(1)}_j\Theta^j+b\sum_{j=2}^4m^{(2)}_jx^j\Theta^j+  \left[\sum_{j=1}^3\mt^{(3)}_jx^{j+2}\Theta^j+b\sum_{j={1}}^{{3}}\mt^{(4)}_jx^{j{+1}}\Theta^j\right](u\Theta')
 \eea
 where, for some large enough universal constant $C=C_{\nu,a}>0$ independent of $b<b^*$ and all $k\ge 0$,
 \be
 \label{nveioneionoenv|}
{|(\mathcal{G}_0)_k|+|(\th_l)_k|+|(\mt^{(l)}_j)_k|}\le C^k b^k
 \ee

\subsection{Bounding the sequence $\theta_k$ and proof of Proposition \ref{propboundbdependent}}

We let 
\be
\label{eq: Holomorphic expansion G}
\theta_k=\frac{1}{k!}\frac{d^k\Theta}{du^k}(0), \ \ g_k=\frac{1}{k!}\frac{d^k\mathcal G}{du^k}(0)
\ee
so that from \eqref{suihfoenioeneoi}: 
\be
\label{cneovnenvonveoinven}
\left|\begin{array}{l}\phi_0=0\\
\phi_k=b\theta_{k-1}, \ \ k\ge 1.
\end{array}\right.
\ee

\begin{lemma}[Boundedness for the $\theta_k$ sequence]
\label{lemmaboundthetak}
There exists $c_{\nu,a}>0$ and $b^*(\nu,a)$ such that for all $0<b<b^*$, for all $0\le k\le K-1$,
\be
\label{esitmaitmitot}
|\theta_k|\leq c_{\nu,a}|w_{\gamma-1,\nu_b+1}(k)|
\ee
(This implies \eqref{bounduniformb}.) Moreover, 
\be
\label{boundednessgk}
\forall 0\le k\le K, \ \ |g_k|\le c_{\nu,a}\frac{|w_{\gamma,\nu_ b}(k)|}{1+k}.
\ee
\end{lemma}

\begin{proof}[Proof of Lemma \ref{lemmaboundthetak}]  This is a direct consequence of the form of $\mathcal G$ given in \eqref{fialfrmaulg} \eqref{nveioneionoenv|} and the proof is verbatim the same as that of Lemma 6.6 in \cite{MRRSprofile}.
\end{proof}

{Proposition \ref{propboundbdependent} follows immediately from \eqref{cneovnenvonveoinven}, \eqref{esitmaitmitot}, \eqref{cneioneineonoen}, \eqref{defphi}, and the following lemma:

\begin{lemma}[Stability by multiplication]
\label{lemmamultiplication}
 Let $h(u)=\sum_{k=0}^{+\infty}b^kh_ku^k$ with the holomorphic bound $$|h_k|\leq C^k.$$ Then there exists $C_h$ and $0<b^*(C_h)\ll 1$ such that for all {$0<b<b^*(C)$ and any $0\le k^*\le K$}, 
 \be
 \label{eoneonveonoe}\max_{0\le k{\le}k^*}\frac{|(h\phi)_k|}{w_{\gamma, {\nu_b}}(k)}\leq C_h \max_{0\le k{\le}k^*}\frac{|\phi_k|}{w_{\gamma, {\nu_b}}(k)}.
 \ee
\end{lemma}
\begin{proof}
The proof is verbatim the same as that of Lemma 6.7 in \cite{MRRSprofile}.
\end{proof}

\section{Quantitative study of the $C^\infty$ solution}
\label{sec:studyCinftysolution}
\label{cinftysolution}

We now turn to the qualitative of the $C^\infty$ solution of \eqref{eq: Quasilinear Equation}. 
Understanding of the Taylor expansion at $u=0$ is not sufficient to analyze the solution away from $u=0$. 
Our main goal  is to show that truncating the Taylor series at $k=K$ yields the dominant terms in the solution
which, together with a remainder, can be computed and estimated thanks to an explicit integral representation.\\

We study the $C^\infty$ solution. We define the operator 
\be
\label{defopeterator}
\mathcal T(\mathcal G)=\frac{u^{\gamma-2}}{(1-u)^{\gamma+\nu_b+1}}\int_0^u\frac{(1-v)^{\gamma+\nu_b}}{v^{\gamma-1}}\frac{\mathcal G}{b}dv.
\ee

\subsection{Remainder function}

We introduce several special functions defined via the integral operator \eqref{defopeterator}. These will be fundamental  in understanding the leading order terms which appear when the Taylor expansion no longer dominates.

\begin{lemma}[Holomorphic representation and bounds]
Let $j\ge 0$ and 
\be
\label{Pdefmj}
M_j(u)=(K+j+\nu_b+2)(-1)^{j}w_{\gamma-1,\nu_b+1}(K-1+j)\mathcal T(bu^{K+j}),
\ee
then we have the following convergent series representation for $|u|<1$:
\be
\label{formualmj}
M_j(u)=\sum_{m=j+1}^{+\infty}(-1)^{m}w_{\gamma-1,\nu_b+1}(K-1+m)u^{K-1+m}.
\ee
Moreover, let $1\le j\le {5}$, then there exist universal constants {$0<c_{\nu,1}<c_{\nu,2}$} such that:\\ \noindent\underline{behavior for small u}:  for $0\le u\le  b$,
\be
\label{lowerobundzerobis}
{c_{\nu,1}}\le \frac{M_j(u)}{\Gamma(\alpha_\gamma)\Gamma(1-\alpha_\gamma)K^{\nu_b+j+4-\alpha_\gamma}u^{K+j}}\le {c_{\nu,2}}\ee
\noindent\underline{behavior for large u}: for $b\le u< \frac 12$:
\be
\label{lowerobundzerobisbis}
{c_{\nu,1}}\le \frac{M_j(u)}{\Gamma(\alpha_\gamma)\Gamma(1-\alpha_\gamma)K^{\nu_b+3}\left(\frac{u}{1-u}\right)^{K-1}u^{\alpha_\gamma}}\leq {c_{\nu,2}}.
\ee
\noindent\underline{control of the iterate}:  let $1\le j\le 5$, then
\be
\label{bounditerate}
\left\|\frac{\T (u^jM_0)}{M_0}\right\|_{L^\infty(u\leq \frac 12)}\le \frac{c_\nu}{b}.
\ee

\noindent\underline{control of the derivative}: 
\be
\label{vniovnioneneneo}
\forall 0\le u\le \frac 12, \ \ \frac{|uM_0'|}{M_0}\leq \frac{c_\nu}b. 
\ee 
\end{lemma}

\begin{proof}
\noindent{\bf step 1} Holomorphic representation. Given a $C^\infty$ function with $\mathcal G=O_{u\to 0}(u^K)$, $\Theta=\mathcal T(\mathcal G)$ satisfies the linear equation
\be
\label{vneoinvenvoe}
u(1-u)\Theta' -\left[\gamma-2+(\nu_b+3)u\right]\Theta=\frac{\mathcal G}{b}.
\ee
We formally expand $$\Theta=\sum_{k=0}^{+\infty}\theta_ku^k, \ \ \mathcal G=\sum_{k=0}^{+\infty}g_ku^k$$ and obtain from \eqref{thetaeqaiotjoihs}:
\bea
\label{vneineionvioenvionoe}
\nonumber &&\sum_{k=1}^{+\infty}(u-u^2)k\theta_ku^{k-1}-(\gamma-2)\sum_{k=0}^{+\infty}\theta_ku^k-(\nu_b+3)\sum_{k=0}^{+\infty}\theta_ku^{k+1}=-\frac{1}{b}\sum_{k=0}^{+\infty}g_ku^{k}\\
\nonumber& \Leftrightarrow& \sum_{k=1}^{+\infty}k\theta_ku^k-\sum_{k=2}^{+\infty}(k-1)\theta_{k-1}u^k-(\gamma-2)\sum_{k=0}^{+\infty}\theta_ku^k-(\nu_b+3)\sum_{k=1}^{+\infty}\theta_{k-1}u^{k}=-\frac{1}{b}\sum_{k=0}^{+\infty}g_{k}u^{k}\\
& \Leftrightarrow&
\left|\begin{array}{l}
\theta_0=\frac{\gamma}{a(\gamma-2)}g_0\\
(\gamma-k-2)\theta_k+(k+\nu_b+2)\theta_{k-1}=\frac{g_{k}}{b}, \ \ k\ge 1.
\end{array}\right.
\eea
Let $$\zeta_k=\frac{\theta_k}{w_{\gamma-1,\nu_b+1}(k)},$$ then \eqref{inucnoitnwrmwjaggmamk},  \eqref{cneioneineonoen} yield:

\bea
\label{eboebveiboebeo}
\nonumber && (\gamma-k-2)w_{\gamma-1,\nu_b+1}(k)\zeta_k+(k+\nu_b+2)w_{\gamma-1,\nu_b+1}(k-1)\zeta_{k-1}=\frac{g_k}{b}\\
\nonumber &\Leftrightarrow& (k+\nu_b+2)w_{\gamma-1,\nu_b+1}(k-1)(\zeta_k+\zeta_{k-1})=\frac{\gamma g_k}{a}\\
&\Leftrightarrow& \zeta_k+\zeta_{k-1}=\frac{\gamma}{a(\gamma-2)}\frac{g_k}{(k+\nu_b+2)w_{\gamma,\nu_b}(k)}
\eea
which yields 
\bee
\zeta_k=-\frac{\gamma}{a(\gamma-2)}(-1)^k\sum_{j=0}^k\frac{(-1)^jg_j}{(j+\nu_b+2)w_{\gamma,\nu_b}(j)}, \ \  k\ge 1
\eee
and thus
\be
\label{enkvnevenenoenoivnedldnl}
\left|\begin{array}{l}
\theta_k=(-1)^kw_{\gamma-1,\nu_b+1}(k) S_k, \ \ k\ge 0\\
S_k=-\frac{\gamma}{a(\gamma-2)}\sum_{j=0}^k\frac{(-1)^jg_j}{(j+\nu_b+2)w_{\gamma,\nu_b}(j)}.
\end{array}\right.
\ee
Given $j\ge 0$ and $$\mathcal G=u^{K+j}$$ this yields:
$$\theta_k=\left|\begin{array}{l} 0\ \ \mbox{for}\ \ k\le K+j-1\\
(-1)^kw_{\gamma-1,\nu_b+1}(k) S_{K+j}\ \ \mbox{for}\ \ k\ge K+j
\end{array}\right.
$$ Therefore, the representation is a normally convergent series\footnote{{From \eqref{cnineneonevnev:0} 
\eqref{cneioneineonoen} and \eqref{aymptoticratio}, we have 
$$w_{\gamma-1,\nu_b+1}(k)=O(k^{K+\nu_b+3+\alpha_\gamma})\textrm{ as }k\to +\infty$$
which implies that the series converges for $|u|<1$.}} for $|u|<1$. Thus, from \eqref{cneioneineonoen}:
\bee
\mathcal T(u^{K+j})&=&-\frac{\gamma}{a(\gamma-2)}\frac{(-1)^{K+j}}{(K+j+\nu_b+2)w_{\gamma,\nu_b}(K+j)}\sum_{k=K+j}^{+\infty}(-1)^kw_{\gamma-1,\nu_b+1}(k) u^k\\
& = & \frac{1}{b}\frac{(-1)^{K-1+j}}{(K+j+\nu_b+2)w_{\gamma-1,\nu_b+1}(K+j-1)}\sum_{k=K+j}^{+\infty}(-1)^kw_{\gamma-1,\nu_b+1}(k) u^k
\eee
gives 
$$M_j(u)=(-1)^{K-1}\sum_{k=K+j}^{+\infty}(-1)^{k}w_{\gamma-1,\nu_b+1}(k)u^k=\sum_{m=j+1}^{+\infty}(-1)^{m}w_{\gamma-1,\nu_b+1}(K-1+m)u^{K-1+m}$$
and \eqref{formualmj} is proved. We now assume $j\ge 1$.\\

\noindent{\bf step 2} Proof of the bounds \eqref{lowerobundzerobis}, \eqref{lowerobundzerobisbis}, \eqref{bounditerate} and \eqref{vniovnioneneneo} can be found in Lemma 7.2 and Lemma 7.3 in \cite{MRRSprofile} and we conclude the proof.
\end{proof}

\subsection{Fixed point formulation of the $C^\infty$ solution}

For a given function $F$ with sufficient regularity at the origin, we denote
\be
\label{definitonioperatorramineder}
r_F(u)=F(u)-\sum_{k=0}^{K-1}f_ku^k, \ \ f_k=\frac{f^{(k)}(0)}{k!}.
\ee

\begin{lemma}[Fixed point formulation for the $C^\infty$ solution]
\label{fromulationcingty}
Let $\mathcal G$ be given by \eqref{fialfrmaulg} and consider the decomposition 
\be
\label{cnekoneonveonoevn}
\left|\begin{array}{l}
\mathcal G=\sum_{k=0}^{K-1}g_k u^k+r_{\mathcal G},\\
\Theta=\sum_{k={0}}^{K-1}\theta_k u^k+r_\Theta,
\end{array}\right.
\ee
where
\be
\label{enkvnevenenoenoivne}
\left|\begin{array}{l}
\theta_k=(-1)^kw_{\gamma-1,\nu_b+1}(k) S_k, \ \ 0\le k\le K-1,\\
S_k=\frac{\gamma}{a(\gamma-2)}\sum_{j=0}^k\frac{(-1)^jg_j}{(j+\nu_b+2)w_{\gamma,\nu_b}(j)},
\end{array}\right.
\ee
then the unique solution to the fixed point problem
\be
\label{fineioneoineeogn}
r_\Theta=(-1)^{K-1}S_{K-1}M_0(u)-\T(r_\mathcal G).
\ee
generates the unique solution $\Theta$ to \eqref{thetaeqaiotjoihs} which satisfies:
\be
\label{neoneneonovenoev}
\forall k\ge 0, \ \ \lim_{u\downarrow 0} \frac{\Theta^{(k)}(u)}{k!}=\theta_k
\ee
where $(\theta_k)_{k\ge 0}$ is computed by induction from \eqref{vneineionvioenvionoe}.
\end{lemma}

\begin{proof}[Proof of Lemma \ref{fromulationcingty}]
{\bf step 1} Polynomial cancellations. Recall \eqref{thetaeqaiotjoihs} and write \eqref{cnekoneonveonoevn}. We compute:
\bee
&&u(1-u)\left(\sum_{k=0}^{K-1}\theta_ku^k\right)'-(\gamma-2+(\nu_b+3)u)\sum_{k=0}^{K-1}\theta_ku^k\\
&=& \sum_{k=1}^{K-1}(k\theta_ku^k-k\theta_ku^{k+1})-\sum_{k=0}^{K-1}(\gamma-2)\theta_ku^k-(\nu_b+3)\sum_{k=0}^{K-1}\theta_ku^{k+1}\\
& = &  -(\gamma-2)\theta_0+\sum_{k=1}^{K-1}\left[(k-\gamma+2)\theta_k-(k+\nu_b+2)\theta_{k-1}\right]u^k-(K+\nu_b+2)\theta_{K-1}u^{K}
\eee
Therefore,
\bee
&&u(1-u)\left(\sum_{k=0}^{K-1}\theta_ku^k\right)'-(\gamma-2+(\nu_b+3)u)\sum_{k=0}^{K-1}\theta_ku^k+\sum_{k=0}^{K-1}\frac{g_k}{b}u^k\\
& = & \sum_{k=1}^{K-1}\left[(k-\gamma+2)\theta_k-(k+\nu_b+2)\theta_{k-1}+\frac{g_k}{b}\right]u^k\\
&-& (\gamma-2)\theta_0+\frac{g_0}{b}-(K+\nu_b+2)\theta_{K-1}u^{K}=-(K+\nu_b+2)\theta_{K-1}u^{K}
\eee
where the final equality follows from the induction relation
$$
\left|\begin{array}{l}
\theta_0=\frac{g_0}{b(\gamma-2)}\\
(\gamma-k-2)\theta_k+(k+\nu_b+2)\theta_{k-1}=\frac{g_k}{b}, \ \ 1\le k\le K-1
\end{array}\right.
$$
which equivalent to \eqref{enkvnevenenoenoivne}.\\

\noindent{\bf step 2} Equation for the remainder. We infer from \eqref{thetaeqaiotjoihs} and \eqref{cnekoneonveonoevn},
\bee
&&-(K+\nu_b+2)\theta_{K-1}u^{K}+u(1-u)r_\Theta ' -\left[\gamma-2+(\nu_b+3)u\right]r_\Theta=-\frac{r_\mathcal G}{b}\\
&\Leftrightarrow& r_\Theta'-\left[\frac{\gamma-2}{u}+\frac{\gamma+\nu_b+1}{1-u}\right]r_\Theta=\frac{1}{u(1-u)}\left[(K+\nu_b+2)\theta_{K-1}u^{K}-\frac{r_\mathcal G}{b}\right]\\
&\Leftrightarrow&\frac{d}{du}\left(\frac{(1-u)^{\gamma+\nu_b+1}}{u^{\gamma-2}}r_\Theta\right)'=\frac{(1-u)^{\gamma+\nu_b+1}}{u^{\gamma-2}}\frac{1}{u(1-u)}\left[(K+\nu_b+2)\theta_{K-1}u^{K}-\frac{r_\mathcal G}{b}\right]\\
\eee
and thus, any $C^\infty$  solution must be the unique solution to the fixed point equation:
$$
r_\Theta=\T\left[(K+\nu_b+2)\theta_{K-1}bu^{K}-r_\mathcal G\right]
$$
with \eqref{neoneneonovenoev} forced by the Taylor expansion\footnote{The statement on existence and uniqueness of the fixed point, the fact that the corresponding solution is smooth, and the fact that \eqref{neoneneonovenoev} holds has in fact already been proved in a more general case in Lemma \ref{lem: Regularity at P_3}.}. 
We now recall the definition of $M_0(u)$ in {\eqref{Pdefmj}} from which we infer
$$r_\Theta=\frac{\theta_{K-1}}{w_{\gamma-1,\nu_b+1}(K-1)}M_0(u)-\T(r_\mathcal G)=(-1)^{K-1}S_{K-1}M_0(u)-\T(r_\mathcal G),$$ this is \eqref{fineioneoineeogn}. 
\end{proof}

\subsection{Convergence of the leading order Taylor coefficient}

The truncation of the Taylor series produces the leading order term, provided the last Taylor coefficient is non zero. This is 
the $S_\infty(d,\ell)\neq 0$ condition.

\begin{lemma}
\label{lenkenopoeje}
We have
\bea\label{behaviorSKasBconvergesto0}
S_{K-1}=(1+o_{b\to 0}(1))S_\infty
\eea
where $S_\infty$ is given by 
\be\label{Sinf}
S_\infty(d,\ell) :=  \frac{1}{a}\sum_{j=0}^{+\infty}\frac{(-1)^ja^jg_j^\infty}{\Gamma(\nu+j+3)}
\ee
and where $g_j^\infty$ corresponds to the limiting problem \eqref{eq: Formal Limit Equation} and is given by  \eqref{eq: Expansion of Theta and G}. 
\end{lemma}

{\begin{proof}[Proof of Lemma \ref{lenkenopoeje}] Since we have from \eqref{eq: Nonlinear Bound on Limiting Sequence},
\bee
\sum_{j=0}^{+\infty}\left|\frac{(-1)^ja^jg_j^\infty}{\Gamma(\nu+j+3)}\right| \leq c_\nu\sum_{j=0}^{+\infty}\frac{1}{1+j^2}\leq c_\nu<+\infty,
\eee 
and from \eqref{boundednessgk}, 
\bee
\sum_{j=0}^{+\infty}\left|\frac{(-1)^jg_j}{(j+\nu_b+2)w_{\gamma,\nu_b}(j)}\right| \leq c_\nu\sum_{j=0}^{+\infty}\frac{1}{1+j^2}\leq c_\nu<+\infty,
\eee 
its suffices to prove the convergence term by term as $b\to 0$. Now, recall \eqref{eq: psi_k Limit},
\bee
\lim_{b\to 0}\frac{g_j}{b^j} = g_j^\infty.
\eee
Since from the definition of $w_{\gamma,\nu_b}(j)$, 
\bee
(j+\nu_b+2)w_{\gamma,\nu_b}(j) &=& \Gamma(\nu_b+j+3)\frac{\Gamma(\gamma-1-j)}{\Gamma(\gamma-1)}= \frac{\Gamma(\nu+j+3)}{\gamma^j}(1+o_{b\to 0}(1))\\
&=& \Gamma(\nu+j+3)\frac{b^j}{a^j}(1+o_{b\to 0}(1))
\eee
we obtain 
\bee
\lim_{b\to 0}(j+\nu_b+2)\frac{w_{\gamma,\nu_b}(j)}{b^j} &=& \frac{\Gamma(\nu+j+3)}{a^j}
\eee
from which we deduce the convergence term by term when $b\to 0$, as desired. 
\end{proof}}

\subsection{The $\Thetam$ leading order term}

We may now extract the leading order terms in $\Theta$. From {\eqref{cnekoneonveonoevn}, \eqref{enkvnevenenoenoivne}, \eqref{fineioneoineeogn}} {and \eqref{behaviorSKasBconvergesto0}}, 
\bea
\label{expressionnonlnieanrtmer}
\nonumber \Theta(u) &=&\sum_{k=0}^{K-2}\theta_k u^k+\theta_{K-1}u^{K-1}+(-1)^{K-1}S_{K-1}M_0(u)-\T(r_\mathcal G)\\
\nonumber & = & \sum_{k=0}^{K-2}\theta_k u^k+(-1)^{K-1}S_{K-1}\left[w_{\gamma-1,\nu_b+1}(K-1)u^{K-1}+M_0(u)\right]-\T(r_\mathcal G)\\
& = & \sum_{k=0}^{K-2}\theta_k u^k+(-1)^{K-1}S_{\infty}\left[1+o_{b\to 0}(1)\right]\Thetam(u)-\T(r_\mathcal G)
\eea
and from \eqref{formualmj}
\bea
\label{deftehtemain}
\nonumber \Thetam(u)&=&w_{\gamma-1,\nu_b+1}(K-1)u^{K-1}+M_0(u)\\
& = &\sum_{j=0}^{+\infty}(-1)^{j}w_{\gamma-1,\nu_b+1}(K-1+j)u^{K-1+j}\nonumber\\&=&{\Gamma(\alpha_\gamma)\Gamma(1-\alpha_\gamma)\sum_{j=0}^{+\infty}\frac{\Gamma(K+j+\nu_b+2)}{\Gamma(K-1+\alpha_\gamma)\Gamma(j+1-\alpha_\gamma)}u^{K-1+j}.}
\eea
 Since we have from \eqref{cnineneonevnev:0} 
\eqref{cneioneineonoen} and \eqref{aymptoticratio} that
\bea
\nonumber&&(-1)^{j}w_{\gamma-1,\nu_b+1}(K-1+j)=\Gamma(\alpha_\gamma)\Gamma(1-\alpha_\gamma)\frac{\Gamma(K+j+\nu_b+2)}{\Gamma(K-1+\alpha_\gamma)\Gamma(j+1-\alpha_\gamma)}\\
\label{nekvneneonene} & = & \frac{\Gamma(\alpha_\gamma)\Gamma(1-\alpha_\gamma)}{\Gamma(j+1-\alpha_\gamma)}\left[1+o_{b\to 0}(1)\right]K^{\nu_b+j+3-\alpha_\gamma},
\eea
using the definition \eqref{formualmj} of $M_j(u)$, \eqref{deftehtemain} is equivalent to
\bea
\label{cneiovnenvenenne}
\nonumber &&\Thetam{(u)} =\Gamma(\alpha_\gamma)\Gamma(1-\alpha_\gamma)\frac{\Gamma(K+\nu_b+2)}{\Gamma(K-1+\alpha_\gamma)\Gamma(1-\alpha_\gamma)}u^{K-1}\\
\nonumber &+& \Gamma(\alpha_\gamma)\Gamma(1-\alpha_\gamma)\frac{\Gamma(K+\nu_b+3)}{\Gamma(K-1+\alpha_\gamma)\Gamma(2-\alpha_\gamma)}u^K+M_1(u) =  \Gamma(\alpha_\gamma)\Gamma(1-\alpha_\gamma)K^{\nu_b+3-\alpha_\gamma}u^{K-1}\\
\nonumber &\times & \left\{\left[1+o_{b\to 0}(1)\right]\left[\frac{1}{\Gamma(1-\alpha_\gamma)}+\frac{K+\nu_b+2}{\Gamma(2-\alpha_\gamma)}u\right]+(u K)^{\alpha_\gamma}\frac{M_1(u)}{u^{K-1}\Gamma(\alpha_\gamma)\Gamma(1-\alpha_\gamma)K^{\nu_b+3}u^{\alpha_\gamma}}\right\}.
\eea
We now turn to the study of $\Thetam$.

\begin{lemma}[Properties of $\Thetam$]
\label{lemmanekvneneoneon}
The function $\Thetam$ is {positive and strictly increasing} on $[0,\frac 12]$. Moreover, pick universal constants $\frac{1}{\delta},\Theta^*\gg1$ {then for all} $0<b<b^*(\Theta^*,\delta)$, the unique solution to 
\be
\label{neknvonenneudatsgeeaa}
\Thetam(u^*(\alpha_\gamma))=\Theta^*,  \ \ u^*(\alpha_\gamma)\in \left(0,\frac 12\right)
\ee 
satisfies the following bounds:\\

\underline{first integer boundary layer}.  {If $\alpha_\gamma$ is such that}
 $${\alpha_\gamma}=\frac{K^{\nu_b+3}(\sigma b)^{K-1}}{\Theta^*}, \ \textrm{{ with} }\delta<\sigma<\frac{1}{\delta}$$ then 
\be
\label{valueapprocimate}
\Gamma(\alpha_\gamma)\Gamma(1-\alpha_\gamma)K^{\nu_b+3-\alpha_\gamma}(u^*(\alpha_\gamma))^{K-1}=\Theta^*e^{O_{\nu}(\sigma)}
\ee
and
\be
\label{firsteataimteboundary}
u^*({\alpha_\gamma})=\sigma b(1+o_{b\to 0}(1)).
\ee

\noindent\underline{second {integer} boundary layer}.  {If $\alpha_\gamma$ is such that}
$${\alpha_\gamma}=1-\frac{K^{\nu_b+3}(\sigma b)^{K-1}}{\Theta^*}, \ \textrm{ {with} } \delta<\sigma<\frac 1{\delta},$$ then 
 \eqref{firsteataimteboundary} holds {and} 
 {\be
\label{valueapprocimate:2ndcase}
\Gamma(\alpha_\gamma)\Gamma(1-\alpha_\gamma)K^{\nu_b+3-\alpha_\gamma}\frac{K+\nu_b+2}{\Gamma(2-\alpha_\gamma)}(u^*(\alpha_\gamma))^{K}=\Theta^*e^{O_{\nu}(\sigma)}
\ee}
and
$$
u^*({\alpha_\gamma})=\sigma b(1+o_{b\to 0}(1)).
$$
\noindent\underline{away from the integer boundary layer $\alpha_\gamma$}.   {If $\alpha_\gamma$ is such that}
$$\frac{K^{\nu_b+3}\left(\frac{b}{\delta}\right)^{K-1}}{\Theta^*}<\alpha_\gamma<1-\frac{K^{\nu_b+3}\left(\frac{b}{\delta}\right)^{K-1}}{\Theta^*}$$
then 
\be
\label{jvheoeonenoenenv}
\Gamma(\alpha_\gamma)\Gamma(1-\alpha_\gamma)K^{\nu_b+3}\left(\frac{u^*(\alpha_\gamma)}{1-u^*(\alpha_\gamma)}\right)^{K-1}(u^*(\alpha_\gamma))^{\alpha_\gamma}=\Theta^*e^{O_\nu(1)}
\ee and 
\be\label{vnieneoneonevnonen}
\frac{b}{2\delta}<u^*(\alpha_\gamma)<\frac 12.
\ee
\end{lemma}

\begin{proof}[Proof of Lemma \ref{lemmanekvneneoneon}]
The proof is verbatim the same as that of Lemma 7.6 in \cite{MRRSprofile}.
\end{proof}

\subsection{Bilinear estimate for $\mathcal T$}

We now develop the set of nonlinear estimates to control the fixed point equation \eqref{fineioneoineeogn}.

\begin{lemma}[Pointwise bilinear estimate]
\label{lemmabilinear}
Let $$F=\sum_{k=0}^{K-1}f_ku^k+r_F, \ \ G=\sum_{k=0}^{K-1}g_ku^k+r_G.$$ 
and $$A_F=\sup_{0\le k\le K-1}\frac{|f_k|}{|w_{\gamma-1,\nu_b+1}(k)|}, \ \ A_G=\sup_{0\le k\le K-1}\frac{|g_k|}{|w_{\gamma-1,\nu_b+1}(k)|},$$ then we have the following pointwise bounds for $0\leq u\leq \frac 12$:\\
\begin{enumerate}
\item{Bound for $\T r$}. Let $1\le j\le 5$, then
\bea
\label{pointwiseboudnproftu}
&&\left|\frac{\mathcal T(r_{u^jFG})}{M_0}(u)\right|\leq   c_{\nu,a}A_FA_G\left[1+u^{K+j-2}\Gamma(\alpha_\gamma)K^{\nu_b+3-\alpha_\gamma}\right]\\
\nonumber & + &  \frac{c_{\nu,a}}{b}\left(1+u^{K-1}\Gamma(\alpha_\gamma)K^{\nu_b+3-\alpha_\gamma}\right)\left[A_F\left\|\frac{r_G}{M_0}\right\|_{L^\infty(v\le u)}+A_G\left\|\frac{r_F}{M_0}\right\|_{L^\infty(v\le u)}\right]\\
\nonumber &+&\frac{c_{\nu,a}}{b} \left\|\frac{r_F}{M_0}\right\|_{L^\infty(v\le u)}\left\|\frac{r_G}{M_0}\right\|_{L^\infty(v\le u)}\|M_0\|_{L^\infty(v\le u)}.
\eea

\item{Bound for $r$}. Let $0\le j\le 5$,
\bea
\label{pointwiseboudnproftubis}
\nonumber&&\left|\frac{r_{u^jFG}}{M_0}(u)\right|\leq   c_{\nu,a}A_FA_G\left[1 {+u^{K-2}\Gamma(\alpha_\gamma)K^{\nu_b+2-\alpha_\gamma}}+u^{K-1}\Gamma(\alpha_\gamma)K^{\nu_b+3-\alpha_\gamma}\right]\\
\nonumber & + &  c_{\nu,a}\left(1+u^{K-1}\Gamma(\alpha_\gamma)K^{\nu_b+3-\alpha_\gamma}\right)\left[A_F\left\|\frac{r_G}{M_0}\right\|_{L^\infty(v\le u)}+A_G\left\|\frac{r_F}{M_0}\right\|_{L^\infty(v\le u)}\right]\\
 &+&c_{\nu,a} \left\|\frac{r_F}{M_0}\right\|_{L^\infty(v\le u)}\left\|\frac{r_G}{M_0}\right\|_{L^\infty(v\le u)}\|M_0\|_{L^\infty(v\le u)}.
\eea
\end{enumerate}
\end{lemma}

\begin{proof}[Proof of Lemma \ref{lemmabilinear}] 
The proof is verbatim the same as that of Lemma 7.7 in \cite{MRRSprofile}.
\end{proof}

\subsection{Controlling the final remainder}

We are now in position to prove the exit condition for the $C^\infty$ solution for a large enough range of parameters.

\begin{lemma}[Uniform control of the final remainder]
\label{propositionfundamental}
Pick universal constants $\frac{1}{\delta},\Theta^*\gg 1$, then for all $0<b<b^*(\delta,\Theta^*)\ll1$ small enough, the following holds. Let 
\be
\label{boudnaryleyeralphag}
\frac{K^{\nu_b+3}\left(b\delta\right)^{K-1}}{\Theta^*}<\alpha_\gamma<1-\frac{K^{\nu_b+3}\left(b\delta\right)^{K-1}}{\Theta^*}
\ee
and  let $u^*(\alpha_\gamma)$ be the solution to \eqref{neknvonenneudatsgeeaa} described by Lemma \ref{lemmanekvneneoneon}. Then the $C^\infty$ solution to \eqref{fineioneoineeogn} satisfies the bound:
\be
\label{neionenoenven}
\frac{|r_{\mathcal G}|}{M_0}+\frac{|\T r_{\mathcal G}|}{M_0}<\sqrt{b}.
\ee
\end{lemma}

\begin{proof}[Proof of Lemma \ref{propositionfundamental}]  We recall the fixed point formulation \eqref{expressionnonlnieanrtmer} of the $C^\infty$ solution and the expression \eqref{fialfrmaulg} for the nonlinear term:
\bee
 &&\mathcal G= \mathcal G_0+\left[b^2x\th_1+bx^2\th_2\right]\Theta+\left[b^2x\th_3+bx^2\th_4\right]u\Theta'\\
 \nonumber & + &  \sum_{j=2}^4x^{j+1}\mt^{(1)}_j\Theta^j+b\sum_{j=2}^4m^{(2)}_jx^j\Theta^j+  \left[\sum_{j=1}^3\mt^{(3)}_jx^{j+2}\Theta^j+b\sum_{j={1}}^{{3}}\mt^{(4)}_jx^{j{+1}}\Theta^j\right](u\Theta').
 \eee
 We now bootstrap the bound
 \be
 \label{bootstrap}
 \frac{|r_\Theta(u)|}{M_0}<\Theta^*.
 \ee
 {Let us first check that \eqref{bootstrap} holds for $u$ small enough. From  \eqref{fineioneoineeogn}, we have
\bee
 \frac{|r_\Theta(u)|}{M_0} &\leq& c_{\nu,a}+ \frac{|\TT(r_{\mathcal{G}})(u)|}{M_0}\leq c_{\nu,a}+ \left(\sup_{v\leq u}\frac{r_{\mathcal{G}}(v)|}{v^K}\right)\frac{|\TT(u^K)|}{M_0}\\
 &\leq& c_{\nu,a}+ \left(\sup_{v\leq u}\frac{|r_{\mathcal{G}}(v)|}{v^K}\right)\frac{1}{b(K+\nu_b+2)w_{\gamma-1,\nu_b+1}(K-1)}  
\eee
where we have used \eqref{Pdefmj} in the last inequality. Also, we have, using \eqref{boundednessgk} and \eqref{cneioneineonoen},
\bee
\lim_{u\to 0}\sup_{v\leq u}\frac{|r_{\mathcal{G}}(v)|}{v^K} = |g_K| \leq c_{\nu,a}\frac{|w_{\gamma,\nu_ b}(K)|}{1+K}\leq c_{\nu,a}\frac{|w_{\gamma-1,\nu_ b+1}(K-1)|}{(1+K)(\gamma-2)}.
\eee
We infer for $u$ small enough that
\bee
\frac{|r_\Theta(u)|}{M_0} \leq c_{\nu,a}<\Theta_*
\eee
so that \eqref{bootstrap} indeed holds for $u$ small enough}. We therefore work on the interval ${u}\in[0,u_{\rm boot}]$ with ${0<}u_{\rm boot}\le u^*(\alpha_\gamma)$ where \eqref{bootstrap} holds, and aim at improving \eqref{bootstrap}.\\

\noindent{\bf step 1} Uniform bounds for $0\le u\le u^*(\alpha_\gamma)$. By definition \eqref{deftehtemain}: $$\Thetam(u)=w_{\gamma-1,\nu_b+1}(K-1)u^{K-1}+M_0(u)$$ and hence, since $\Thetam$ is non decreasing: 
\be
\label{uniformbound}
\forall u\in [0,u^*(\alpha_\gamma)], \ \ 0\le M_0(u)\le \Thetam(u)\le \Thetam(u^*(\alpha_\gamma))=\Theta^*.
\ee
Observe that in the regime \eqref{valueapprocimate}:
$$
\Gamma(\alpha_\gamma)\Gamma(1-\alpha_\gamma)K^{\nu_b+3-\alpha_\gamma}(u^*(\alpha_\gamma))^{K-1}=C_{\Theta^*,\delta},$$
{ in the regime \eqref{valueapprocimate:2ndcase}, recalling \eqref{firsteataimteboundary}:
\bee
\Gamma(\alpha_\gamma)\Gamma(1-\alpha_\gamma)K^{\nu_b+3-\alpha_\gamma}(u^*(\alpha_\gamma))^{K-1}\leq \frac{\Gamma(2-\alpha_\gamma)\Theta^*e^{O_{\nu}(\sigma)}}{u^*(\alpha_\gamma)(K+\nu_b+2)}\leq \frac{\Theta^*e^{O_{\nu}(\sigma)}}{\de}\leq C_{\Theta^*,\delta},
\eee}
and in the regime \eqref{jvheoeonenoenenv}, recalling \eqref{vnieneoneonevnonen}:
$$\Gamma(\alpha_\gamma)\Gamma(1-\alpha_\gamma)K^{\nu_b+3-\alpha_\gamma}(u^*(\alpha_\gamma))^{K-1}\leq  \frac{(1-u^*(\alpha_\gamma))^{K-1}}{(Ku^*(\alpha_\gamma))^{\alpha_\gamma}}\Theta^*e^{O_\nu(1)}\leq C_{\Theta^*,\delta}
$$
In {all three} cases 
\be
\label{noennoeoenvoen}
\forall u\in[0,u^*(\alpha_\gamma)], \ \  \Gamma(\alpha_\gamma)\Gamma(1-\alpha_\gamma)K^{\nu_b+3-\alpha_\gamma}u^{K-1}\le C_{\Theta^*,\delta}.
\ee
{Also, we have, using \eqref{noennoeoenvoen},
\bee
\Gamma(\alpha_\gamma)\Gamma(1-\alpha_\gamma)K^{\nu_b+2-\alpha_\gamma}u^{K-2}\le \Gamma(\alpha_\gamma)\Gamma(1-\alpha_\gamma)K^{\nu_b+2-\alpha_\gamma}(u^*(\alpha_\gamma))^{K-2}\le \frac{C_{\Theta^*,\delta}}{Ku^*(\alpha_\gamma)}
\eee
and thus, using \eqref{firsteataimteboundary} \eqref{vnieneoneonevnonen}, we deduce
\be
\label{noennoeoenvoen:0:0}
\forall u\in[0,u^*(\alpha_\gamma)], \ \  \Gamma(\alpha_\gamma)\Gamma(1-\alpha_\gamma)K^{\nu_b+2-\alpha_\gamma}u^{K-2}\le C_{\Theta^*,\delta}.
\ee
}
We now {decompose $r_{\mathcal{G}}$ according to the decomposition \eqref{fialfrmaulg} of $\mathcal{G}$} and estimate all the terms.\\

\noindent{\bf step 2} Source term. For any holomorphic function $H(x)$, we have from {\eqref{definitonioperatorramineder}}:
$$|r_{H}|=\left|\sum_{k=K}^{+\infty}h_ku^{k}\right|\leq \sum_{k=K}^{+\infty}(bC_hu)^{k}\le (bC_h)^Ku^K.$$ We therefore estimate for $u\le b$ from \eqref{lowerobundzerobis}:
$$\frac{|r_{H}|}{M_0}\leq  \frac{(bC_H)^Ku^K}{\Gamma(\alpha_\gamma)\Gamma(1-\alpha_\gamma)K^{\nu_b+4-\alpha_\gamma}u^{K}}\le    \frac{(bC_H)^K}{\Gamma(\alpha_\gamma)\Gamma(1-\alpha_\gamma)K^{\nu_b+3}}$$
and for $b\le u\le \frac 12$ from {\eqref{lowerobundzerobisbis}}:
\bee
\frac{|r_H|}{M_0}\le \frac{(bC_H)^Ku^K}{\Gamma(\alpha_\gamma)\Gamma(1-\alpha_\gamma)K^{\nu_b+3}\left(\frac{u}{1-u}\right)^{K-1}u^{\alpha_\gamma}}\leq \frac{(bC_H)^K}{\Gamma(\alpha_\gamma)\Gamma(1-\alpha_\gamma)K^{\nu_b+3}}
\eee
It implies the rough bound
\be
\label{vebivebibebeibv}
\left\|\frac{r_{H}}{M_0}\right\|_{L^\infty(0\le u\le \frac 12)}\leq b^4.
\ee 
In view of \eqref{nveioneionoenv|} this bound can be applied to the source term $\mathcal G_0$.\\

\noindent{\bf step 3} Derivative term. From \eqref{cnekoneonveonoevn}, we have
\be
\label{firstrkatjoijit}
\left|\begin{array}{l}
(u\Theta')_k=k\theta_k, \ \  0\le k\le K-1\\
r_{u\Theta'}=ur_\Theta'.
\end{array}\right.
\ee
Moreover, from \eqref{fineioneoineeogn},
$$ur'_\Theta=(-1)^{K-1}S_{K-1}uM'_0(u)-u\left[\T(r_\mathcal G)\right]'.$$ 
 By taking derivative of \eqref{vneoinvenvoe},
$$u\left[\T(\mathcal G)\right]'=\frac{1}{1-u}\left[\frac{\mathcal G}{b}+[(\gamma-2)+(\nu_b+3)u]\mathcal T(\mathcal G)\right]$$
which yields 
$$
r_{u\Theta'}=(-1)^{K-1}S_{K-1}uM'_0(u)-\frac{1}{1-u}\left[\frac{r_\mathcal G}{b}+[(\gamma-2)+(\nu_b+3)u]\mathcal T(r_\mathcal G)\right]
$$ 
We obtain the estimate,  using \eqref{vniovnioneneneo} {and \eqref{esitmaitmitot}}:
\be
\label{estiamtieboudnary}
\left|\begin{array}{l}
\frac{|r_{u\Theta'}|}{M_0}\leq \frac{c_\nu}{b}\left[1+\frac{|r_\mathcal G|}{M_0}+\frac{|\T(r_\mathcal G)|}{M_0}\right]\\
\sup_{0\le k\le K-1}\frac{|(u\Theta')_k|}{w_{\gamma-1,\nu_b+1}(k)}\leq \frac{c_{\nu}}{b}.
\end{array}\right.
\ee

\noindent{\bf step 4} Linear term. Recall from \eqref{eoneonveonoe} that for a holomorphic function: $$\forall 0\le k\le K-1, \ \ |(H \Theta)_k|\le c_H w_{\gamma-1,\nu_b+1}(k).$$
\noindent\underline{no derivative term}. We apply \eqref{pointwiseboudnproftu} with the bounds \eqref{vebivebibebeibv}, \eqref{uniformbound}, \eqref{noennoeoenvoen} and \eqref{esitmaitmitot} to derive for $ 1\le j\le 5$:
\bee
&&\left|\frac{\mathcal T(r_{u^jH\Theta})}{M_0}(u)\right|\leq   c_{\nu,a}A_\Theta A_H\left[1+u^{K+j-2}\Gamma(\alpha_\gamma)K^{\nu_b+3-\alpha_\gamma}\right]\\
\nonumber & + &  \frac{c_{\nu,a}}{b}\left(1+u^{K-1}\Gamma(\alpha_\gamma)K^{\nu_b+3-\alpha_\gamma}\right)\left[A_\Theta\left\|\frac{r_H}{M_0}\right\|_{L^\infty(v\le u)}+A_H\left\|\frac{r_\Theta}{M_0}\right\|_{L^\infty(v\le u)}\right]\\
\nonumber &+&\frac{c_{\nu,a}}{b} \left\|\frac{r_\Theta}{M_0}\right\|_{L^\infty(v\le u)}\left\|\frac{r_H}{M_0}\right\|_{L^\infty(v\le u)}\|M_0\|_{L^\infty(v\le u)}\leq c_{\nu,a}+\frac{C_{\Theta^*,\delta}}{b}
\eee
where we used the bootstrap bound \eqref{bootstrap} in the last step.
Now, writing $$\left|\begin{array}{l}
b^2x\th_1\Theta =b^3u(\th_1\Theta)\\
bx^2\th_2\Theta=b^3u^2(\th_2\Theta)
\end{array}\right.
$$
gives
$$\left|\frac{\T r_{\left[b^2x\th_1+bx^2\th_2\right]\Theta}}{M_0}\right|\leq b^3\left[c_{\nu,a}+\frac{C_{\Theta^*,\delta}}{b}\right]\le b^2C_{\Theta^*,\delta}.$$
Similarly, from \eqref{pointwiseboudnproftubis}, \eqref{bootstrap}, {\eqref{vebivebibebeibv}, \eqref{uniformbound}, \eqref{noennoeoenvoen} \eqref{noennoeoenvoen:0:0} and \eqref{esitmaitmitot}}:
$$\left|\frac{r_{u^jH\Theta}}{M_0}(u)\right|\leq   C_{\Theta^*,\delta}.$$
Therefore,
$$\left|\frac{r_{\left[b^2x\th_1+bx^2\th_2\right]\Theta}}{M_0}\right|\le b^3C_{\Theta^*,\delta}.$$

\noindent\underline{derivative term}. We first use \eqref{estiamtieboudnary}, \eqref{pointwiseboudnproftubis} to estimate for a holomorphic function $H$:
\bee
&&\left|\frac{r_{u^jH(u\Theta')}}{M_0}(u)\right|\leq   \frac{c_{\nu,a}}{b}\left[1{+u^{K-1}\Gamma(\alpha_\gamma)K^{\nu_b+3-\alpha_\gamma}
+u^{K-2}\Gamma(\alpha_\gamma)K^{\nu_b+2-\alpha_\gamma}}\right]\\
\nonumber & + &  c_{\nu,a}\left(1+u^{K-1}\Gamma(\alpha_\gamma)K^{\nu_b+3-\alpha_\gamma}\right)\left[\frac{1}{b}+\left\|\frac{r_{u\Theta'}}{M_0}\right\|_{L^\infty(v\le u)}\right]\\
\nonumber &+&C_{\Theta^*,\delta} \left\|\frac{r_{u\Theta'}}{M_0}\right\|_{L^\infty(v\le u)}\leq  C_{\Theta^*,\delta}\left[\frac{1}{b}+\left\|\frac{r_{u\Theta'}}{M_0}\right\|_{L^\infty(v\le u)}\right]
\eee
and from {\eqref{estiamtieboudnary}}, \eqref{pointwiseboudnproftu}:
$$\left|\frac{\mathcal T(r_{u^jH(u\Theta')})}{M_0}(u)\right|\leq   C_{\Theta^*,\delta}\left[\frac{1}{b^2}+\frac{1}{b}\left\|\frac{r_{u\Theta'}}{M_0}\right\|_{L^\infty(v\le u)}\right].$$
We conclude, using \eqref{estiamtieboudnary}:
\bee
&&\left|\frac{\T r_{\left[b^2x\th_3+bx^2\th_4\right]u\Theta'}}{M_0}\right|\leq b^3C_{\Theta^*,\delta}\left[\frac{1}{b^2}+\frac{1}{b}\left\|\frac{r_{u\Theta'}}{M_0}\right\|_{L^\infty(v\le u)}\right]\leq  C_{\Theta^*,\delta}\left[b+b^2\left\|\frac{r_{u\Theta'}}{M_0}\right\|_{L^\infty(v\le u)}\right]\\
& \leq & C_{\Theta^*,\delta}\left[b+b\left\|\frac{r_\mathcal G}{M_0}\right\|_{L^\infty(v\le u)}+b\left\|\frac{\mathcal T(r_\mathcal G)}{M_0}\right\|_{L^\infty(v\le u)}\right]
\eee
and
\bee
&&\left|\frac{r_{\left[b^2x\th_3+bx^2\th_4\right]u\Theta'}}{M_0}\right|\le b^3C_{\Theta^*,\delta}\left[\frac{1}{b}+\left\|\frac{r_{u\Theta'}}{M_0}\right\|_{L^\infty(v\le u)}\right]\\
& \leq & C_{\Theta^*,\delta}\left[b^2+b^2\left\|\frac{r_\mathcal G}{M_0}\right\|_{L^\infty(v\le u)}+b^2\left\|\frac{\mathcal T(r_\mathcal G)}{M_0}\right\|_{L^\infty(v\le u)}\right]
\eee

\noindent{\bf step 5} Nonlinear term.\\

\noindent\underline{no derivative term}. First, {we have in view of \eqref{esitmaitmitot}}  \eqref{tobeprovoeonorbibib} that for $1\le m\le 5$ and $0\le k\le K-1$: 
\bee
|(\Theta^m)_k|&\le& {\sum_{k_1+\cdots k_m=k}|\Theta_{k_1}|\cdots |\Theta_{k_m}|\le c_{\nu,a}\sum_{k_1+\cdots k_m=k}w_{\gamma-1, \nu_b+1}(k_1)\cdots w_{\gamma-1, \nu_b+1}(k_m)}\\
&\le& c_{\nu,a}w_{\gamma-1,\nu_b+1}{(k)}
\eee
We then estimate, {using} \eqref{pointwiseboudnproftubis} {iteratively in $m$ }for $2\le m\le 5$, $0\le j\le 5$, {and using also \eqref{bootstrap}, \eqref{vebivebibebeibv}, \eqref{uniformbound}, \eqref{noennoeoenvoen} \eqref{noennoeoenvoen:0:0}}:
\bee
\left|\frac{r_{u^j\Theta^m}}{M_0}(u)\right|&\leq& c_{\nu,a}\left(A_{\Theta}+\left\|\frac{r_\Theta}{M_0}\right\|_{L^\infty(v\le u)}\right)\left(\sum_{\ell=1}^mA_{\Theta^\ell}+\left(A_{\Theta}+\left\|\frac{r_\Theta}{M_0}\right\|_{L^\infty(v\le u)}\right)^{m-2}\right)\\
&&\times\left[1+u^{K-1}\Gamma(\alpha_\gamma)K^{\nu_b+3-\alpha_\gamma}{+u^{K-2}\Gamma(\alpha_\gamma)K^{\nu_b+2-\alpha_\gamma}}\right]^{{m-1}}\leq C_{\Theta^*,\delta}
\eee
Similarly,  for a holomorphic function $H$:
\be
\label{neevnnel;m;ldeoen}
\left|\frac{r_{u^jH\Theta^m}}{M_0}(u)\right|   {\leq C_{\Theta^*,\delta}}
\ee
which implies that
$$
 \left|\frac{r_{\sum_{j=2}^4x^{j+1}\mt^{(1)}_j\Theta^j+b\sum_{j=2}^4m^{(2)}_jx^j\Theta^j}}{M_0}\right|\le C_{\Theta^*,\delta}b^3.$$
Similarly, for $1\le j\le 5$ from {\eqref{bootstrap}} \eqref{vebivebibebeibv}, \eqref{uniformbound}, \eqref{noennoeoenvoen}:
$$
\left|\frac{\mathcal T(r_{u^jH\Theta^j})}{M_0}(u)\right|\le \frac{C_{\Theta^*,\delta}}{b}
$$
gives
\bee
 \left|\frac{\T r_{\sum_{j=2}^4x^{j+1}\mt^{(1)}_j\Theta^j+b\sum_{j=2}^4m^{(2)}_jx^j\Theta^j}}{M_0}\right|\le b^3\frac{C_{\Theta^*,\delta}}{b}\le C_{\Theta^*,\delta}b^2.
 \eee
 
 \noindent\underline{derivative term}. We estimate from \eqref{neevnnel;m;ldeoen}, \eqref{pointwiseboudnproftubis}, \eqref{estiamtieboudnary}, {\eqref{uniformbound}, \eqref{noennoeoenvoen} \eqref{noennoeoenvoen:0:0}}:
 \bee
&& b\left|\frac{\T r_{u^jH\Theta^m(u\Theta')}}{M_0}(u)\right|+ \left|\frac{r_{u^jH\Theta^m(u\Theta')}}{M_0}(u)\right|\leq C_{\Theta^*,\delta}\left[\frac{1}{b}+\left\|\frac{r_{u\Theta'}}{M_0}\right\|_{L^\infty(v\le u)}\right]\\
& \leq & \frac{C_{\Theta^*,\delta}}{b}\left[1+{\left\|\frac{r_\mathcal G}{M_0}\right\|_{L^\infty(v\le u)}+\left\|\frac{\mathcal T(r_\mathcal G)}{M_0}\right\|_{L^\infty(v\le u)}}\right].
 \eee
 Therefore, collecting all the nonlinear terms above,
 \bee
 b\frac{|\T r_{\NL(\Phi)}|}{M_0}+\frac{|r_{\NL(\Phi)}|}{M_0}\leq b^2 C_{\Theta^*,\delta}\left[1+{\left\|\frac{r_\mathcal G}{M_0}\right\|_{L^\infty(v\le u)}+\left\|\frac{\mathcal T(r_\mathcal G)}{M_0}\right\|_{L^\infty(v\le u)}}\right].
 \eee
for $\NL(\Phi)$ as in \eqref{formaulnonlienaterm}.\\\\
\noindent{\bf step 6} Conclusion. The collection of the above bounds yields:
\bee
b\frac{|\T r_{\mathcal G}|}{M_0}+\frac{|r_{\mathcal G}|}{M_0}\le b^2C_{\Theta^*,\delta}{\left[1+\left\|\frac{r_\mathcal G}{M_0}\right\|_{L^\infty(v\le u)}+\left\|\frac{\mathcal T(r_\mathcal G)}{M_0}\right\|_{L^\infty(v\le u)}\right]}
\eee
which imply for $0<b<b^*(\Theta*,\delta)$ small enough $$\frac{|r_{\mathcal G}|}{M_0}+\frac{|\T r_{\mathcal G}|}{M_0}<\sqrt{b}$$ which, reinserted into \eqref{fineioneoineeogn}, yields:
$$\frac{|r_\Theta|}{M_0}\le c_{\nu,a}+\sqrt{b}<\frac{\Theta^*}{2},$$ and \eqref{bootstrap}, \eqref{neionenoenven} are proved.
\end{proof}

\subsection{{Exit on the right of $P_3$}}

We are now in position to establish the fundamental exit property of the $C^\infty$ solution to the right of $P_3$. We assume without loss of generality that 
\be
\label{neinenvoneonve}
S_\infty(d,\ell)>0
\ee and need only to reverse the parity of $K$ in the following Lemma if $S_\infty(d,\ell)<0$.

\begin{lemma}[Exit on the right]
\label{lemmaexitleft}
Pick universal constants $\frac{1}{\delta},\Theta^*\gg1 $ large enough as in Lemma \ref{propositionfundamental}, then  for all $0<b<b^*(\Theta^*,\delta)$ small enough and $\alpha_\gamma$ in the range \eqref{boudnaryleyeralphag}, the $C^\infty$ solution exits on the right of $P_3$ { at $u=U_*$, where $0<U_*<\frac{3}{4}$}, by crossing $\Delta_1=0$ for $K$ odd, and by crossing $\Delta_2=0$ for $K$ even.
\end{lemma}

\begin{proof}[Proof of Lemma \ref{lemmaexitleft}]  
\noindent{\bf step 1} Reaching $\Theta^*$. Recall {\eqref{expressionnonlnieanrtmer}}
$$
\Theta(u) =\sum_{k=0}^{K-2}\theta_k u^k+(-1)^{K-1}S_{\infty}\left[1+o_{b\to 0}(1)\right]\Thetam(u)-\T(r_\mathcal G).
$$ 
We now claim the uniform bound for the Taylor expansion term above for $u\le \frac 12$:
\be
\label{estraminderwrror}
\sum_{k=0}^{K-2}w_{\gamma-1,\nu_b+1}(k)u^k\le c_{\nu,a}
\ee
From {\eqref{cneioneineonoen} and} \eqref{summationphi}, for some large enough $K_\nu$:
$$\sum_{k=0}^{K-K_\nu}w_{\gamma-1,\nu_b+1}(k)u^k\le \sum_{k=0}^{K-K_\nu}w_{\gamma-1,\nu_b+1}(k){=(\gamma-2)\sum_{k=1}^{K+1-K_\nu}w_{\gamma,\nu_b}(k)}\le c_\nu$$ 
and from {\eqref{cneioneineonoen} and \eqref{weightnenoe}}:
\bee
&&\sum_{k=K-K_\nu+1}^{K-2}w_{\gamma-1,\nu_b+1}(k)u^k = {\sum_{k=K-K_\nu+1}^{K-2}(\gamma-2)w_{\gamma,\nu_b}(k+1)u^k}\\
&\le& c_\nu\sum_{k=K-K_\nu+1}^{K-2}{(\gamma-2)}\Gamma(\gamma-{2}-k)\gamma^{\nu_b+2-(\gamma-{2}-k)}\frac{1}{2^k}\le \frac{K^{c_\nu}}{2^K}\le c_\nu
\eee
and \eqref{estraminderwrror} is proved.\\\\
Then, provided $\Theta^*$ has been chosen large enough, we conclude from \eqref{estraminderwrror}, \eqref{neionenoenven}, \eqref{uniformbound} that for $0<b<b^*(\Theta^*,\delta)$ small enough, for all $u\in [0,u^*(\alpha_\gamma)]$:
\be\label{cenoenneeo:bbb}
|\Theta(u)-(-1)^{K-1}S_{\infty}\Thetam(u)|\leq \frac{S_\infty\Theta^*}{10}.
\ee
Therefore,
\be
\label{cenoenneeo}
{\frac{\Theta^*}{2}\leq \frac{(-1)^{K-1}\Theta(u^*(\alpha_\gamma))}{S_\infty}\le 2\Theta^*.}
\ee

\noindent{\bf step 2} Computation of $\Delta_1,\Delta_2$. 
We now unfold our changes of variables and show that, depending on the sign of $(-1)^{K-1}S_\infty$,  
we must have passed through either the green or the red curves. {To this end, we examine $\Delta_1$ and $\Delta_2$ under the assumption 
\bea\label{eq:aprioriestimateonPsiwhichisusedtoestimateNL1and2andholdsintheend}
|u|\leq \frac{3}{4}, \ \ |\Psi(u)|\leq 2\Theta_*.
\eea}
From {\eqref{suihfoenioeneoi}, \eqref{defphi}, \eqref{eq: Psi}} 
$$\left|\begin{array}{l}
\Phit=(1-u)\Psi\\
\Psi(u)=M_b(x)\Phi(u)\\
\Phi=bu\Theta\\
\end{array}\right.
$$
where $M_b$ is bounded and given by \eqref{deflkenrnal}. Moreover, from \eqref{eq: G_2 reexpressed} and \eqref{eq: Nonlinear Terms in F_2},
\bee
&& \Delta_1-c_-\Delta_2=\mathcal G_2\\
 & = & -b^2|\l_-|\psite\wte(c_+-c_-)u\left[u(1-u)bH_1(b,u)+(1+G_1(bu))\Phit+\NL_1(u,\Phit)\right]\\
 & = & -{b^2\wte^2\et_{20}}{(1+O(b))}u\left[u(1-u)bH_1(b,u)+(1+G_1(bu))bu(1-u)M_b\Theta+\NL_1(u,\Phit)\right]\\
 & = & -{b^3\wte^2\et_{20}}{(1+O(b))}u^2(1-u)\left[H_1(b,u)+(1+G_1(bu))M_b\Theta+\frac{\NL_1(u,\Phit)}{bu(1-u)}\right]
 \eee
 and from \eqref{eq: G_1 reexpressed} and \eqref{eq: Nonlinear Terms in F_1},
\bee
&&-\Delta_1+c_+\Delta_2= \mathcal G_1\\
& = & b^2\wte(c_+-c_-)|\mu_+|u\left[(1-u)\left[1+H_2(b,u)\right]+G_2(bu)\Phit+\NL_2(u,\Phit)\right]\\
& = & {-b^2\dt_{20}\wte^2}{(1+O(b))}u(1-u)\left[1+H_2(b,u)+G_2(bu)M_b\Phi+\frac{\NL_2(u,\Phit)}{1-u}\right]
\eee
From \eqref{deflkenrnal}, for $x=bu$ and $u\leq 1$,
\bea\label{eq:controlofMbforulessthan1}
M_b(x)=1+O(b),
\eea
and from \eqref{eq: Nonlinear Terms in F_1}, \eqref{eq: Nonlinear Terms in F_2}, and \eqref{eq:aprioriestimateonPsiwhichisusedtoestimateNL1and2andholdsintheend}
$$\left|\begin{array}{l}
H_1(b,u)=-(\Et_{11}+\Et_{30})+O(b)\\
G_1=b\Et_{11}u+O(b^2)\\
\NL_1=O({b^2u})
\end{array}\right.
,\quad
\left|\begin{array}{l}
H_2(b,u)=O(b)\\
G_2(x)=O(b)\\
\NL_2=O({b^2u})
\end{array}\right.
$$
It implies that, as long as $\Theta(u)\neq 0$, we have:
$$\left|\begin{array}{l}
\Delta_1-c_-\Delta_2= -\Theta(u)bu\left\{\wte^2{\et_{20}}b^2u(1-u)\left(1{+O\left(\frac{1}{\Theta(u)}\right)}+O(b)\right)\right\}\\
-\Delta_1+c_+\Delta_2=-\dt_{20}\wte^2{b^2}u(1-u)(1+O(b))
\end{array}\right.
$$
i.e.,
\be
\label{cnonwinwnwno}
\left|\begin{array}{l}
\Delta_1=\frac1{c_+-c_-}\wte^2{b^2}u(1-u)\left(1{+O\left(\frac{1}{\Theta(u)}+b\right)}\right)\left[-c_+\et_{20}\Phi(u)+|c_-||\dt_{20}|\right]\\[2mm]
\Delta_2={\frac{1}{c_+-c_-}}\wte^2{b^2}u(1-u)\left(1{+O\left(\frac{1}{\Theta(u)}+b\right)}\right)\left[{-}\et_{20}\Phi(u)-|\dt_{20}|\right]
\end{array}\right.
\ee
Note that we have used the signs, valid for all $d\geq 2$ and $0\leq d\leq \ell$:
$$\dt_{20}>0, \ \ \et_{20}>0,$$
see \eqref{signofdt20andet20bis}.\\

\noindent{\bf step 3} Touching {$\Delta_1=0$ or $\Delta_2=0$}. At $u=u^*(\alpha_\gamma)$, 
{by \eqref{cenoenneeo}} we have
$$\Phi(u^*)=bu^*{\Theta(u^*), \ \ \frac{(-1)^{K-1}\Theta(u^*)}{S_\infty}\geq \frac{\Theta^*}{2}}.$$ 
We now claim that there exists 
$$u^*\le U^*\le \frac34, \ \ \Phi(U^*)={(-1)^{K-1}}\Theta^*$$ 
Then, from \eqref{cnonwinwnwno}, {since $\dt_{20}\neq 0$ and $\et_{20}> 0$ by \eqref{signofdt20andet20bis}, and since $\Phi(0)=0$,} we must have crossed $\Delta_1=0$ or $\Delta_2=0$, depending on {wether $K$ is even or odd}.\\
{Assume $K$ odd, so that from \eqref{cenoenneeo} $\Theta(u_*)>0$. The case $K$ even can be treated similarly.} Since $u\le \frac 34$ and $M_b(u)=1+{O(b)}$:
$$ \Psi(u^*)=bu^*\Theta(u^*){(1+O(b))}.$$ 
Then, on the interval $u\in[u^*, {U^*}]$ with $U^*\le \frac 34$ and 
\be
\label{botboudnpsishuert}
\frac{{|S_{\infty}|}bu^{{*}}\Theta^*}{{4}}<\Psi(u)\le 2\Theta^*
\ee 
we have for $b<b^*(\Theta^*)$ from {\eqref{eq: H_i}}, \eqref{eq: NL_i}, \eqref{eq: Nonlinear Terms in F_2}, {\eqref{eq: Nonlinear Terms in F_1}}:
$$\left|\begin{array}{l}
\left|H_1(b,u)+(\Et_{11}+\Et_{30})\right|+|H_2|+|G_2|+|G_1|\le C_{\Theta^*}b\\
|\NLt_2|+|\frac{\NLt_1}{x}|\le C_{\Theta^*}b\\
\end{array}\right.
$$
We insert this into \eqref{eq: Quasilinear Equation} and conclude from \eqref{botboudnpsishuert}, provided $\Theta^*>0$ has been chosen large enough,
\bea\label{eq:ODEinequalityforPsithatshowsinparitcularnochangeofsign}
u\Psi'\geq \frac{\gamma}{4}\Psi.
\eea
Therefore,
$$\Psi(u)\ge \Psi(u^*)\left(\frac{u}{u^*}\right)^{\frac{\gamma}{4}}\geq \frac{{|S_\infty|}b{u^*}\Theta^*}{{4}}\left(\frac{u}{u^*}\right)^{\frac{\gamma}{4}}\ge \Theta^*$$ 
for 
$$\left(\frac{u}{u^*}\right)^{\frac{\gamma}{4}}\ge {\frac{4}{|S_\infty|bu^*\Theta^*}}, \ \ u\ge u^*\left({\frac{4}{|S_\infty|bu^*\Theta^*}}\right)^{\frac{4}{{\gamma}}}=u^*\left(1+O({b|\log b|})\right)$$ 
{where we have used the fact that $u^*\geq b\delta$ in view of \eqref{firsteataimteboundary} \eqref{vnieneoneonevnonen}}.
Since $u^*\le \frac 12$, we established that the contact happens before $u=\frac 34$.
\end{proof}

\subsection{{Exit on the left}}

{Below, we obtain an analog of Lemma \ref{lemmaexitleft} on the left, albeit in a significantly more restricted range of $\alpha_\gamma$ in $(0,1)$.} We assume \eqref{neinenvoneonve}.

\begin{lemma}[Exit on the left]
\label{lemmaexitright}
Pick universal constants $\frac{1}{\delta},\Theta^*\gg1 $ large enough as in Lemma \ref{propositionfundamental}, then  for all $0<b<b^*(\Theta^*,\delta)$ small enough and $\alpha_\gamma$ given by 
\be
\label{boudnaryleyeralphagright}
\alpha_\gamma=\frac{K^{\nu_b+3}\big(b\sqrt{\delta}\big)^{K-1}}{\Theta^*}
\ee
or 
\be
\label{boudnaryleyeralphagrightbis}
\alpha_\gamma=1-\frac{K^{\nu_b+3}\big(b\sqrt{\delta}\big)^{K-1}}{\Theta^*},
\ee
the $C^\infty$ solution exits on the left of $P_3$ at $u=U_*$, where $-\frac{3}{4}<U_*<0$, by crossing $\Delta_2=0$ in the case \eqref{boudnaryleyeralphagright} and by crossing $\Delta_1=0$ in the case \eqref{boudnaryleyeralphagrightbis}.
\end{lemma}

\begin{proof}[Proof of Lemma \ref{lemmaexitright}] For $\alpha_\gamma$ given by \eqref{boudnaryleyeralphagright} or \eqref{boudnaryleyeralphagrightbis} we only need to consider $u$ in the the range $-b\leq u\leq 0$. In that range, the proof of  Lemma \ref{lemmaexitright} follows very closely the one of Lemma \ref{lemmaexitleft} for the case $0\leq u\leq b$. The $C^\infty$ regularity at the left of $P_3$ all the way to $u=0$ together with the property \eqref{neoneneonovenoev} follow again from the explicit integral representation of the remainder function. We focus on the exit behavior.\\

{\noindent{\bf step 1} Bounds on $M_j$. For $-b\leq u\leq b$, we have 
$$(1-u)^K=e^{K\log(1-u)}=e^{O(1)}$$
so that the cases $0\leq u\leq b$ and $-b\leq u\leq 0$ can be treated similarly in  the definition of $\TT$ and, as a consequence, 
in $M_j$. In particular, the proof of \eqref{lowerobundzerobis}, \eqref{bounditerate}, \eqref{vniovnioneneneo} and \eqref{lowerobundzerobis} obtained for $0\leq u\leq b$ immediately extends to the case $-b\le u\le  0$, i.e., we have for $1\le j\le 5$ and $-b\le u\le  0$, 
\be
\label{bounditerate:right}
\left\|\frac{\T (u^jM_0)}{M_0}\right\|_{L^\infty(-b\le u\le  0)}\le \frac{c_\nu}{b},\quad \frac{|uM_0'|}{M_0}\leq \frac{c_\nu}b. 
\ee 
and for  $0\le j\le 5$ and $-b\le u\le  0$
\be
\label{lowerobundzerobis:right}
c_{\nu,1}\le \frac{M_j(u)}{\Gamma(\alpha_\gamma)\Gamma(1-\alpha_\gamma)K^{\nu_b+j+4-\alpha_\gamma}u^{K+j}}\le c_{\nu,2}.
\ee}

\noindent{\bf step 2} Estimate on $\Thetam$. Recall from \eqref{cneiovnenvenenne} that for $-b\de^{\frac{1}{3}}\leq u\leq -b\delta$,
\bea
\label{cneiovnenvenenne:right}
&&\frac{\Thetam(u)}{ \Gamma(\alpha_\gamma)\Gamma(1-\alpha_\gamma)K^{\nu_b+3-\alpha_\gamma}u^{K-1}}\\
\nonumber &= & \left[1+o_{b\to 0}(1)\right]\left[\frac{1}{\Gamma(1-\alpha_\gamma)}+\frac{K+\nu_b+2}{\Gamma(2-\alpha_\gamma)}u\right]+(u K)^{\alpha_\gamma}\frac{M_1(u)}{u^{K-1}\Gamma(\alpha_\gamma)\Gamma(1-\alpha_\gamma)K^{\nu_b+3}u^{\alpha_\gamma}}\\
\nonumber &= & \left[1+o_{b\to 0}(1)\right]\left[\frac{1}{\Gamma(1-\alpha_\gamma)}+\frac{K+\nu_b+2}{\Gamma(2-\alpha_\gamma)}u\right]+O\left(c_\nu\de^{\frac{2}{3}}\right).
\eea
where the last equality follows from \eqref{lowerobundzerobis:right}.\\

\noindent{\bf step 3} Boundary layer. In view of \eqref{cneiovnenvenenne:right}, we easily obtain the following analog of the first two cases of Lemma \ref{lemmanekvneneoneon} for $\alpha_\gamma$ given by  \eqref{boudnaryleyeralphagright} or \eqref{boudnaryleyeralphagrightbis}. Pick universal constants $\frac{1}{\delta},\Theta^*\gg1$ then  for all $0<b<b^*(\Theta^*,\delta)$, we have:\\
\noindent\underline{first layer}: if $\alpha_\gamma$ is given by  \eqref{boudnaryleyeralphagright}, then there exist a  solution to 
\be
\label{neknvonenneudatsgeeaa:right}
\Thetam(u^*(\alpha_\gamma))=(-1)^{K-1}\Theta^*,  
\ee 
satisfying the following bounds
\be
\label{valueapprocimate:right}
\Gamma(\alpha_\gamma)\Gamma(1-\alpha_\gamma)K^{\nu_b+3-\alpha_\gamma}(u^*(\alpha_\gamma))^{K-1}=\Theta^*e^{O_{\nu}(\delta^{-1})}
\ee
and
\be
\label{firsteataimteboundary:right}
u^*(\alpha_\gamma)=-\sqrt{\delta} b(1+o_{b\to 0}(1)).
\ee
\noindent\underline{second layer}: if $\alpha_\gamma$ is given by  \eqref{boudnaryleyeralphagrightbis}, then there exist a  solution to 
\be
\label{neknvonenneudatsgeeaa:right:bis}
\Thetam(u^*(\alpha_\gamma))=(-1)^{K}\Theta^*,  
\ee 
satisfying  \eqref{firsteataimteboundary:right} and 
 \be
\label{valueapprocimate:2ndcase:right}
\Gamma(\alpha_\gamma)\Gamma(1-\alpha_\gamma)K^{\nu_b+3-\alpha_\gamma}\frac{K+\nu_b+2}{\Gamma(2-\alpha_\gamma)}(u^*(\alpha_\gamma))^{K}=\Theta^*e^{O_{\nu}(\delta^{-1})}.
\ee
 
{\noindent{\bf step 4} Estimate on $r_\mathcal G$. In view of \eqref{lowerobundzerobis:right} and \eqref{firsteataimteboundary:right}, we have for $u^*(\alpha_\gamma)$ defined in step 3
\be\label{uniformbound:right}
\forall u\in [u^*(\alpha_\gamma),0], \ \ 0\le |M_0(u)|\le c_\nu |M_0(u^*(\alpha_\gamma))|\leq \frac{c_\nu\Theta^*}{\sqrt{\delta}}=C_{\Theta^*,\delta}.
\ee
Also, proceeding as in the proof of \eqref{noennoeoenvoen} and \eqref{noennoeoenvoen:0:0}, we obtain the following analogs
\be
\label{noennoeoenvoen:right}
\forall u\in[u^*(\alpha_\gamma),0], \ \  \Gamma(\alpha_\gamma)\Gamma(1-\alpha_\gamma)K^{\nu_b+3-\alpha_\gamma}|u|^{K-1}\le C_{\Theta^*,\delta}.
\ee
and
\be
\label{noennoeoenvoen:0:0:right}
\forall u\in[u^*(\alpha_\gamma),0], \ \  \Gamma(\alpha_\gamma)\Gamma(1-\alpha_\gamma)K^{\nu_b+2-\alpha_\gamma}|u|^{K-2}\le C_{\Theta^*,\delta}.
\ee}
The estimate \eqref{neionenoenven} holds for $-b\leq u\leq 0$, i.e. for  $-b\leq u\leq 0$, and we now claim:
 \be
\label{neionenoenven:right}
\frac{|r_{\mathcal G}|}{M_0}+\frac{|\T r_{\mathcal G}|}{M_0}<\sqrt{b}.
\ee
Indeed, the proof of \eqref{neionenoenven} for the range $0\leq u\leq b$ does not use the sign of $u$, and all estimates hold by replacing everywhere $u$ with $|u|$. Using also \eqref{uniformbound:right} \eqref{noennoeoenvoen:right} \eqref{noennoeoenvoen:0:0:right}, the proof immediately extends to the range $-b\leq u\leq 0$. Thus, \eqref{neionenoenven:right} holds for $-b\leq u\leq 0$.\\

\noindent{\bf step 6} Conclusion. Recall \eqref{expressionnonlnieanrtmer}. Then, provided $\Theta^*$ has been chosen large enough, we conclude from \eqref{estraminderwrror} (with $u$ replaced by $|u|$), \eqref{neionenoenven:right}, \eqref{uniformbound:right} that for $0<b<b^*(\Theta^*,\delta)$ small enough, for all $u\in [u^*(\alpha_\gamma),0]$:
$$|\Theta(u)-(-1)^{K-1}S_{\infty}\Thetam(u)|\leq \frac{S_\infty\Theta^*}{10}.$$ 
Therefore, if $\alpha_\gamma$ is given by  \eqref{boudnaryleyeralphagright}
\be
\label{cenoenneeo:right}
\frac{\Theta^*}{2}\leq \frac{\Theta(u^*(\alpha_\gamma))}{S_\infty}\le 2\Theta^*,
\ee
and, if $\alpha_\gamma$ is given by  \eqref{boudnaryleyeralphagrightbis},
\be
\label{cenoenneeo:bis:right}
\frac{\Theta^*}{2}\leq \frac{-\Theta(u^*(\alpha_\gamma))}{S_\infty}\le 2\Theta^*.
\ee
We now note that \eqref{cnonwinwnwno} holds independently of the sign of $u\in (-1,1)$. If $\alpha_\gamma$ is given by  \eqref{boudnaryleyeralphagright}, then, by \eqref{cenoenneeo:right},
$$\Phi(u^*)=bu^*\Theta(u^*), \ \ \frac{\Theta(u^*)}{S_\infty}\geq \frac{\Theta^*}{2},$$ 
 and, if $\alpha_\gamma$ is given by  \eqref{boudnaryleyeralphagrightbis}, then, by \eqref{cenoenneeo:bis:right},
$$\Phi(u^*)=bu^*\Theta(u^*), \ \ \frac{-\Theta(u^*)}{S_\infty}\geq \frac{\Theta^*}{2}.$$ 
The rest of the argument of step 3 in the proof of Lemma \ref{lemmaexitleft} extends to the case $u<0$. Therefore,\\
\noindent\underline{first layer}: if $\alpha_\gamma$ is given by  \eqref{boudnaryleyeralphagright}, there exists $U^*(\alpha_\gamma)$ such that
$$-\frac34\le U^*\le u^*, \ \ \Phi(U^*)=-\Theta^*,$$
and the smooth solution crosses $\Delta_2=0$ for $u\ge -3/4$.\\
\noindent\underline{second layer}: if $\alpha_\gamma$  is given by  \eqref{boudnaryleyeralphagrightbis}, there exists $U^*(\alpha_\gamma)$ such that
$$-\frac34\le U^*\le u^*, \ \ \Phi(U^*)=\Theta^*,$$
and the smooth solution crosses $\Delta_1=0$ for $u\ge -3/4$.
This concludes the proof of Lemma \ref{lemmaexitright}.
\end{proof}

\begin{corollary}[Behaviour at the first boundary layer]
\label{cor: First Boundary Layer}
Choose $\frac{1}{\delta}$, $\Theta^*\gg1$ and $0<b<b^*(\Theta^*,\delta)$ as in Lemma \ref{lemmaexitright}. Recall from Lemma \ref{lemmaexitright} that for $\alpha_\gamma$ as in \eqref{boudnaryleyeralphagright}, there exists $0<U_*<\frac{3}{4}$ such that the $C^\infty$ solution crosses $\Delta_2=0$ at $u=U_*$. Then the $P_1-P_3$ trajectory is in the left of the unique $C^\infty$ at $P_3$ solution.
\end{corollary}
\begin{proof}
If we denote by $(\sigma_*,w_*)=(\sigma_3+ \Sigma_*,w_3+W_*)$ the point corresponding to $u=U_*$, then from \eqref{eq: Quasilinear Variables} and \eqref{defphi}, we infer
\be
\label{eq: Original Variables}
\left |\begin{array}{l} \Wt=-b\wt_{\hskip -.1pc\peye}u\\\Sigmat=b^2\wt_{\hskip -.1pc\peye}\psit_{\hskip -.1pc\peye} u(u+(1-u)M_b(x)\Phi)\end{array}\right.
\ee
when evaluated at $u=U_*$, 
\be
\label{eq: Point of contact}
\left|\begin{array}{l}
\Sigma_*=\Wt+\Sigmat=-b\wt_{\hskip -.1pc\peye}U_*(1+\mathcal O(b)),\\
W_*=c_- \Wt+c_+\Sigmat= -bc_-\wt_{\hskip -.1pc\peye}U_*(1+\mathcal O(b)).
\end{array}\right.
\ee
Let $x=x_*$ at $u=U_*$. Then, in view of the sign of $\Delta_1$, $\Delta_2$, we infer
$$
\frac{d\sigma}{dx}\bigg|_{x=x_*}=0,\quad \frac{dw}{dx}\bigg|_{x=x_*}<0,
$$
one can parameterize the solution curve $\sigma(w)$ near $w=w_*$. Let 
$$
w_{**}=\min \Big\{w> w_*\Big|\Delta_2(\sigma,w)=0\Big\} \in (w_*,\infty].
$$
Then, by uniqueness of the solution, it suffices to prove that $\sigma_{P_1-P_3}(w)<\sigma(w)$ for some $w_*<w<w_{**}$ where $\sigma_{P_1-P_3}(w)$ is the parameterization of the $P_1-P_3$ trajectory. Note that if $w_{**}=\infty$, we're done since $\sigma(w=1)>0$ while $\sigma_{P_1-P_3}$ terminates at $P_1=(0,1)$. Otherwise, from \eqref{eq: Point of contact},
$$
\forall w_*<w<w_{**} ,\quad \frac{d\sigma}{dw}>0 \quad \Rightarrow \quad \sigma(w_{**})>\sigma(w_*)= \sigma_3+\mathcal O(b).
$$
In particular, the trajectory of the solution for $w_*<w<w_{**}$ stays on the right of the curve $w_2^{\pm}(\sigma)$ where $\Delta_2=0$. Thus,  
$$
w_{**}>w_2^+(\sigma_*)=w_2^+(\sigma_3)+\mathcal O(b).
$$
In view of \eqref{eq: P_3^hat} we have the limit 
$$
\lim_{b\rightarrow 0} w_2^+(\sigma_3)>1.
$$
Thus, $w_{**}>1$ for $b\ll 1$ so the point $(\sigma(w=1), 1)$ is on the right of the curve $w_2^{\pm}(\sigma)$ which is on the right of $P_1$.
\end{proof}

\subsection{Proof of Theorem \ref{thmmain}}

We are now in position to conclude the proof of Theorem \ref{thmmain}.\\

\noindent{\bf step 1} Continuous deformation of $\alpha_\gamma$. Let $K$ be even and large enough, we claim that there exists $\alpha^K_\gamma$ in $(0,1)$ such that the $C^\infty$ solution 
$\Phi[K, \alpha^K_\gamma](u)$ coincides with the unique $P_1-P_3$ solution (to the left of $P_3$)
and exits to the right of $P_3$ by crossing $\Delta_2=0$ before reaching $P_2$ (and, as a result, extends to $P_4$.).\\

\noindent Indeed, let $\frac{1}{\delta},\Theta^*\gg1 $ large enough and $0<b<b^*(\Theta^*,\delta)$ small enough as in Lemma \ref{lemmaexitleft}, and $\alpha_\gamma$ in the range \eqref{boudnaryleyeralphag}. Assume also that $K$ is even. Then, in view of Lemma \ref{lemmaexitleft}, 
 the $C^\infty$ solution exits on the right of $P_3$ by crossing $\Delta_2=0$ before $u=\frac{3}{4}$. Consider then the $C^\infty$ solution $\Phi[K, \alpha_\gamma](u)$ on the left of $P_3$. Let also $\Phi_{P_1-P_3}[K, \alpha_\gamma](u)$ denote the unique solution, constructed earlier, which corresponds to the $P_1-P_3$ trajectory. 
 Then, let
\bea
F(\alpha_\gamma) &:=& \Phi[K, \alpha_\gamma]\left(-\frac{3}{4}\right) - \Phi_{P_1-P_3}[K, \alpha_\gamma]\left(-\frac{3}{4}\right). 
\eea
Then, for $\alpha_\gamma$ given by \eqref{boudnaryleyeralphagright}, we have from  Corollary \ref{cor: First Boundary Layer} that for any $u<0$, $\Phi_{P_1-P_3}[K, \alpha_\gamma](u)$ is located to the left of the curve $\Phi[K, \alpha_\gamma](u)$ the $C^\infty$-solution crossing $\Sigma=0$. Hence, in $[U_*,0]$, the curve $\Phi_{P_1-P_3}[K, \alpha_\gamma](u)$ is located below the curve $\Phi[K, \alpha_\gamma](u)$. Since the line $u=$ constant has positive gradient on $(\sigma,w)$-plane, when evaluated at $u\in[U_*,0]$, in view of \eqref{eq: Original Variables},
$$
0< \Sigma-\Sigma_{P_1-P_3}=b^2\wt_{\hskip -.1pc\peye}\psit_{\hskip -.1pc\peye} u(1-u)M_b(x)(\Phi-\Phi_{P_1-P_3}).
$$
By uniqueness of the solution, this extends to $u\in[-\frac{3}{4},0]$. Evaluate at $u=-\frac{3}{4}$, we infer
\bea
F\bigg(\frac{K^{\nu_b+3}\big(b\sqrt{\delta}\big)^{K-1}}{\Theta^*}\bigg)<0.
\eea
Also, for $\alpha_\gamma$ given by \eqref{boudnaryleyeralphagrightbis}, since $\Phi_{P_1-P_3}[K, \alpha_\gamma](u)$ lies in the set $\{\Delta_1\le0\}$ for any $u<0$, and, since $\Phi[K, \alpha_\gamma](u)$ has crossed $\Delta_1=0$ before $u=-3/4$ and cannot cross $\Delta_1=0$ twice, in view of \eqref{cnonwinwnwno} we deduce
\bea
F\bigg(1-\frac{K^{\nu_b+3}\big(b\sqrt{\delta}\big)^{K-1}}{\Theta^*}\bigg) >0.
\eea
Continuous dependence of the ODE on the parameter $\alpha_\gamma\in (0,1)$ implies the continuity of $F$. We then 
infer by the mean value theorem the existence of  $\alpha^K_\gamma$ such that
\bea
F\left(\alpha^K_\gamma\right)=0, \ \ \alpha^K_\gamma\in \bigg(\frac{K^{\nu_b+3}\big(b\sqrt{\delta}\big)^{K-1}}{\Theta^*},  1-\frac{K^{\nu_b+3}\big(b\sqrt{\delta}\big)^{K-1}}{\Theta^*}\bigg).
\eea
Then, by the uniqueness of solutions to the ODE at $u=3/4$, $\Phi[K, \alpha^K_\gamma]$ must coincide with $\Phi_{P_1-P_3}[K, \alpha^K_\gamma]$. 

To the right of $P_3$, with $\alpha^K_\gamma$ in the range \eqref{boudnaryleyeralphag}, we have that $\Phi[K, \alpha^K_\gamma]$ crosses $\Delta_2=0$ before reaching $P_2$. 

Thus, we have obtained for any even $K$ large enough the existence of $\alpha^K_\gamma$ in $(0,1)$ such that the smooth profile  $\Phi[K, \alpha^K_\gamma](u)$ coincides with the $P_1-P_3$ solution to the left of $P_3$ and exits on the left of $P_3$ by crossing $\Delta_2=0$ before reaching $P_2$. The constructed solution is $C^\infty$ to the right and the left of $P_3$, with derivatives satisfying \eqref{neoneneonovenoev} on both sides.\\

\noindent{\bf step 2} Conclusion. Since the curve crosses $\Delta_2=0$ for $\sigma_3<\sigma<\sigma_2$, it is attracted to $P_4$ by Lemma \ref{lem: Solutions Near P_4}. It remains to show that in the $(\sigma(x),w(x))$ parametrization, the point $P_3$ is reached in finite time. Indeed, $$\lim_{\Sigma \to 0}\frac{W}{\Sigma}=c_-$$ and thus, from \eqref{eq: Delta_1 in W, Sigma}, \eqref{eq: lambda_pm in c}, \eqref{eq: Slopes at P_3}:
$$\frac{d\Sigma}{dx}=\frac{d\sigma}{dx}=-\frac{\Delta_2}{\Delta}=-\frac{(c_2c_-+c_4)\Sigma(1+o(1))}{-2\sigma_3(1+c_-)\Sigma(1+o(1))}=\frac{\l_+}{2\sigma_3(1+c_-)}(1+o(1))$$
which proves the claim. The resulting $C^\infty$ solution corresponds to the global $P_1-P_3-P_4$ trajectory.

\begin{appendix}
\section{Facts related to the $\Gamma$ function}

\noindent\underline{Asymptotics}. We recall Stirling's formula
\be
\label{striling}
\Gamma(x+1)=(1+o_{x\to +\infty}(1))\left(\frac{x}{e}\right)^x\sqrt{2\pi x}.
\ee
Let $$|x|\lesssim 1\ll \gamma,$$ this yields:
\bea
\label{aymptoticratio}
\nonumber &&\frac{\Gamma(\gamma+x+1)}{\Gamma(\gamma+1)}=(1+o_{\gamma\to +\infty}(1))\frac{\left(\frac{\gamma+x}{e}\right)^{\gamma+x}\sqrt{2\pi (\gamma+x)}}{\left(\frac{\gamma}{e}\right)^\gamma\sqrt{2\pi \gamma}}\\
\nonumber & = & (1+o_{\gamma\to +\infty}(1))\frac{1}{e^{x}}e^{(\gamma+x)\left[\log \gamma +\frac{x}{\gamma}+O\left(\frac{1}{\gamma^2}\right)\right]-\gamma\log \gamma}\\
& = & (1+o_{\gamma\to +\infty}(1))\frac{1}{e^{x}}e^{x\log \gamma+x+O_x\left(\frac{1}{\gamma}\right)}=(1+o_{\gamma\to +\infty}(1))\gamma^x.
\eea

\noindent\underline{Value on $\Bbb R\backslash \Bbb N_-$}.  

\begin{lemma}[Value of $\Gamma(x)$ for $x\in \Bbb R\backslash \Bbb N_-$]
Let $$x=-K_x+\alpha_x, \ \ K_x\in \Bbb N^*, \ \ 0<\alpha_x<1$$ then
\be
\label{formulafdebasenegatif}
\Gamma(x)=(-1)^{K_x}\frac{\Gamma(\alpha_x)\Gamma(1-\alpha_x)}{\Gamma(1-x)}.
\ee
\end{lemma}

\begin{proof} By definition
$$\Gamma(x)=\frac{\Gamma(x+1)}{x}=\frac{\Gamma(x+J+1)}{\Pi_{j=0}^J(x+j)}$$ and thus, with $J=K_x-1$,
\bee
\Gamma(x)&=&\frac{\Gamma(-K_x+\alpha_x+K_x-1+1)}{\Pi_{j=0}^{K_x-1}\Gamma(\alpha_x-K_x+j)}=\frac{\Gamma(\alpha_x)}{\Pi_{m=1}^{K_x}(\alpha_x-m)}=(-1)^{K_x}\frac{\Gamma(\alpha_x)}{\Pi_{m=1}^{K_x}(-\alpha_x+m)}\\
& = & (-1)^{K_x}\frac{\Gamma(\alpha_x)}{\frac{\Gamma(-\alpha_x+K_x+1)}{\Gamma(-\alpha_x+1)}}=(-1)^{K_x}\frac{\Gamma(\alpha_x)\Gamma(1-\alpha_x)}{\Gamma(1-x)}.
\eee
\end{proof}

\section{Study of the weight $w_{\gamma,\nu_b}$ for $k\le K$}

 We derive estimates and convolution bounds for the weight $w_{\gamma,\nu_b}(k)$ given by \eqref{defweight}. Proof of the results in this section can be found in \cite{MRRSprofile}.

\begin{lemma}[Summation bound]
\label{lemmasummationbound}
For some $c_{\nu,a}>0$, $K_\nu\gg1$ and all $0<b<b^*(\nu)$ we have the following uniform bounds:
\be
\label{summationphi}
\sum_{k=1}^{K-K_{\nu}}w_{\gamma,\nu_b}(k)\le c_{\nu,a} b.
\ee
and for ${K}-K_\nu{+1}\le k\le K{-1}$:
\be
\label{weightnenoe}
w_{\gamma,\nu_b}(k)=e^{O_\nu(1)}\Gamma(\gamma-1-k)\gamma^{\nu_b+2-(\gamma-1-k)}.
\ee
\end{lemma}

We now turn to the convolution type estimate on the weight.

\begin{lemma}[Convolution estimate]
\label{lemamconvoutitio}
There exist universal constants $c_{\nu,a}>0$, $0<b^*(\nu,a)\ll 1$ such that the following holds: for all $0<b<b^*$, for all $0\le k\le K$, 
\be
\label{tobeprovoeonorbibib}
\sum_{k_1+\dots+k_j=k}w_{\gamma,\nu}(k_1)\dots w_{\gamma,\nu}(k_j)\leq c^j_{\nu,a}w_{\gamma,\nu}(k). 
\ee
\end{lemma}

\section{Coefficients and eigenvalues of \eqref{eq: Autonomous System} at $P_3$}
\label{app: Slopes and Eigenvalues}
For $1<r<r_+(d,\ell)$,
\be
\label{eq: Coefficients of Autonomous System}
\left|\begin{array}{l}
d_{20}=3w_3-(r+1)\\
d_{11}=-2d\sigma_3\\
d_{02}=\ell(r-1)-dw_3
\end{array}\right.
,\quad
\left|\begin{array}{l}
e_{20}=\frac{\sigma_3(\ell+d-1)}{\ell}\\
e_{11}=\frac{2w_3(\ell+d-1)-(\ell+d+\ell r-r)}{\ell}\\
e_{02}=-3\sigma_3\\
e_{21}=\frac{\ell+d-1}{\ell}.
\end{array}\right.
\ee

\begin{lemma}[Critical values of the slopes at $P_3$]
\label{lem: Slopes at P_3}
 Let $$\ell<d, \ \ r=r_+(d,\ell)=1+\frac{d-1}{(1+\sqrt{\ell})^2},$$ then at $P_3$:
 \be
\label{eq: Slopes at P_3}
\left|\begin{array}{l}
\sigma^\infty_3=\frac{1}{1+\sqrt{\ell}}\\
w^\infty_3=\frac{\sqrt{\ell}}{1+\sqrt{\ell}}
\end{array}\right.,\quad
\left|\begin{array}{l}
c^\infty_1=-\frac{2\sqrt{\ell}(d+\sqrt{\ell})}{(1+\sqrt{\ell})^3}\\
c^\infty_2=-\frac{2}{(1+\sqrt{\ell})^2}\\
c^\infty_3=-\frac{2\sqrt{\ell}(d+\sqrt{\ell})}{(1+\sqrt{\ell})^3}\\
c^\infty_4=-\frac{2}{(1+\sqrt{\ell})^2}
\end{array}\right.,\quad
\left|\begin{array}{l}
c^\infty_-=-1\\
c^\infty_+=\frac{\sqrt{\ell}(d+\sqrt{\ell})}{1+\sqrt{\ell}}\\
\l^\infty_+=0\\
\l^\infty_-=-\frac{2\left[\ell+(d+1)\sqrt{\ell}+1\right]}{(1+\sqrt{\ell})^3}.
\end{array}\right.
\ee
Moreover,
\be
\label{eq: mu_+^infty}
\mu^\infty_+=(\pa_b\l_+)_{b=0}=\frac{2\sqrt{d-1}(\ell-d)}{(1+\sqrt{\ell})\ell^{\frac{1}{4}}(1+(d+1)\sqrt{\ell}+\ell)}<0.
\ee
\end{lemma}

\begin{proof}
This follows from a direct computation. \eqref{eq: mu_+^infty} has been computed with Mathematica.
\end{proof}

We compute explicitly the sign of the coefficients $\dt_{20}$, $\et_{30}$ and $\nu$ at $r^+(d,\ell)$.

\begin{lemma}
\label{lem: Limiting Values of d, e}
For all $d\geq 2$ and $\ell<d$, we have
\be\label{signofdt20andet20bis}
\dt^\infty_{20}>0, \quad \et^\infty_{20}>0, \quad \nu_\infty>0.
\ee
\end{lemma}

\begin{proof} We collect the values
\be
\label{valuedijbis}
\left|\begin{array}{l}
d^\infty_{20}=\frac{\ell-\sqrt{\ell}-d-1}{(1+\sqrt{\ell})^2}\\
d^\infty_{11}=\frac{-2d}{1+\sqrt{\ell}}\\
d^\infty_{02}=\frac{-\ell-d\sqrt{\ell}}{(1+\sqrt{\ell})^2}\\
\end{array}\right.
,\quad
\left|\begin{array}{l}
e^\infty_{20}=\frac{\ell+d-1}{\ell(1+\sqrt{\ell})}\\
e^\infty_{11}=\frac{-2}{1+\sqrt{\ell}}\\
e^\infty_{02}=-\frac{3}{1+\sqrt{\ell}}\\
e^\infty_{21}=\frac{\ell+d-1}{\ell}.
\end{array}\right.
\ee

We compute directly
$$
\left|\begin{array}{ll}
\dt^\infty_{20}=-\frac{(d-1)(\ell-d)}{\sqrt{\ell}(1+\sqrt{\ell})^2}\\
\et^\infty_{20}=\frac{(d-1)(\ell+1)}{\ell(1+\sqrt{\ell})}.
\end{array}\right.
$$
and 
\bea
\label{eq: nu_infty at r^+}
 \nu_\infty &=& \frac{2}{(d-1)(\ell-d)^2}\Bigg[d^3(\ell+2)+d^2(\ell^2+2\ell^{\frac{3}{2}}-2\ell+2\sqrt{\ell}-2)\\
\nonumber&&+d(2\ell^{\frac{5}{2}}+5\ell^2+4\ell^{\frac{3}{2}}+7\ell+2\sqrt{\ell}+1)+\ell(\ell^2+2\ell^{\frac{3}{2}}+\ell+2\sqrt{\ell}+1)\Bigg]>0
\eea
and the claim is proved.
\end{proof}

\section{Expansion of the functionals in \eqref{eq: Quasilinear Equation}}

In this Appendix we collect all the formulas for all the terms appearing in \eqref{eq: Quasilinear Equation}. We recall $x=bu$.\\

\noindent\underline{$\Dt_{ij},\Et_{ij}$ coefficients} Recall the definition \eqref{eq: Coefficients of Autonomous System 2} of $\dt_{ij}$ and $\et_{ij}$. We define
\be
\label{eq: D and E}
\Dt_{ij}=\frac{\dt_{ij}\wte^{i-1}\sigmate^j}{|\mu_+|(c_+-c_-)},\quad \Et_{ij}=\frac{\et_{ij}\wte^{i+2}\sigmate^{j-1}}{|\l_-|(c_+-c_-)}
\ee
 Note that the coefficients $\Dt_{ij},\Et_{ij}$ have a well defined limit $\Dt^\infty_{ij},\Et^\infty_{ij}$ as $b\to 0$ from \eqref{eq: w, psi at P_3}.

\noindent\underline{Polynomials $H_1,H_2$}
\be
\label{eq: H_i}
\left|\begin{array}{l}
H_1(b,u)=\sum_{j=0}^3b^jH_{1,j}(x)\\
H_2(b,u)=\sum_{j=0}^3 b^jH_{2,j}(x)
\end{array}\right.
\ee
with
\be
\label{eq: H_ij}
\left|\begin{array}{l}
H_{1,0}(x)=-(\Et_{11}+\Et_{30})+(\Et_{02}+\Et_{21})x-\Et_{12}x^2+\Et_{03}x^3\\
H_{1,1}(x)=(\Et_{02}+\Et_{21})-\Et_{12}x+\Et_{03}x^2\\
H_{1,2}(x)=-\Et_{12}+\Et_{03}x\\
H_{1,3}(x)=\Et_{03},\\
H_{2,0}(x)=(\Dt_{11}+\Dt_{30})x-(\Dt_{02}+\Dt_{21})x^2+\Dt_{12}x^3-\Dt_{03}x^4\\
H_{2,1}(x)=-(\Dt_{02}+\Dt_{21})x+\Dt_{12}x^2-\Dt_{03}x^3\\
H_{2,2}(x)=\Dt_{12}x-\Dt_{03}x^2\\
H_{2,3}=-\Dt_{03}x.
\end{array}\right.
\ee
\noindent\underline{Polynomials $G_1,G_2$}
\be
\label{eq: G_i}
\left|\begin{array}{l}
G_1(b,u)=\Et_{11}x-(2\Et_{02}+\Et_{21})x^2+2\Et_{12}x^3-3\Et_{03}x^4\\
G_2(b,u)=-\Dt_{11}x+(2\Dt_{02}+\Dt_{21})x^2-2\Dt_{12}x^3+3\Dt_{03}x^4
\end{array}\right.
\ee
\noindent\underline{Nonlinear terms}
\be
\label{eq: NL_i}
\left|\begin{array}{l}
\NLt_1=\sum_{j=0}^2b^j\NLt_{1j}\\
\NLt_2=\sum_{j=0}^2b^j\NLt_{2j}
\end{array}\right.
\ee
with
$$\left|\begin{array}{l}
\NLt_{10}=-xM_{11}\Psi^2+x^2M_{12}\Psi^3\\
\NLt_{11}=M_{11}\Psi^2-2xM_{12}\Psi^3\\
\NLt_{12}=M_{12}\Psi^3
\end{array}\right., \ \  \left|\begin{array}{l}
M_{11}= -\Et_{02}x+\Et_{12}x^2-3\Et_{03}x^3\\
M_{12}=-\Et_{03}x^2
\end{array}\right.
$$
$$\left|\begin{array}{l}
\NLt_{20}=-xM_{21}\Psi^2+x^2M_{22}\Psi^3\\
\NLt_{21}=M_{21}\Psi^2-2xM_{22}\Psi^3\\
\NLt_{22}=M_{22}\Psi^3
\end{array}\right., \ \ \left|\begin{array}{l}
M_{21}=\Dt_{02}x-\Dt_{12}x^2+3\Dt_{03}x^3\\
M_{22}=\Dt_{03}x^2.
\end{array}\right.
$$
\end{appendix}

%%%%%%%%%%%%%%%%%%%%%%%%%%%%%%%%
%%%%%%%%%% BIBLIOGRAPHY%%%%%%%%%%%%%%%%%%%%%%%%%%%%%%%%%%%%%%%%%%%%%%

\bibliographystyle{acm}
\bibliography{bibtex}

\end{document}